\documentclass{amsart}
\usepackage{amsfonts,amscd,amsthm,amsgen,amsmath,amssymb}
\usepackage[all]{xy}
\usepackage{epsfig,color}
\usepackage[vcentermath]{youngtab}
\newtheorem{theorem}{Theorem}[section]
\newtheorem{lemma}[theorem]{Lemma}
\newtheorem{corollary}[theorem]{Corollary}
\newtheorem{proposition}[theorem]{Proposition}

\theoremstyle{definition}
\newtheorem{definition}[theorem]{Definition}
\newtheorem{example}[theorem]{Example}

\theoremstyle{remark}
\newtheorem{remark}[theorem]{Remark}

\numberwithin{equation}{section}

\begin{document}

\setlength\parskip{0.5em plus 0.1em minus 0.2em}

\title[Exotic embedded surfaces and involutions]{Exotic embedded surfaces and involutions from Real Seiberg--Witten theory}
\author{David Baraglia}

\address{School of Mathematical Sciences, The University of Adelaide, Adelaide SA 5005, Australia}

\email{david.baraglia@adelaide.edu.au}


\date{\today}

\begin{abstract}
Using Real Seiberg--Witten theory, Miyazawa introduced an invariant of certain $4$-manifolds with involution and used this invariant to construct infinitely many exotic involutions on $\mathbb{CP}^2$ and infinitely many exotic smooth embeddings of $\mathbb{RP}^2$ in $S^4$. In this paper we extend Miyazawa's construction to a large class of $4$-manifolds, giving many infinite families of  involutions on $4$-manifolds which are conjugate by homeomorphisms but not by diffeomorphisms and many infinite families of exotic embeddings of non-orientable surfaces in $4$-manifolds, where exotic means continuously isotopic but not smoothly isotopic. Exoticness of our construction is detected using Real Seiberg--Witten theory. We study Miyazawa's invariant, relate it to the Real Seiberg--Witten invariants of Tian--Wang and prove various fundamental results concerning the Real Seiberg--Witten invariants such as: relation to positive scalar curvature, wall-crossing, a mod $2$ formula for spin structures, a localisation formula relating ordinary and Real Seiberg--Witten invariants, a connected sum formula and a fibre sum formula.
\end{abstract}

\maketitle




\section{Introduction}

Seiberg--Witten theory is one of the main to tools used to detect exotic phenomena on $4$-manifolds. The term {\em exotic} here refers to structures that are topologically equivalent, but not smoothly so. Examples of interest include exotic smooth structures, exotic embeddings of surfaces in $4$-manifolds and exotic group actions.

Because of the close relationship between Seiberg--Witten theory and complex geometry, it is natural to consider a real version of the Seiberg--Witten equations which interacts well with real algebraic geometry. This variant of Seiberg--Witten theory, which we will call {\em Real Seiberg--Witten theory} was introduced by Tian and Wang \cite{tw} and further studied by Nakamura in the case of Real structures without fixed points \cite{na,na2}. Real Seiberg--Witten theory has also been used to prove a Real version of the $10/8$-inequality \cite{kato} and Real Seiberg--Witten Floer theory has been studied in \cite{kmt,mi,bh}. Despite these developments, the Real Seiberg--Witten invariants themselves have received relatively little attention. In this paper we substantially develop the theory of Real Seiberg--Witten invariants, proving numerous fundamental results. As an application we give a very general construction of infinite families of exotic involutions and a construction of infinite families of exotic non-orientable embedded surfaces. While many constructions of exotic involutions exist in the literature \cite{ue,fss,sung,mi}, our construction is noteworthy in that it is applicable to a very large class of $4$-manifolds.

\subsection{Exotic involutions and exotic embedded surfaces}

Let $X$ be a smooth $4$-manifold and $S \subset X$ a smoothly embedded surface. An embedded surface $S' \subset X$ is said to be an {\em exotic copy} of $S$ if $S'$ is continuously isotopic to $S$ but not smoothly isotopic. More generally, we say that a collection of embedded surfaces $S_i \subset X$ are {\em exotic} if they are all continuously isotopic (through locally flat topological embeddings) but no two are smoothly isotopic. Exotic surfaces and related notions (for example, surfaces that are mapped onto each other by a homeomorphism but not by a diffeomorphism) have been constructed by various methods \cite{fs,fin,kim,kr,mark,bar3,mi}.

A {\em Real structure} on $X$ is a smooth, orientation preserving involution $\sigma$. A smooth involution $\sigma' : X \to X$ is said to be an {\em exotic copy} of $\sigma$ if there is a homeomorphism $f : X \to X$ such that $\sigma' = f^{-1} \circ \sigma \circ f$, but there is no such diffeomorphism.

In order to state our main result on exotic involutions and embedded surfaces, we introduce the notion of an admissible pair. Given a $4$-manifold $X$ and a Real structure $\sigma$, let $b_1(X)^{\pm \sigma}, b_+(X)^{\pm \sigma}$ denote the dimensions of the $\pm 1$ eigenspaces of $\sigma$ on $H^1(X ; \mathbb{R})$ and $H^+(X)$. 

\begin{definition}
Let $X$ be a compact, oriented, smooth $4$-manifold and $\sigma$ an orientation preserving smooth involution on $X$. We will say the pair $(X , \sigma)$ is {\em admissible} if $b_1(X)^{-\sigma} = 0$, $\sigma$ has a non-isolated fixed point and $(X,\sigma)$ satisfies one of the following conditions:
\begin{itemize}
\item[(1)]{$X$ admits a symplectic structure with $\sigma^*(\omega) = -\omega$, and $b_+(X) - b_1(X) = 3 \; ({\rm mod} \; 4)$.}
\item[(2)]{$X$ has a spin structure preserved by $\sigma$ and $b_1(X)^{-\sigma} = \sigma(X) = 0$.}
\item[(3)]{$X$ has a spin$^c$-structure $\mathfrak{s}$ with $\sigma^*(\mathfrak{s}) = -\mathfrak{s}$, $SW(X,\mathfrak{s})$ is odd, $b_+(X) - b_1(X) = 3 \; ({\rm mod} \; 4)$, and
\[
\frac{ c(\mathfrak{s})^2 - \sigma(X) }{8} = \frac{ b_+(X) - b_1(X) + 1}{2} = b_+(X)^{-\sigma}.
\]
}
\item[(4)]{$X = \overline{\mathbb{CP}}^2$ with an involution such that $b_+(X)^{-\sigma} = 0$.}
\item[(5)]{$X = N \# N$ and $\sigma$ is the involution which swaps the two factors (see Section \ref{sec:csf}), where $N$ is negative definite, $b_1(N) = 0$ and there is a spin$^c$-structure on $N$ with $c(\mathfrak{s})^2 = -b_2(N)$.}
\end{itemize}

\end{definition}

Our main result is:

\begin{theorem}\label{thm:exotic0}
Let $(X_1, \sigma_1), \dots , (X_k , \sigma_k)$ be admissible pairs. Let $c : \mathbb{CP}^2 \to \mathbb{CP}^2$ denote complex conjugation. Let $X = X_1 \# \cdots \# X_k$ and $\sigma = \sigma_1 \# \cdots \# \sigma_k$. Then 
\begin{itemize}
\item[(1)]{The involution $\sigma \# c$ on $X \# \overline{\mathbb{CP}}^2$ admits infinitely many exotic copies.}
\item[(2)]{Suppose the fixed point sets of $\sigma_1, \dots , \sigma_k$ are connected. Let $S \subset X/\sigma$ denote the image of the fixed point set of $\sigma$ in $X_0 = X/\sigma$. Then the embedding $P_+ \# S \subset X_0$ admits infinitely many exotic copies, where $P_+$ is the standard embedding of $\mathbb{RP}^2$ in $S^4$ with self-intersection $2$.}
\end{itemize}
\end{theorem}

For example, a $K3$-surface equipped with an odd, non-free involution is admissible. Theorem \ref{thm:exotic0} gives infinitely many exotic involutions on $\# n K3 \# m \overline{\mathbb{CP}}^2$ for any $n,m \ge 1$.

Taking $(X^4 , \sigma) = (S^4 , diag(-1,-1,1,1,1))$, Theorem \ref{thm:exotic0} gives infinitely many exotic involutions on $\overline{\mathbb{CP}}^2$ and infinitely many exotic embeddings $\mathbb{RP}^2 \subset S^4$. These exotic structures were originally constructed by Miyazawa \cite{mi} using Real Seiberg--Witten theory. Theorem \ref{thm:exotic0} is also proved with Real Seiberg--Witten theory and can be seen as a far-reaching extension of Miyazawa's results.

\subsection{Real Seiberg--Witten invariants}

Aside from the contruction of exotic involutions and embeddings, the aim of this paper is to develop the fundamental properties of the Real Seiberg--Witten invariants. Here we give a summary of the main results.

Let $X$ be a compact, oriented, smooth $4$-manifold and $\sigma$ a Real structure. Define the {\em Real Jacobian} of $X$ to be the torus 
\[
Jac_R(X) = \frac{H^1(X ; i\mathbb{R})^{-\sigma}}{H^1(X ; 2\pi i \mathbb{Z})^{-\sigma}}.
\]
Given a spin$^c$-structure $\mathfrak{s}$ we let $d = (c(\mathfrak{s})^2 - \sigma(X))/8$, the index of the spin$^c$ Dirac operator. If $\mathfrak{s}$ is a Real spin$^c$-structure, then the families index $D \to Jac(X)$ of the spin$^c$-Dirac operator corresponding to $\mathfrak{s}$ admits a Real structure. Let $D_R \to Jac_R(X)$ denote the real part of $D|_{Jac_R(X)}$. Then $D_R$ is a real virtual vector bundle of rank $d$.

Just as the ordinary Seiberg--Witten invariants depend on a spin$^c$-structure, the Real Seiberg--Witten invariants depend on a {\em Real spin$^c$-structure}. This is a spin$^c$-structure together with an antilinear lift of $\sigma$ to the spinor bundles of $\mathfrak{s}$ (see Section \ref{sec:rspinc} for details). The invariants come in two types: mod $2$ invariants and integer invariants.

\noindent {\bf Mod $2$ invariants.} The mod $2$ invariant is defined whenever $b_+(X)^{-\sigma} > 0$. If $b_+(X)^{-\sigma} = 1$, then the invariant depends on a chamber. The mod $2$ invariants are given by a collection of cohomology classes
\[
SW_{R,m}(X,\mathfrak{s}) \in H^*(Jac_R(X) ; \mathbb{Z}_2), \; m \ge 0
\]
of degree $m - (d - b_+(X)^{-\sigma})$. In the case $b_1(X)^{-\sigma} = 0$, we just have a single invariant $SW_R(X,\mathfrak{s}) \in \mathbb{Z}_2$ corresponding to $m = d-b_+(X)^{-\sigma}$.

\noindent {\bf Integer invariants.} To define integer invariants requires orienting the Real Seiberg--Witten moduli spaces. The easiest case is when $b_1(X)^{-\sigma} = 0$, for then the moduli spaces are guaranteed to be orientable provided $d$ is even. If $d - b_+(X)^{-\sigma} = 0$, $d$ is even and $b_+(X)^{-\sigma} > 0$, then the moduli space of the Real Seiberg--Witten equations is zero-dimensional and we obtain an integer invariant
\[
SW_{R,\mathbb{Z}}(X,\mathfrak{s}) \in \mathbb{Z}
\]
by taking a signed count of solutions. The orientation on the moduli space is only canonical up to an overall sign, so $SW_{R,\mathbb{Z}}(X,\mathfrak{s})$ is to be understood as an integer modulo an overall sign. Alternatively, the absolute value $|SW_{R,\mathbb{Z}}(X,\mathfrak{s})|$ is a well-defined non-negative integer.

When $b_1(X)^{-\sigma} > 0$ and $d$ is even, the integer invariant is defined provided $w_1(D_R) = 0$ and is a cohomology class
\[
\mathbf{SW}_{R,\mathbb{Z}}(X,\mathfrak{s}) \in H^{-(d-b_+(X)^{-\sigma})}( Jac_R(X) ; \mathbb{Z} ).
\]
If $d - b_+(X)^{-\sigma} + b_1(X)^{-\sigma} = 0$, we get an integer invariant by taking the pairing
\[
SW_{R,\mathbb{Z}}(X,\mathfrak{s}) = \langle \mathbf{SW}_{R,\mathbb{Z}}(X,\mathfrak{s}) , [Jac_R(X)] \rangle \in \mathbb{Z}.
\]
The integer and mod $2$ invariants are related by $\mathbf{SW}_{R,\mathbb{Z}}(X,\mathfrak{s}) = SW_{R,0}(X,\mathfrak{s}) \; ({\rm mod} \; 2)$.

When $b_1(X)^{-\sigma} = 0$ and $d - b_+(X)^{-\sigma} = 0$, there is another integer invariant
\[
deg_R(X , \mathfrak{s}) \in \mathbb{Z}
\]
called the {\em degree} of $(X , \mathfrak{s})$. This invariant was introduced and studied by Miyazawa in \cite{mi}. Just as with $SW_{R,\mathbb{Z}}$, the degree is only defined up to an overall sign, so the absolute value $|deg_R(X,\mathfrak{s})|$ is a well defined integer. In this paper we extend the definition of $deg_R$ to the case $b_1(X)^{-\sigma} > 0$ and $w_1(D_R)=0$. In general, the degree is a cohomology class
\[
deg_R(X,\mathfrak{s}) \in H^{-(d-b_+(X)^{-\sigma})}(Jac_R(X) ; \mathbb{Z})
\]
defined up to an overall sign. Although this invariant is defined for all values of $d$, it turns out to only be interesting when $d$ is even (see Proposition \ref{prop:deginteger} (1)).

Having introduced the mod $2$ and integer invariants, we now outline some of their fundamental properties proven in this paper.

\noindent {\bf 1. Positive scalar curvature.}
\begin{proposition}
Suppose that $\sigma$ preserves a metric of positive scalar curvature. Let $\mathfrak{s}$ be a Real spin$^c$-structure on $X$.
\begin{itemize}
\item[(1)]{If $b_+(X)^{-\sigma} > 1$, then the mod $2$ Real Seiberg--Witten invariants and the integral Real Seiberg--Witten invariants (when defined) all vanish.}
\item[(2)]{If $b_+(X)^{-\sigma} = 1$ then the mod $2$ Real Seiberg--Witten invariants and the integral Real Seiberg--Witten invariants (when defined) vanish for the chamber containing the zero perturbation. If the zero perturbation lies on the wall, then the Real Seiberg--Witten invariants vanish for both chambers}
\item[(3)]{If $b_+(X)^{-\sigma} = b_1(X)^{-\sigma} = 0$ and $\mathfrak{s}$ is a Real spin structure, then $|deg_R(X , \mathfrak{s})| = 1$.}
\end{itemize}

\end{proposition}

\noindent {\bf 2. Wall-crossing formula.}

\begin{theorem}
If $b_+(X)^{-\sigma}=1$, then
\[
SW^+_{R,m}(X,\mathfrak{s}) - SW^-_{R,m}(X, \mathfrak{s}) = w_{m-(d-1)}(-D_R).
\]
Furthermore, the integer-valued Real Seiberg--Witten invariants (when defined) are independent of the choice of chamber.
\end{theorem}

\noindent{\bf 3. A mod 2 identity.}

\begin{theorem}
The Real mod $2$ Seiberg--Witten invariants satisfy the following identities
\[
\sum_{l=0}^j \binom{ d-1 - m + j}{l} w_{j-l}(-D_R) SW^\phi_{R,m+l}(X,\mathfrak{s}) = 0
\]
for all $m \ge 0$, $j > 0$.
\end{theorem}

\noindent{\bf 4. Real spin structures.} By a {\em Real spin structure} we mean a spin structure $\mathfrak{s}$ such that $\sigma$ admits a lift to the corresponding principal $Spin(4)$-bundle which squares to $-1$. Any Real spin structure induces a Real spin$^c$-structure (Section \ref{sec:spin}).

\begin{theorem}
Let $\mathfrak{s}$ be a Real spin structure. If $b_+(X)^{-\sigma} > 2$, then $SW_{R,m}(X,\mathfrak{s}) = 0$ for all even $m$. If $b_+(X)^{-\sigma} = 2$, then $SW_{R,m}(X,\mathfrak{s}) = 0$ for all even positive $m$ and $SW_{R,0}(X,\mathfrak{s}) = w_{2 + \sigma(X)/8}(-D_R)$.
\end{theorem}

\noindent{\bf 5. Relation to the ordinary Seiberg--Witten invariant.} First we state the result in the case $b_1(X) = 0$.

\begin{theorem}
Let $X$ be a compact, oriented, smooth $4$-manifold with $b_1(X) = 0$. Let $\sigma$ be a Real structure on $X$ and $\mathfrak{s}$ a Real spin$^c$-structure. Assume that $b_+(X)^{-\sigma}>0$ and let $\phi$ be a chamber. If $2d - b_+(X) - 1 \ge 0$, then
\[
SW^\phi(X,\mathfrak{s}) = \binom{ 2b_+(X)^{\sigma} }{ b_+(X)-1 } SW^\phi_R(X,\mathfrak{s}) \; ({\rm mod} \; 2).
\]
\end{theorem}

We also give a more general result which holds for $b_1(X)^{-\sigma} = 0$. Let $\mathfrak{s}$ be a spin$^c$-structure such that $\sigma^*(\mathfrak{s}) = -\mathfrak{s}$. Assume that $\sigma$ does not act freely. Then it can be shown that Real structures on $\mathfrak{s}$ correspond to points of order $2$ in $Jac(X)$ (Proposition \ref{prop:jacfix}). To state the theorem first note that if $\mathfrak{s}'$ is a point of order $2$, then the inclusion $\iota_{\mathfrak{s}'} : \{ \mathfrak{s}' \} \to Jac(X)$ is $\mathbb{Z}_2$-equivariant, where $\mathbb{Z}_2$ acts on $Jac(X)$ by inversion. This induces push-forward maps $(\iota_{\mathfrak{s}'})_* : H^*_{\mathbb{Z}_2}(pt ; \mathbb{Z}_2) \to H^*_{\mathbb{Z}_2}(Jac(X) ; \mathbb{Z}_2) \cong H^*(Jac(X) ; \mathbb{Z}_2)[u]$, where $deg(u) = 1$. Let $\Delta = 2d - b_+(X) +b_1(X) - 1$ and $\delta = d - b_+(X)^{-\sigma}$ denote the dimensions of the ordinary and Real Seiberg--Witten moduli spaces respectively. Then we have:

\begin{theorem}
Let $X$ be a compact, oriented, smooth $4$-manifold. Let $\sigma$ be a Real structure on $X$ with $b_1(X)^{-\sigma} = 0$ and $b_+(X)^{-\sigma} > 0$. Let $\mathfrak{s}$ be a spin$^c$-structure which admits a Real structure and let $\phi$ be a chamber. Then for each $m \ge 0$, the expression
\[
u^{2m-\Delta} \binom{m-d}{\delta-m} \sum_{\mathfrak{s}'} (\iota_{\mathfrak{s}'})_*(1) SW^\phi_R(X, \mathfrak{s}') \in H^*( Jac^{\mathfrak{s}}(X) ; \mathbb{Z}_2)[u,u^{-1}]
\]
contains no negative powers of $u$ and the $u^0$-term equals $SW^\phi_m(X,\mathfrak{s}) \; ({\rm mod} \; 2)$. The sum $\Sigma_{\mathfrak{s}'}$ is over all Real structures on $\mathfrak{s}$.
\end{theorem}

For instance, if both moduli spaces are zero-dimensional, so $\Delta = \delta = 0$, this reduces to an equality
\[
SW^\phi(X , \mathfrak{s}) = \sum_{\mathfrak{s}'} SW^\phi_R(X , \mathfrak{s}') \; ({\rm mod} \; 2).
\]

\noindent{\bf 6. Properties of the integer invariants.}

\begin{proposition}
Let $X$ be a compact, oriented, smooth $4$-manifold, $\sigma$ a Real structure and $\mathfrak{s}$ a Real spin$^c$-structure. Suppose $d$ is even.
\begin{itemize}
\item[(1)]{If $b_+(X)^{-\sigma}>0$, then $deg_R(X,\mathfrak{s}) = 2 \, \mathbf{SW}^\phi_{R,\mathbb{Z}}(X , \mathfrak{s})$ for any chamber $\phi$.}
\item[(2)]{If $d$ is even and $b_+(X)^{-\sigma} = 0$, then we have $d \le 0$, $w_j(-D_R) = 0$ for $j > -d$ and $deg_R(X,\mathfrak{s}) = w_{-d}(-D_R) \; ({\rm mod} \; 2)$. In particular, if $d=0$, then $w_1(D_R) = 0$ and $deg_R(X , \mathfrak{s}) = 1 \; ({\rm mod} \; 2)$.}
\end{itemize}
\end{proposition}

\noindent {\bf 7. Connected sums.} Given $4$-manifolds $X_1,X_2$ with Real structures $\sigma_1,\sigma_2$ with non-empty fixed point set, we can form an equivariant connected sum $(X_1 \# X_2 , \sigma_1 \# \sigma_2)$. Furthermore, if $\mathfrak{s}_1,\mathfrak{s}_2$ are Real spin$^c$-structures on $X_1,X_2$ then it is possible to form their connected sum $\mathfrak{s}_1 \# \mathfrak{s}_2$ as a Real spin$^c$-structure on $X_1 \# X_2$ (Section \ref{sec:csf}).

First we consider the behaviour of the mod $2$ invariants under connected sum.

\begin{theorem}
Assume that $b_+(X_1)^{-\sigma} > 0$ and let $\phi$ be a chamber for $(X_1 , \sigma_1)$. Then $\phi$ also defines a chamber for $(X_1 \# X_2 , \sigma )$.
\begin{itemize}
\item[(1)]{If $b_+(X_2)^{-\sigma} > 0$, then $SW^\phi_{R,m}(X_1 \# X_2 , \mathfrak{s}_1 \# \mathfrak{s}_2) = 0$ for all $m \ge 0$.}
\item[(2)]{If $b_+(X_2)^{-\sigma} = 0$, then
\[
SW^\phi_{R,m}(X_1 \# X_2 , \mathfrak{s}_1 \# \mathfrak{s}_2) = \sum_{k \ge 0} w_k( -D_R(X_2,\mathfrak{s}_2) ) SW^\phi_{R,m-d_2-k}(X_1 , \mathfrak{s}_1)
\]
for all $m \ge 0$. In particular, if $b_1(X_2)^{-\sigma} = 0$, then
\[
SW^\phi_{R,m}(X_1 \# X_2 , \mathfrak{s}_1 \# \mathfrak{s}_2) = SW^\phi_{R,m-d_2}(X_1 , \mathfrak{s}_1).
\]
}
\end{itemize}

\end{theorem}

Next we consider the behaviour of the integer invariants under connected sum. Assume that $w_1(D_R(X_i , \mathfrak{s}_i)) = 0$ for $i=1,2$.

\begin{theorem}
We have
\[
deg_R(X_1 \# X_2 , \mathfrak{s}_1 \# \mathfrak{s}_2) = deg_R(X_1 ,\mathfrak{s}_1) \smallsmile deg_R(X_2 , \mathfrak{s}_2)
\]
where the right hand side is understood as an external cup product
\[
H^i( Jac_R(X_1) ; \mathbb{Z}) \times H^j( Jac_R(X_2) ; \mathbb{Z}) \to H^{i+j}(Jac_R(X_1) \times Jac_R(X_2) ; \mathbb{Z} )
\]
followed by the isomorphism $Jac_R(X_1) \times Jac_R(X_2) \cong Jac_R(X_1 \# X_2)$.
\end{theorem}

Set $d_i = (c(\mathfrak{s}_i)^2 - \sigma(X_i))/8$ for $i=1,2$.

\begin{theorem}
Suppose that $b_+(X_1)^{-\sigma} > 0$.
\begin{itemize}
\item[(1)]{If $d_1,d_2$ are even, then
\[
\mathbf{SW}_{R,\mathbb{Z}}(X_1 \# X_2 , \mathfrak{s}_1 \# \mathfrak{s}_2) = \mathbf{SW}_{R,\mathbb{Z}}(X_1 , \mathfrak{s}_1) \smallsmile deg_R(X_2 , \mathfrak{s}_2).
\]
In particular, if $b_+(X_1)^{-\sigma}, b_+(X_2)^{-\sigma} > 0$, then
\[
\mathbf{SW}_{R,\mathbb{Z}}(X_1 \# X_2 , \mathfrak{s}_1 \# \mathfrak{s}_2) = 2 \, \mathbf{SW}_{R,\mathbb{Z}}(X_1 , \mathfrak{s}_1) \smallsmile \mathbf{SW}_{R,\mathbb{Z}}(X_2, \mathfrak{s}_2).
\]
}
\item[(2)]{If $d_1,d_2$ are odd, then
\[
\mathbf{SW}_{R,\mathbb{Z}}(X_1 \# X_2 , \mathfrak{s}_1 \# \mathfrak{s}_2) = 0.
\]
}
\end{itemize}

\end{theorem}

A noteworthy feature of the connected sum formula is that if $b_+(X_1)^{-\sigma}, b_+(X_2)^{-\sigma}$ are both positive and $\mathbf{SW}_{R,\mathbb{Z}}(X_1 , \mathfrak{s}_1)$, $\mathbf{SW}_{R,\mathbb{Z}}(X_2, \mathfrak{s}_2)$ are non-zero, then \linebreak $\mathbf{SW}_{R,\mathbb{Z}}(X_1 \# X_2 , \mathfrak{s}_1 \# \mathfrak{s}_2)$ is non-zero. This is completely unlike the behaviour of the ordinary Seiberg--Witten invariant which vanishes on connected sums with $b_+(X_1) , b_+(X_2)$ both positive.

\noindent {\bf 8. Generalised fibre sums.} Let $X_1,X_2$ be compact, oriented, smooth $4$-manifolds and let $\sigma_1,\sigma_2$ be Real structures on $X_1,X_2$. Suppose for $i=1,2$, that the fixed point set of $\sigma_i$ contains an embedded torus $T_i \subset X_i$ of self-intersection zero. We construct a $4$-manifold $X = X_1 \#_{\varphi} X_2$ by removing tubular neighbourhoods $\nu(T_1), \nu(T_2)$ of $T_1,T_2$ and identifying their boundaries by an orientation reversing diffeomorphism $\varphi : \partial(X_1 \setminus \nu(T_1)) \to \partial( X_2 \setminus \nu(T_2))$. Choosing the diffeomorphism $\varphi$ to respect the involutions, we obtain a Real structure $\sigma$ on $X = X_1 \#_\varphi X_2$. If $\mathfrak{s}_1, \mathfrak{s}_2$ are Real spin$^c$-structures on $X_1,X_2$, then $\varphi$ can also be chosen so that the spin$^c$-structures patch together to form a Real spin$^c$-structure $\mathfrak{s}_1 \#_\varphi \mathfrak{s}_2$ on $X$, which is well-defined up to isomorphism. In the case that $\varphi : \partial X'_1 \to \partial X'_2$ is induced by a diffeomorphism $T_1 \to T_2$, together with trivialisations $\nu T_1 \cong D^2 \times T_1$, $\nu T_2 \cong D^2 \times T_2$ of the normal bundles, then $X$ is called a {\em generalised fibre sum} \cite[Definition 7.1.11]{gs}. Set $d_i = (c(\mathfrak{s}_i)^2 - \sigma(X_i))/8$.

\begin{theorem}
Suppose $b_1(X_1)^{-\sigma_1} = b_1(X_2)^{-\sigma_2} = 0$ and $d_1, d_2$ are even. Then
\[
deg_R(X_1 \#_\varphi X_2 , \mathfrak{s}_1 \#_\varphi \mathfrak{s}_2) = deg_R(X_1 ,\mathfrak{s}_1) deg_R(X_2 , \mathfrak{s}_2).
\]
In particular, if $b_+(X_i)^{-\sigma_i} > 0$ for $i=1,2$, then
\[
SW_{R,\mathbb{Z}}( X_1 \#_\varphi X_2 , \mathfrak{s}_1 \#_\varphi \mathfrak{s}_2 ) = 2 \, SW_{R,\mathbb{Z}}(X_1 , \mathfrak{s}_1) SW_{R,\mathbb{Z}}(X_2 ,\mathfrak{s}_2).
\]
\end{theorem}

For example, the $K3 = E(2)$ surface admits a Real structure whose fixed point set consists of two tori, which are regular fibres of an elliptic fibration. If $\mathfrak{s}$ denotes the spin structure on $K3$, then one finds $| SW_{R,\mathbb{Z}}( K3 , \mathfrak{s}) | = 1$. Taking the $n$-fold fibre sum, we obtain an involution $\sigma$ on $E(2n)$ for which $| SW_{R,\mathbb{Z}}( E(2n) , \mathfrak{s}_{E(2n)} ) | = 2^{n-1}$, where $\mathfrak{s}_{E(2n)}$ denotes the spin structure on $E(2n)$. By way of comparison the ordinary Seiberg--Witten invariant is $|SW(E(2n) , \mathfrak{s}_{E(2n)})| = \binom{ 2n-2}{n-1}$.

\subsection{Structure of the paper}

The structure of the paper is as follows. In Section \ref{sec:rspinc} we recall the notion of Real spin$^c$-structures and some basic results concerning them. In Section \ref{sec:realswi} we introduce the mod $2$ Real Seiberg--Witten invariants, along the same lines as Tian--Wang. In Section \ref{sec:rbf}, we study Real versions of the Bauer--Furuta invariant. There are actually two types of Real Bauer--Furuta map that we consider. In \textsection \ref{sec:O(2)}, we consider the usual Bauer--Furuta map but keeping track of the additional symmetry given by the Real structure. This gives an $O(2)$-equivariant version of the usual Bauer--Furuta map. In \textsection \ref{sec:realbf}, we consider the construction of a Real Bauer--Furuta map obtained by finite-dimensional approximation of the Real Seiberg--Witten equations. The Real Bauer--Furuta map is $\mathbb{Z}_2$-equviariant and we show that the Real Seiberg--Witten invariants can be recovered from it. In Section \ref{sec:integral} we introduce the integer Real Seiberg--Witten invariant. We also consider (a generalisation of) Miyazawa's degree invariant and establish the relation between the two invariants. In Section \ref{sec:prop}, we prove some fundamental properties of the Real Seiberg--Witten invariants related to: positive scalar curvature (\textsection \ref{sec:psc}), wall-crossing (\textsection \ref{sec:wcf}) and a series of identities satisfied by the mod $2$ invariants (\textsection \ref{sec:ident}). In Section \ref{sec:spin} we prove a formula for the mod $2$ invariants in the case of Real spin structures. In Section \ref{sec:loc} we use localisation in equivariant cohomology to obtain a formula for the ordinary mod $2$ Seiberg--Witten invariants in terms of the Real mod $2$ Seiberg--Witten invariants. In Section \ref{sec:csf} we prove a connected sum formula for the Real Seiberg--Witten invariants of equivariant connected sums. In Section \ref{sec:gfs} we prove a formula for the intger Real Seiberg--Witten invariants of fibres sums. Lastly in Section \ref{sec:exotic}, we give an application of Real Seiberg--Witten theory to the construction of exotic involutions and exotic embedded surfaces.

\noindent{\bf Acknowledgments.} The author was financially supported by an Australian Research Council Future Fellowship, FT230100092.


\section{Real spin$^c$-structures}\label{sec:rspinc}

Let $X$ be a topological space. A {\em Real structure} on $X$ is a homeomorphism $\sigma : X \to X$ which is an involution. Let $V \to X$ be a complex vector bundle on $X$. A {\em Real structure} on $V$ (with respect to a Real structure $\sigma$ on $X$) is an anti-linear involutive lift $\widetilde{\sigma} : V \to V$ of $\sigma$ to $V$. Similarly a {\em Quaternionic structure} on $V$ (with respect to $\sigma$) is an anti-linear lift $\widetilde{\sigma} : V \to V$ such that $\widetilde{\sigma}^2 = -1$. Denote by $\mathbb{Z}_{-}$ the equivariant locally constant sheaf on $X$ which is the constant sheaf $\mathbb{Z}$, but where $\sigma$ acts as multiplication by $-1$. Real line bundles are classified by $H^2_{\mathbb{Z}_2}( X ; \mathbb{Z}_-)$. This follows by the equality of Borel cohomology and equivariant sheaf cohomology \cite{sti}.

Now assume that $X$ is a compact, oriented, smooth $4$-manifold. Let $c : Spin^c(4) \to Spin^c(4)$ be the involution which is trivial on $Spin(4)$ and is given by complex conjugation on $S^1 \subset Spin^c(4)$. Fix a $\sigma$-invariant Riemannian metric $g$ on $X$. Let $\mathfrak{s}$ be a spin$^c$-structure on $X$, which we regard as a principal $Spin^c(4)$-bundle $P \to X$ together with an isomorphism $P/S^1 \cong Fr(X)$ of principal $SO(4)$-bundles, where $Fr(X)$ is the oriented frame bundle of $X$ with respect to the metric $g$. A {\em Real structure} on $\mathfrak{s}$ is a lift $\widehat{\sigma} : P \to P$ of $\sigma$ to $P$ covering the derivative $\sigma_* : Fr(X) \to Fr(X)$ of $\sigma$ and satisfying $\widehat{\sigma}(ph) = \widehat{\sigma}(p) c(h)$ for all $p \in P$, $h \in Spin^c(4)$ and such $\widehat{\sigma}^2 = -1$. An equivalent and often more useful definition is as follows. Let $S_{\pm}$ denote the spinor bundles for $\mathfrak{s}$. Then a Real structure on $\mathfrak{s}$ consists of Real structures $\widetilde{\sigma} : S_{\pm} \to S_{\pm}$ on $S_{\pm}$ respecting the Hermitian structures and respecting Clifford multiplication in the evident sense. The two definitions are related as follows: if $\widehat{\sigma} : P \to P$ is a Real structure, then $S_{\pm} = P \times_{Spin^c(4)} V_{\pm}$, where $V_{\pm}$ are the spinor representations of $Spin^c(4)$. Recall that $Spin^c(4) \cong (SU(2) \times SU(2) \times U(1))/(-1,-1,-1)$. The spinor representations $V_{\pm}$ can be thought of as the quaternions $\mathbb{H}$ with the following $Spin^c(4)$-action:
\[
( q_+ , q_- , z ) v = q_{\pm} v z.
\]
Here we are identifying $SU(2)$ with unit quaternions. Define $j : V_{\pm} \to V_{\pm}$ to be right multiplication by the unit quaternion $j$. Then $j$ is anti-linear and squares to $-1$. Define $\widetilde{\sigma}( p , v ) = (\widehat{\sigma}(p) , jv )$. This is well-defined as a map $S_{\pm} \to S_{\pm}$ because $j(gv) = c(g)jv$ for all $v \in V_{\pm}$, $g \in Spin^c(4)$. Then $\widetilde{\sigma}^2(p,v) = (\widehat{\sigma}^2(p) , j^2 v ) = ( p(-1) , -v ) = (p,v)$, so $\widetilde{\sigma}^2 = 1$. Conversely, given $\widetilde{\sigma}$ there is a uniquely determined $\widehat{\sigma}$ such that $\widetilde{\sigma}(p,v) = (\widehat{\sigma}(p) , jv)$. 

If $(\mathfrak{s} , \widetilde{\mathfrak{\sigma}})$ is a Real spin$^c$-structure, then the determinant line bundle $L = det(S_+)$ of the spin$^c$-structure inherits a Real structure from that of $S_+$. To each spin$^c$-structure $\mathfrak{s}$ is an associated Chern class $c(\mathfrak{s}) = c_1(L) \in H^2(X ; \mathbb{Z})$, which is a characteristic (an integral lift of $w_2(X) \in H^2(X ; \mathbb{Z}_2)$). A Real structure on $\mathfrak{s}$ promotes $L$ to a Real line bundle and hence $c(\mathfrak{s})$ lifts to a class in $H^2_{\mathbb{Z}_2}(X ; \mathbb{Z}_-)$. In particular $\sigma^*(L) \cong L^*$, hence $\sigma^*(c(\mathfrak{s})) = -c(\mathfrak{s})$.

Although we have defined spin$^c$-structures with respect to a choice of Riemannian metric $g$, it is easily shown that the set of isomorphism classes of spin$^c$-structures with respect to any two metrics are canonically isomorphic. A similar statement holds for Real spin$^c$-structures. Let $\mathcal{S}^c(X)$ denote the set of isomorphism classes of spin$^c$-structures and $\mathcal{S}^c_R(X)$ the set of isomorphism classes of Real spin$^c$-structures. If $\mathcal{S}^c(X)$ is non-empty then it is a torsor for $H^2(X ; \mathbb{Z})$. If $\mathcal{S}^c_R(X)$ is non-empty then it is a torsor for $H^2_{\mathbb{Z}_2}(X ; \mathbb{Z}_-)$.

We will also need the notion of the Real Jacobian of $X$. Let $(L , \widetilde{\sigma})$ be a Real Hermitian line bundle ($L$ is a Hermitian line bundle and $\widetilde{\sigma}$ is a Real structure on $L$ preserving the Hermitian structure). Let $A$ be a Hermitian connection on $L$. We say that $A$ is a {\em Real connection} (with respect to $\widetilde{\sigma}$) if the covariant derivative $\nabla_A$ corresponding to $A$ commutes with $\widetilde{\sigma}$. If $A$ is a Hermitian connection on $L$ then we define $\sigma^*A$ by declaring $\nabla_{\sigma^*A} = \widetilde{\sigma} \circ \nabla_A \circ \widetilde{\sigma}$. Given a Hermitian connection $A$, we obtain a Real connection by taking $(A + \widetilde{\sigma}^*A)/2$. In particular, Real connections always exist. If $A$ is a Real connection, any other Hermitian connection is of the form $A' = A + ia$ where $a$ is a real $1$-form. Then $\widetilde{\sigma}^*(A') = A - i\sigma^*(a)$. Hence $A'$ is Real if and only if $\sigma^*(a) = -a$. Let $\Omega^k(X)^{-\sigma}$ denote the space of real $k$-forms $\omega$ on $X$ such that $\sigma^*(\omega) = -\omega$. Thus the space of Real Hermitian connections on $L$ is an affine space over $i \Omega^1(X)^{-\sigma}$.

If $L$ is a Hermitian line bundle then a unitary gauge transformation $L \to L$ is given by multiplication by an $S^1$-valued function. A gauge transformation $\varphi : X \to S^1$ is said to be {\em Real} if it commutes with $\widetilde{\sigma}$, equivalently $\sigma^*(\varphi) = \varphi^{-1}$. Let $\mathcal{G} = Map(X , S^1)$ denote the gauge group and $\mathcal{G}_R = \{ \varphi \in \mathcal{G} \; | \; \sigma^*(\varphi) = \varphi^{-1} \}$ the group of Real gauge transformations. For $\varphi \in \mathcal{G}$, let $[\varphi] \in H^1(X ; \mathbb{Z})$ denote the homotopy class of $\varphi$ (recall that $H^1(X ; \mathbb{Z})$ can be identified with homotopy classes of maps from $X$ to $S^1$). If $\varphi \in \mathcal{G}_R$, then $[\varphi] \in H^1(X ; \mathbb{Z})^{-\sigma} = \{ \theta \in H^1(X ; \mathbb{Z}) \; | \; \sigma^*(\theta) = -\theta \}$. Let $g$ be a $\sigma$-invariant Riemannian metric on $X$. We say $\varphi \in \mathcal{G}$ is {\em harmonic} (with respect to $g$) if $\varphi^{-1} d\varphi$ is a harmonic $1$-form with respect to $g$. Let $\mathcal{H} \subset \mathcal{G}$ denote the group of harmonic maps from $X$ to $S^1$ and $\mathcal{H}_R = \mathcal{H} \cap \mathcal{G}_R$.

\begin{proposition}\label{prop:Rgauge}
Let $X$ be a compact, oriented, smooth $4$-manifold and $\sigma$ a Real structure on $X$. Let $g$ be a $\sigma$-invariant Riemannian metric on $X$.
\begin{itemize}
\item[(1)]{The map $\Omega^0(X)^{-\sigma} \times \mathcal{H}_R \to \mathcal{G}_R$ given by $( if , \varphi ) \mapsto e^{2 \pi if} \varphi$ is an isomorphism of groups.}
\item[(2)]{The map $\varphi \mapsto [\varphi]$ sending a gauge transformation to its homotopy class defines a split exact sequence
\begin{equation}\label{equ:exhr}
1 \to \{ \pm 1 \} \to \mathcal{H}_R \to H^1(X ; \mathbb{Z})^{-\sigma} \to 1.
\end{equation}
}
\end{itemize}
\end{proposition}
\begin{proof}
(1) Let $g \in \mathcal{G}_R$, so $\sigma^*(g) = g^{-1}$. Since $\mathcal{G} \cong \Omega^0(X)_0 \times \mathcal{H}$ where $\Omega^0(X)_0$ denotes functions $f$ such that $\int_X f dvol_X = 0$, we can uniquely write $g$ as $g = e^{2\pi i f} \varphi$ where $f \in \Omega^0(X)_0$ and $\varphi \in \mathcal{H}$. Then $g = \sigma^*(g)^{-1} = e^{-2 \pi i \sigma^*(f)} \sigma^*(\varphi^{-1})$. Uniqueness of the decomposition implies that $\sigma^*(f) = -f$ and $\varphi \in \mathcal{H}_R$.

(2) The map sending $\varphi \in \mathcal{H}_R$ to its underlying homotopy class $[\varphi] \in H^1(X ; \mathbb{Z})^{-\sigma}$ defines a homomorphism $p : \mathcal{H}_R \to H^1(X ; \mathbb{Z})^{-\sigma}$. Recall the exact sequence 
\[
1 \to S^1 \to \mathcal{H} \to H^1(X ; \mathbb{Z}) \to 1.
\]
If $\varphi \in \mathcal{H}_R$ is in the kernel of $p$, then $\varphi \in S^1 \cap \mathcal{H}_R = \{ \pm 1\}$. If $\theta \in H^1(X ; \mathbb{Z})^{-\sigma}$, then there exists a $\varphi \in \mathcal{H}$ with $[\varphi] = \theta$. Since $\sigma^*(\varphi) \in \mathcal{H}$ maps to $\sigma^*(\theta) = -\theta$, it follows that $\sigma^*(\varphi) = c \varphi^{-1}$ for some $c \in S^1$. Choose $u \in S^1$ with $u^2 = c^{-1}$. Then $u \varphi \in \mathcal{H}_R$ and $[u \varphi] = [\varphi] = \theta$. So the map $\mathcal{H}_R \to H^1(X;\mathbb{Z})^{-\sigma}$ is surjective with kernel $\{\pm 1\}$. The sequence (\ref{equ:exhr}) splits because $H^1(X ; \mathbb{Z})^{-\sigma}$ is a free group.
\end{proof}

Define the {\em Jacobian} $Jac(X)$ of $X$ to be the group of isomorphism classes of flat Hermitian line bundles $(L,A)$ whose underlying line bundle $L$ is trivial. Thus $Jac(X) \cong H^1(X ; i\mathbb{R})/H^1(X ; 2\pi i \mathbb{Z})$. Similarly, we define the {\em Real Jacobian} $Jac_R(X)$ to be the group of isomorphism classes of flat Real Hermitian line bundles $(L , \widetilde{\sigma} , A)$ whose underlying Real line bundle $(L , \widetilde{\sigma})$ is trivial. Choosing a Real trivialisation $L \cong \mathbb{C}$, we can write $\nabla_A = d + \alpha$ where $\alpha \in i\Omega^1(X)^{-\sigma}$ and $d\alpha = 0$. By Proposition \ref{prop:Rgauge}, any Real gauge transformation can be written as $h = e^{f} \varphi$, $f \in i \Omega^0(X)^{-\sigma}$, $\varphi \in \mathcal{H}_R$. Then $h^{-1}dh = df + \varphi^{-1}d\varphi$. Hence
\[
Jac_R(X) \cong \frac{ ker( d : i\Omega^1(X)^{-\sigma} \to i\Omega^2(X)^{-\sigma} ) }{ im( d : i\Omega^0(X)^{-\sigma} \to i\Omega^1(X)^{-\sigma} )  + 2 \pi i \mathcal{H}^1_g(X ; \mathbb{Z})^{-\sigma} } \cong \frac{ H^1(X ; i \mathbb{R})^{-\sigma} }{ H^1(X ; 2\pi i \mathbb{Z})^{-\sigma} },
\]
where $\mathcal{H}^1_g(X ; \mathbb{Z})^{-\sigma}$ denotes the space of $g$-harmonic $1$-forms on $X$ which are $\sigma$-anti-invariant and have integral periods.

Consider the forgetful map $Jac_R(X) \to Jac(X)$. We are interested in the kernel and image of this map. Consider the group $\mathbb{Z}_2 = \langle \sigma \rangle$. If $M$ is a $\mathbb{Z}_2$-module, write $M_-$ for $M = M \otimes_{\mathbb{Z}} \mathbb{Z}_-$. Start with the following short exact sequence of $\mathbb{Z}_2$-modules
\[
0 \to H^1(X ; 2 \pi i \mathbb{Z} )_- \to H^1(X ; i \mathbb{R})_- \to Jac(X)_- \to 0.
\]
The associated long exact sequence in group cohomology gives
\[
0 \to H^1(X ; 2\pi i \mathbb{Z})^{-\sigma} \to H^1(X ; i \mathbb{R})^{-\sigma} \to Jac(X)^{-\sigma} \to H^1(\mathbb{Z}_2 ; H^1(X ; 2\pi i \mathbb{Z})_- ) \to 0
\]
where $Jac(X)^{-\sigma} = \{ L \in Jac(X) \; | \; \sigma^*(L) = L^{-1} \}$. Since $H^1(X ; 2 \pi i \mathbb{Z})$ is a free abelian group of finite rank, the action of $\sigma^*$ on $H^1(X ; 2 \pi i \mathbb{Z})$ can be (non-canonically) decomposed into a direct sum of three types of $\mathbb{Z}_2$-modules: (1) trivial, $M = \mathbb{Z}$, where $\sigma$ acts as $+1$, (2) cyclotomic $M = \mathbb{Z}$, where $\sigma$ acts as $-1$ and (3) regular, the underlying module of the group ring $\mathbb{Z}[\mathbb{Z}_2]$. Then it is easily seen that $H^1(\mathbb{Z}_2 ; H^1( X ; 2 \pi i \mathbb{Z})_-) \cong \mathbb{Z}_2^t$, where $t$ is the number of trivial summands in $H^1(X ; 2 \pi i \mathbb{Z})$. We therefore have a short exact sequence
\[
1 \to Jac(X) \to Jac_R(X)^{-\sigma} \to \mathbb{Z}_2^t \to 1.
\]
In particular, $Jac_R(X)$ is the identity component of $Jac(X)^{-\sigma}$.

\begin{proposition}\label{prop:realspinc}
Suppose $H^1(X ; \mathbb{Z})^{-\sigma} = 0$ and $\sigma$ as a non-isolated fixed point. Then the image of the forgetful map $\mathcal{S}^c_R(X) \to \mathcal{S}^c(X)$ is the set of spin$^c$-structures for which $\sigma^*(\mathfrak{s}) = -\mathfrak{s}$.
\end{proposition}
\begin{proof}
If a spin$^c$-structure $\mathfrak{s}$ admits a Real structure, then $\sigma^*(\mathfrak{s}) = -\mathfrak{s}$, hence the image of the forgetful map is contained in the set of spin$^c$-structures with this property. Now suppose that $\mathfrak{s}$ is a spin$^c$-structure for which $\sigma^*(\mathfrak{s}) = -\mathfrak{s}$. Then there exists an antilinear lift $\sigma'$ of $\sigma$ to the spinor bundles $S^\pm$ of $\mathfrak{s}$. Then $(\sigma')^2$ is a lift of the identity, hence $(\sigma')^2 = h$ for some $h : X \to S^1$. Then
\[
h \circ (\sigma') = (\sigma')^3 = \sigma' \circ (\sigma)^2 = \sigma' \circ h = \sigma^*(h)^{-1} \circ \sigma'.
\]
So $\sigma^*(h) = h^{-1}$, that is, $h \in \mathcal{G}_R$. Since $H^1(X ; \mathbb{Z})^{-\sigma} = 0$, Proposition \ref{prop:Rgauge} implies that $h = \pm e^{if}$ for some $f \in \Omega^0(X)^{-\sigma}$. By assumption $\sigma$ has a non-isolated fixed point $x \in X$. Then $\sigma'_x : S^+_x \to S^+_x$ is an anti-linear map of $S^+_x$ to itself which respects the Hermitian structure and Clifford multiplication. Since there is no local obstruction to lifting a spin$^c$-structure to a spin-structure, there is a spin structure defined in a $\sigma$-invariant neighbourhood $U$ of $x$. Since $x$ is a non-isolated fixed point, $\sigma|_U$ is an odd involution, meaning that $\sigma|_U$ has a linear lift to $S^+|_U$ which squares to $-1$. Coupling to charge conjugation, $\sigma|_U$ has an anti-linear involutive lift $\sigma''$. Then we must have $\sigma'_x = t \sigma''_x$ for some $t \in S^1$. It follows that $h(x) = (\sigma'_x)^2 = t \overline{t} = 1$. Therefore, $h = e^{if}$. Now define $\widetilde{\sigma} = e^{-if/2} \sigma'$. Then $\widetilde{\sigma}$ is an anti-linear involutive lift of $\sigma$, hence defines a Real structure on $\mathfrak{s}$.
\end{proof}

\begin{remark}
An argument similar to the one used in the proof of Proposition \ref{prop:realspinc} shows that if a Real spin$^c$-structure exists, then the fixed point set of $\sigma$ contains no isolated points.
\end{remark}

\section{Real Seiberg--Witten invariants}\label{sec:realswi}

In this section we introduce the mod $2$ Real Seiberg--Witten invariants, following Tian--Wang \cite{tw}. Integer valued invariants will be considered in Section \ref{sec:integral}.

Let $X$ be a compact, oriented, smooth $4$-manifold. Let $\sigma$ be a Real structure on $X$ and let $g$ be a $\sigma$-invariant Riemannian metric. Let $(\mathfrak{s} , \widetilde{\sigma})$ be a Real spin$^c$-structure. Let $\eta \in i \Omega^+(X)^{-\sigma}$ be an imaginary self-dual $2$-form such that $\sigma^*(\eta) = -\eta$. The {\em Real Seiberg--Witten equations} with respect to $(X , \sigma , g , \eta , \mathfrak{s} , \widetilde{\sigma})$ are the usual Seiberg--Witten equations
\begin{align*}
D_A \psi &= 0, \\
F^+_A + \eta &= \sigma(\psi),
\end{align*}
but where $(A , \psi)$ is a Real configuration, that is, $A$ is a Real spin$^c$-connection (equivalently $2A$ is a Real connection on the determinant line) and $\psi$ is a Real positive spinor ($\widetilde{\sigma}^*(\psi) = \psi$). The Real gauge group $\mathcal{G}_R$ acts on solutions in the usual manner: $h(A , \psi ) = ( A - h^{-1}dh , h \psi)$. We say that a configuration $(A , \psi)$ is {\em reducible} if $\psi = 0$ and {\rm irreducible} otherwise. The stabiliser group of $\mathcal{G}_R$ at $(A , \psi)$ is $\{ \pm 1\}$ for a reducible configuration and is trivial for an irreducible configuration.

Let $b_+(X)^{-\sigma}$ denote the dimension of $H^+(X)^{-\sigma}$, the $-1$-eigenspace of $\sigma$ acting on $H^+(X)$. Then the space of perturbations $\eta$ for which a reducible solution exists form a codimension $b_+(X)^{-\sigma}$ subspace of $i \Omega^+(X)^{-\sigma}$. In particular, if $b_+(X)^{-\sigma} > 0$ then we can always choose $\eta$ so that no reducible solutions exist.

Let $\mathcal{C}_R$ denote the space of Real configurations $\{ (A,\psi) \}$. Let $\mathcal{C}^*_R$ denote the subspace of Real irreducible configurations. Let $\mathcal{B}_R = \mathcal{C}_R/\mathcal{G}_R$, $\mathcal{B}^*_R = \mathcal{C}_R/\mathcal{G}_R$. It is easily seen that $\mathcal{C}_R$ and $\mathcal{C}^*_R$ are contractible. Hence $\mathcal{B}^*_R$ has the homotopy type of $B\mathcal{G}_R \cong B( \mathbb{Z}_2 \times H^1(X ; \mathbb{Z})^{-\sigma} ) \cong \mathbb{RP}^\infty \times Jac_R(X)$.

Let $\mathcal{N}_R \subset \mathcal{C}_R$ denote the space of solutions to the Real Seiberg--Witten equations, $\mathcal{N}^*_R$ the space of irreducible solutions, $\mathcal{M}_R = \mathcal{N}_R/\mathcal{G}_R \subset \mathcal{BC}$ the moduli space of gauge equivalence classes of solutions to the Real Seiberg--Witten equations and $\mathcal{M}^*_R = \mathcal{N}^*_R/\mathcal{G}_R \subset \mathcal{BC}^*$ the moduli space of irreducible solutions.

The usual compactness argument for Seiberg--Witten moduli spaces shows that $\mathcal{M}_R$ is compact. If $b_+(X)^{-\sigma} > 0$, then for generic $\eta \in i \Omega^+(X)^{-\sigma}$, there are no reducible solutions and $\mathcal{N}_R \subset \mathcal{C}_R$ is cut out transversally by the Real Seiberg--Witten equations \cite[\textsection 3]{tw}. Then $\mathcal{M}_R = \mathcal{M}_R^*$ is a compact, smooth manifold. The dimension of $\mathcal{M}_R$ is given by 
\[
d - b_+(X)^{-\sigma} + b_1(X)^{-\sigma},
\]
where $d = (c(\mathfrak{s})^2 - \sigma(X))/8$ is the index of the Dirac operator associated to $\mathfrak{s}$ and $b_1(X)^{-\sigma}$ is the dimension of the $-1$-eigenspace of $\sigma$ acting on $H^1(X ; \mathbb{R})$.

Real mod $2$ Seiberg--Witten invariants are obtained by evaluating the mod $2$ fundamental class of $\mathcal{M}_R$ against cohomology classes on $\mathcal{B}_R$. For the mod $2$ invariants we do not have to consider the orientability of $\mathcal{M}_R$. We have a mod $2$ homology class $[\mathcal{M}_R] \in H_{\delta}( \mathcal{B}_R ; \mathbb{Z}_2)$, where $\delta = d - b_+(X)^{-\sigma} + b_1(X)^{-\sigma}$ is the dimension of the moduli space. Thus we can pair it with a class $\theta \in H^\delta( \mathcal{B}_R^* ; \mathbb{Z}_2)$ to obtain an element of $\mathbb{Z}_2$:
\[
\langle [\mathcal{M}_R] , \theta \rangle \in \mathbb{Z}_2.
\]
For this to be an invariant, we need to check that it is independent of the choice of metric $g$ and perturbation $\eta$. The usual cobordism argument in ordinary Seiberg--Witten theory extends to the Real case and shows that $\langle [\mathcal{M}_R] , \theta \rangle$ is independent of $(g, \eta)$ provided that $b_+(X)^{-\sigma} > 1$. In the case $b_+(X)^{-\sigma} = 1$, the value of $\langle [\mathcal{M}_R] , \theta \rangle$ depends on a choice of chamber. A chamber (with respect to the metric $g$) is a connected component of $i H^+(X)^{-\sigma} \setminus \{ w \}$, where $w = \pi i c(\mathfrak{s})^+$ and $c(\mathfrak{s})^+$ denotes the projection of $c(\mathfrak{s})$ to $H^+(X)$ determined by the metric $g$ (via the splitting $H^2(X ; \mathbb{R}) \cong H^+(X) \oplus H^-(X)$). We denote this invariant as
\[
\mathbf{SW}_R(X , \sigma , \mathfrak{s} , \widetilde{\sigma} ) : H^\delta( B\mathcal{G}_R ; \mathbb{Z}_2 ) \to \mathbb{Z}_2,
\]
or by $\mathbf{SW}_R^\phi(X , \sigma , \mathfrak{s} , \widetilde{\sigma} )$ in the case $b_+(X)^{-\sigma} = 1$, where $\phi$ is a chamber. In order to keep notation simple, we will often denote the invariant as $\mathbf{SW}_R(X , \mathfrak{s})$ when $\sigma$ and $\widetilde{\sigma}$ are understood.

Let $\mathcal{H}_R$ denote the group of Real harmonic gauge transformations. Recall that the inclusion $\mathcal{H}_R \to \mathcal{G}_R$ is a homotopy equivalence and that we have a short exact sequence $1 \to \mathbb{Z}_2 \to \mathcal{H}_R \to H^1(X ; \mathbb{Z})^{-\sigma} \to 0$. This gives rise to a fibration $B\mathbb{Z}_2 \to B\mathcal{H}_R \to B(H^1( X ; \mathbb{Z})^{-\sigma} )$. Furthermore we can identify $B(H^1(X ; \mathbb{Z})^{-\sigma} )$ with $Jac_R(X)$, the Real Jacobian. Choosing a splitting of the exact sequence for $\mathcal{H}_R$, we have a (non-canonical) homotopy equivalence $B\mathcal{H}_R \cong B\mathbb{Z}_2 \times Jac_R(X)$ and hence an isomorphism 
\[
H^*( B\mathcal{H}_R ; \mathbb{Z}_2) \cong H^*( B\mathbb{Z}_2 \times Jac_R(X) ; \mathbb{Z}_2 ) \cong H^*(Jac_R(X) ; \mathbb{Z}_2)[v],
\]
where $v$ is the generator of $H^1(B\mathbb{Z}_2 ; \mathbb{Z}_2)$.

To get numerical invariants, we need cohomology classes of $B\mathcal{G}_R$. One way to do this is as follows. Let $x \in X$ be any point of $X$.  Consider the evaluation map $ev_x : \mathcal{G}_R \to S^1$ given by $ev_x(h) = h(x)$. Let $U \in H^2( BS^1 ; \mathbb{Z}_2)$ be the generator. Pulling back by $ev_x$, we get a class $ev_x^*(U) \in H^2( B\mathcal{G}_R ; \mathbb{Z}_2)$. As the map $ev_x$ depends continuously on $x$, it follows that all of the evaluation maps are homotopic and hence $ev_x^*(U)$ is independent of the choice of $x$. Abusing notation, we denote this class by $U$. If the dimension $\delta$ of $\mathcal{M}_R$ is even and non-negative, so $\delta = 2m$ for some $m \ge 0$, then we can define the {\em pure Real Seiberg--Witten invariant} $SW_R(X , \mathfrak{s}) \in \mathbb{Z}_2$ by 
\[
SW_R(X , \mathfrak{s}) = \mathbf{SW}_R(X , \mathfrak{s})( U^m ).
\]

If $b_1(X)^{-\sigma} = 0$, then there is a uniquely determined class $v$ of degree $1$ such that $H^*( B\mathcal{G}_R ; \mathbb{Z}_2) \cong \mathbb{Z}_2[v]$. In this case we have $U = v^2$. Then there is only one non-trivial cohomology class that we can evaluate against $[\mathcal{M}_R]$, namely $v^\delta$, where $\delta$ is the dimension of $\mathcal{M}_R$ (assuming $\delta \ge 0$). So when $b_1(X)^{-\sigma} = 0$, we may define the {\em pure Real Seiberg--Witten invariant} $SW_R(X , \mathfrak{s}) \in \mathbb{Z}_2$ by 
\[
SW_R(X , \mathfrak{s}) = \mathbf{SW}_R(X , \mathfrak{s})( v^\delta).
\]
Thus we have a Real Seiberg--Witten invariant even when the dimension of the moduli space is odd. Note that when $\delta = 2m$ is even and non-negative this agrees with the previous definition since $U = v^2$.

Choose a splitting $s : H^1(X ; \mathbb{Z})^{-\sigma} \to \mathcal{H}_R$, inducing an isomorphism $H^*( B \mathcal{G}_R ; \mathbb{Z}_2) \cong H^*(Jac_R(X) ; \mathbb{Z}_2)[v]$, where $deg(v) = 1$. Using Poincar\'e duality on $Jac_R(X)$, the Seiberg--Witten invariant
\[
\mathbf{SW}_R(X,\mathfrak{s}) : H^*(Jac_R(X) ; \mathbb{Z}_2)[v] \to \mathbb{Z}_2
\]
can be re-written as a map
\[
\mathbf{SW}_R(X , \mathfrak{s})' : \mathbb{Z}_2[v] \to H^*(Jac_R(X) ; \mathbb{Z}_2),
\]
such that
\[
\mathbf{SW}_R(X,\mathfrak{s})( \alpha v^m ) = \langle [Jac_R(X)] , \alpha \mathbf{SW}_R(X,\mathfrak{s})'( v^m) \rangle
\]
for all $m \ge 0$ and all $\alpha \in H^*( Jac_R(X) ; \mathbb{Z}_2)$. As will be seen in Section \ref{sec:rbf}, the map $\mathbf{SW}_R(X,\mathfrak{s})'$ emerges naturally from the point of view of the Real Bauer--Furuta invariant. We will often switch between these two points and view and we will furthermore abuse notation and use $\mathbf{SW}_R(X,\mathfrak{s})$ to denote either map. However we caution that the equivalence of these two approaches depends on a choice of splitting $s : H^1(X ; \mathbb{Z})^{-\sigma} \to \mathcal{H}_R$.

We have seen that a splitting $s : H^1(X ; \mathbb{Z})^{-\sigma} \to \mathcal{H}_R$ determines a class $v \in H^1( B\mathcal{G}_R ; \mathbb{Z}_2)$ with the property that $v^2 = U$. Conversely, any such class arises from some choice of splitting. As we shall now explain, a fixed point of $\sigma$ determines such a class and also a splitting. Suppose that $x \in X$ is fixed by $\sigma$. If $h \in \mathcal{G}_R$, then $\sigma^*(h) = h^{-1}$ and hence $h(x) \in \mathbb{Z}_2 = \{ \pm 1\}$. So we have an evaluation map $ev_x : \mathcal{G}_R \to \mathbb{Z}_2$. Restricted to $\mathcal{H}_R$, $ev_x$ determines a splitting of the sequence $1 \to \mathbb{Z}_2 \to \mathcal{H}_R \to H^1(X ; \mathbb{Z})^{-\sigma}$. Let $v$ denote the generator of $H^1(B\mathbb{Z}_2 ; \mathbb{Z}_2)$. Then we obtain a class $V_x = ev_x^*(v) \in H^1( B\mathcal{G}_R ; \mathbb{Z}_2)$ corresponding to this splitting. In general, the class $V_x$ depends on the choice of fixed point $x$. However $ev_x$ depends continuously on $x$, so $V_x$ depends only on the connected component of $X^\sigma$ to which it belongs. 

Suppose $x_1,x_2 \in X^\sigma$ are fixed points. Choose a path $\gamma_{12}$ from $x_1$ to $x_2$. Then $l_{12} = \gamma_{12} \cup -\sigma(\gamma_{12})$ is a loop in $X$ and defines a homology class $[l_{12}] \in H_1(X ; \mathbb{Z})^{-\sigma}$. Since $Jac_R(X) = H^1(X ; i\mathbb{R})^{-\sigma}/H^1(X ; 2\pi i\mathbb{Z})^{-\sigma}$, we have that $H^1(Jac_R(X) ; \mathbb{Z}_2) \cong H_1(X ; \mathbb{Z}_2)^{-\sigma}$ and hence $[l_{12}]$ defines a class in $H^1( Jac_R(X) ; \mathbb{Z}_2)$. 

\begin{proposition}
We have $V_{x_1} - V_{x_2} = [l_{12}]$.
\end{proposition}
\begin{proof}
Let $s_1,s_2 : H^1(X ; \mathbb{Z})^{-\sigma} \to \mathcal{H}_R$ be the splittings corresponding to $x_1,x_2$, namely
\[
(s_i(\theta))(x) = exp \left( 2 \pi i \int_{x_i}^x \theta \right).
\]
Using $s_1$ to identify $\mathcal{H}_R$ with $\mathbb{Z}_2 \times H^1(X ; \mathbb{Z})^{-\sigma}$, we then have $s_2( \theta ) = ( \lambda(\theta) , \theta )$ for some homomorphism $\lambda : H^1( X ; \mathbb{Z})^{-\sigma} \to \mathbb{Z}_2$. Then $\lambda$ is a class in $Hom( H^1(X ; \mathbb{Z})^{-\sigma} , \mathbb{Z}_2) \cong H^1( Jac_R(X) ; \mathbb{Z}_2)$ and we clearly have that $V_{x_1} - V_{x_2} = \lambda$. Hence it remains to prove that $\lambda(\theta) = exp( \pi i \int_{l_{12}} \theta )$ for all $\theta \in H^1(X ; \mathbb{Z})^{-\sigma}$.

By the definition of $\lambda$, we have that $s_2(\theta)(x) = \lambda(\theta) s_1(\theta)(x)$ for all $x \in X$. Taking $x = x_1$, we get
\[
\lambda(\theta) = exp \left( 2 \pi i \int_{x_2}^{x_1} \theta \right) = exp \left( -2\pi i \int_{\gamma_{12}} \theta \right).
\]
But since $\sigma^*(\theta) = -\theta$, we have that $2\pi i \int_{\gamma_{12}} \theta = \pi i \int_{ l_{12} } \theta$, and the result follows.
\end{proof}

\section{Real Bauer--Furuta invariants}\label{sec:rbf}

In order to prove various results concerning the Real Seiberg--Witten invariant it will be helpful to make use of Bauer--Furuta invariants. There are two possible approaches and both will be useful. One is to construct the usual Bauer--Furuta invariant, but equivariantly with respect to $\sigma$. Then we obtain the Real Bauer--Furuta invariant by restricting to the $\sigma$-invariant part. Alternatively we can imitate the usual construction of the Bauer--Furuta invariant, but with Real configurations throughout. The first approach is carried out in Section \ref{sec:O(2)} and the second approach in Section \ref{sec:realbf}.

\subsection{$O(2)$-equivariant Bauer--Furuta map}\label{sec:O(2)}

In this section we will carry out the construction of the Bauer--Furuta invariant of $X$ following \cite{bf}, but keeping track of the Real structure. This will yield an $O(2)$-equivariant stable homotopy class, refining the usual Bauer--Furuta invariant which is $S^1$-equivariant. This will be needed when we consider localisation in Section \ref{sec:loc}.

Let $X$ be a compact, oriented, smooth $4$-manifold. Let $\sigma$ be a Real structure on $X$ and $g$ a $\sigma$-invariant Riemannian metric. Let $(\mathfrak{s} , \widetilde{\sigma})$ be a Real spin$^c$-structure. Let $F$ be an imaginary $2$-form such that $\sigma^*(F) = -F$ and such that $(i/\pi)F$ represents the image of $c(\mathfrak{s})$ in $H^2(X ; \mathbb{R})$. Let $Conn$ denote the space of spin$^c$-connections for $\mathfrak{s}$ with curvature equal to $F$. The gauge group $\mathcal{G}$ acts on $Conn$. Denote the quotient space by $Jac^{\mathfrak{s}}(X)$. It is a torsor over $Jac(X)$. Let $\mathcal{A}, \mathcal{C}$ be the following trivial Hilbert bundles over $Conn$:
\begin{align*}
\mathcal{A} &= Conn \times L^2_k(S^+) \oplus L^2_k( i\wedge^1 T^*X) \oplus i\mathbb{R}, \\
\mathcal{C} &= Conn \times L^2_{k-1}(S^-) \times L^2_{k-1}( i \wedge^+ T^*X) \oplus L^2_{k-1}( i\wedge^0 T^*X) \oplus H^1(X ; i\mathbb{R})
\end{align*}
where $L^2_k(E)$ denotes sections of $E$ of Sobolev class $L^2_k$ and where $k > 4$.

Let $\mu : \mathcal{A} \to \mathcal{C}$ be the Seiberg--Witten equations
\[
\mu( A' , \psi , a , f ) = ( A' , D_{A'+a}(\psi) , F^+_{A'+a} - \sigma(\psi) , d^*(a) + f , pr(a) )
\]
where $pr$ denotes a projection map from $1$-forms to $H^1(X ; \mathbb{R})$ (that is, $pr$ has the property that $pr(a) = [a]$ whenever $a$ is a harmonic $1$-form). The group of $L^2_{k+1}$ gauge transforms acts on $\mathcal{A}$ and $\mathcal{C}$ by
\begin{align*}
h(A' , \psi , a , f) &= (A' - h^{-1}dh , h \psi , a , f ), \\
h(A' , \varphi , a_2 , a_0 , \omega) &= (A' - h^{-1}dh , h \varphi , a_2 , a_0 , \omega )
\end{align*}
and $\mu$ is equivariant with respect to these actions. Notice also that $\sigma$ acts on $\mathcal{A}$ and $\mathcal{C}$ and $\mu$ is $\sigma$-equviariant. The action of $\mu$ on gauge transformations is given by $h \mapsto \sigma^*(h)^{-1}$.

Let $A$ be a reference connection, which we assume to be a Real connection. Any other connection in $Conn$ is of the form $A' = A + b$, where $b$ is a closed $1$-form. Then one can gauge transform so that $b$ is harmonic, or equivalently, $A'$ is in Coulomb gauge with respect to $A$. The gauge transform is unique up to a harmonic gauge transformation. Effectively what this means is that we will restrict $\mathcal{A}, \mathcal{C}$ to Hilbert bundles $\mathcal{A}_H = \mathcal{A}|_{Conn_H}$, $\mathcal{C}_H = \mathcal{C}|_{Conn_H}$ over $Conn_H$, the set of spin$^c$-connections of the form $A' = A + \alpha$, where $\alpha$ is harmonic (so $Conn_H$ is a torsor form $H^1(X ; i\mathbb{R})$). We also restrict the gauge group to the group $\mathcal{H}$ of harmonic gauge transformations. The restricted monopole map $\mu : \mathcal{A}_H \to \mathcal{C}_H$ is an $\mathcal{H}$-equivariant map of (trivial) Hilbert bundles over $Conn_H$.

In addition to the action of $\mathcal{H}$ we also have the involution $\sigma$ and thus $\mu$ is equivariant with respect to the semidirect product $\mathcal{H}' = \mathbb{Z}_2 \ltimes \mathcal{H}$. We have an exact sequence
\[
0 \to O(2) \to \mathcal{H}' \to H^1(X ; \mathbb{Z}) \to 0.
\]
The subgroup $O(2)$ is the semidirect product $O(2) = \mathbb{Z}_2 \ltimes S^1$ of $\mathbb{Z}_2 = \langle \sigma \rangle$ with the group $S^1$ of constant gauge transformations.

Suppose that $s : H^1(X ; \mathbb{Z}) \to \mathcal{H}'$ is a splitting. Since the action of $s(H^1(X ; \mathbb{Z}))$ on $Conn_H$ is free, with quotient $Jac^{\mathfrak{s}}(X)$, we get that $\mu$ descends to a map of Hilbert bundles over $Jac^{\mathfrak{s}}(X)$. If, in addition, $s(H^1(X ; \mathbb{Z}))$ is a normal subgroup, then $\mu$ is equivariant with respect to the quotient group, $O(2)$. This is the Fredholm map whose finite-dimensional approximation is the Bauer--Furuta map.

Choose a splitting $s : H^1(X ; \mathbb{Z}) \to \mathcal{H}$. This induces a splitting $H^1(X ; \mathbb{Z}) \to \mathcal{H}'$ by the inclusion $\mathcal{H} \to \mathcal{H}'$. The image of this splitting is a normal subgroup if and only if $s$ is $\sigma$-equivariant (where $\sigma$ acts on $H^1(X ; \mathbb{Z})$ as minus pullback). Equivalently, $s( \sigma^*(\theta) )  = \sigma^*( s(\theta))$. We will refer to such a splitting as an {\em equivariant splitting}. 

\begin{lemma}\label{lem:eqs}
If $\sigma$ has a fixed point or if $H^1(X ; \mathbb{Z})^{\sigma} = 0$, then an equivariant splitting exists.
\end{lemma}
\begin{proof}
Suppose a fixed point $x_0 \in X$ exists. Then the splitting $s(\theta)(x) = exp( 2\pi i \int_{x_0}^x \theta)$ is equivariant.

Next, consider the following short exact sequence of $\mathbb{Z}_2$-modules
\[
1 \to S^1 \to \mathcal{H} \to H^1(X ; \mathbb{Z} ) \to 0.
\]
An equivariant splitting exists if and only if this is a trivial extension. This is an extension is classified by a class in $H^1( \mathbb{Z}_2 ; Hom( H^1(X ; \mathbb{Z}) , S^1) )$. One finds that this group is isomorphic to $\mathbb{Z}_2^t$, where $t$ is the number of trivial summands in the action of $\sigma$ on $H^1(X ; \mathbb{Z})$. In particular, if $H^1( X ; \mathbb{Z})^{\sigma} = 0$ then this cohomology group vanishes and an equivariant splitting exists.
\end{proof}

Supposing an equivariant splitting $s : H^1(X ; \mathbb{Z}) \to \mathcal{H}$ exists. Then taking the quotient of $\mu : \mathcal{A}_H \to \mathcal{C}_H$ by $s(H^1(X ; \mathbb{Z}))$, we obtain an $O(2)$-equivariant map of Hilbert bundles over $Jac^{\mathfrak{s}}(X)$,
\[
\mu_s : \mathcal{A}_{H,s} \to \mathcal{C}_{H,s}
\]
where $\mathcal{A}_{H,s}, \mathcal{C}_{H,s}$ denote the Hilbert bundles over $Jac^{\mathfrak{s}}(X)$ obtained by quotienting $\mathcal{A}_H, \mathcal{C}_H$ by $s(H^1(X ; \mathbb{Z}))$. Carrying out finite-dimensional approximation in the sense of \cite{bf} to this map, one obtains an $O(2)$-equivariant map
\[
f : S^{V,U} \to S^{V',U'}
\]
of finite-dimensional sphere bundles over $B = Jac^{\mathfrak{s}}(X)$. Here $V,V'$ are complex vector bundles over $B$, $U,U'$ are real vector bundles over $B$, $S^{V,U}$ is the unit sphere bundle of $V \oplus U \oplus \mathbb{R}$ and $S^{V',U'}$ is defined similarly. The subgroup $S^1 \subset O(2)$ acts by scalar multiplication on $V,V'$ and trivially on $U,U'$. To describe the action of $\sigma \in O(2)$, we note that our chosen reference connection $A$ yields an identification $Jac^{\mathfrak{s}}(X) \cong Jac(X)$, so we will regard the base as $Jac(X)$. Then $\sigma$ acts on $Jac(X)$ in the usual way: $L \mapsto \sigma^*(L)^*$. This action lifts to $V,V',U,U'$ and makes $U,U'$ into equivariant vector bundles and makes $V,V'$ into Real vector bundles over $Jac(X)$.

Restrict $f$ to $Jac_R(X) \subseteq Jac(X)$ and then further restruct to the fixed points of $\sigma$. This gives a map
\[
f_R : S^{V_R , U_R} \to S^{V'_R , U'_R}
\]
of sphere bundles over $Jac_R(X)$, where $V_R = V^\sigma |_{Jac_R(X)}$ and $V'_R,U_R,U'_R$ are defined similarly. This map is $\mathbb{Z}_2$-equivariant, where $\mathbb{Z}_2$ is to be thought of as the subgroup $\{ \pm 1 \} \subset S^1 \subset O(2)$. In this way, we have constructed a $\mathbb{Z}_2$-equivariant stable homotopy class of map of sphere bundles over $Jac_R(X)$. We call it the {\em Real Bauer--Furuta invariant of} $X,\mathfrak{s}$. In the following section, we will give a different construction of the Real Bauer--Furuta invariant which has the advantage that it does not require the existence of an equivariant splitting to construct.

A Real structure $\widetilde{\sigma}$ on $\mathfrak{s}$ defines an involution $\sigma$ on $Jac^{\mathfrak{s}}(X)$ by $A \mapsto \widetilde{\sigma}^*(A)$. The fixed point set $Jac^{\mathfrak{s}}(X)^{\sigma}$ is a union of tori, each of which is a translation of $Jac_R^{\mathfrak{s}}(X)$. As we will now demonstrate, the components of the fixed point set are in bijection with Real structures on $\mathfrak{s}$. Observe that for any Real structure $\mathfrak{s}'$ on $\mathfrak{s}$, we have a natural inclusion $Jac^{\mathfrak{s}'}_R(X) \to Jac^{\mathfrak{s}}(X)$.

\begin{proposition}\label{prop:jacfix}
Assume that an equivariant splitting $s : H^1(X ; \mathbb{Z}) \to \mathcal{H}$ exists. Then we have an equality
\[
Jac^{\mathfrak{s}}(X)^{\sigma} = \bigcup_{\mathfrak{s}'} Jac^{\mathfrak{s}'}_R(X)
\]
where the union is over all Real structures on the underlying spin$^c$-structure $\mathfrak{s}$.
\end{proposition}
\begin{proof}
Recall that $Jac^{\mathfrak{s}}(X)$ is the set of gauge equivalence classes of spin$^c$-connections for $\mathfrak{s}$ with curvature equal to $F$. Suppose $A \in Jac^{\mathfrak{s}}(X)^{\sigma}$. Then $\widetilde{\sigma}^*(A)$ is gauge equivalent to $A$, so there exists $h \in \mathcal{G}$ such that $\widetilde{\sigma}^*(A) = A + h^{-1}dh$. Applying $\widetilde{\sigma}$ to this equality, we get that $h^{-1}dh = \sigma^*(h^{-1}dh)$. Let $\theta \in H^1(X ; \mathbb{Z})$ be the underlying homotopy class of $h : X \to S^1$. Then $\sigma^*(\theta) = \theta$. Recall that we have chosen an equivariant splitting $s : H^1(X ; \mathbb{Z}) \to \mathcal{H}$. Set $h_0 = s(\theta)$. Then $\sigma^*(h_0) = h_0$. Write $h = e^{if} h_0$ for some function $f : X \to \mathbb{R}$. Since $h^{-1}dh = \sigma^*(h^{-1}dh)$, we get $\sigma^*(df) = df$ and hence $\sigma^*(f) = f + c$ for some $c \in \mathbb{R}$. Applying $\sigma^*$ to both sides of this equality, we see that $c = 0$, hence $\sigma^*(f) = f$ and thus $\sigma^*(h) = h$. Setting $\sigma' = h \widetilde{\sigma}$, it follows that $(\sigma')^2 = 0$, so that $\sigma'$ defines a Real structure on $\mathfrak{s}$. Moreover, it is easily checked that $(\sigma')^*(A) = \widetilde{\sigma}^*(A) - h^{-1}dh = A$, so $A \in Jac_R^{\mathfrak{s}'}(X)$, where $\mathfrak{s}'$ denotes $\mathfrak{s}$ equipped with the Real structure $\sigma'$.

Conversely, if $A \in Jac_R^{\mathfrak{s}'}(X)$ for some Real structure $\sigma'$ on $\mathfrak{s}$, then $\sigma' = h \sigma$ for some $h : X \to S^1$. From $(\sigma')^2 = 1$, it follows that $\sigma^*(h) = h$ and from $(\sigma')^*(A) = A$, we deduce that $\widetilde{\sigma}^*(A) = A + h^{-1}dh$, hence $A \in Jac^{\mathfrak{s}}(X)^{\sigma}$.
\end{proof}

Let $f : S^{V,U} \to S^{V',U'}$ be the $O(2)$-equivariant Bauer--Furuta map of $(X , \sigma , \mathfrak{s})$ over $B = Jac^{\mathfrak{s}}(X)$. Let $f^{\sigma} : S^{V^\sigma,U^{\sigma}} \to S^{(V')^{\sigma}, (U')^{\sigma}}$ be the restriction of $f$ to the fixed point sets of $\sigma$. By Proposition \ref{prop:jacfix}, we have $B^\sigma = \bigcup_{\mathfrak{s}'} B^{\mathfrak{s}'}$, where $B^{\mathfrak{s}'} = Jac^{\mathfrak{s}'}_R(X)$ and where the union is over all Real structures on $\mathfrak{s}$. It follows that for any $\mathfrak{s}'$, the restriction $f^\sigma|_{B^{\mathfrak{s}'}}$ is the Real Bauer--Furuta invariant of $(X , \sigma , \mathfrak{s}')$. In other words, the Real Bauer--Furuta invariants for all Real structures on $\mathfrak{s}$ are obtained by restricting the $O(2)$-equivariant Bauer--Furuta map to different components of the fixed point set.

\subsection{Real Bauer--Furuta map}\label{sec:realbf}

In this section, we consider the direct construction of a Real Bauer--Furuta invariant. The construction is virtually identical to the contruction in Section \ref{sec:O(2)}, except that we work throughout with Real configurations. 

The Real Seiberg--Witten equations define a map $\mu_R : \mathcal{A}_{R,H} \to \mathcal{C}_{R,H}$ of trivial Hilbert bundles over $Conn_{R,H}$ the space of connections of the form $A' = A + \alpha$, where $A \in Conn$ is a fixed Real connection and $\alpha$ is an imaginary harmonic $1$-form such that $\sigma^*(\alpha) = -\alpha$. The map $\mu_R$ is equivariant with respect to $\mathcal{H}_R$. Recall that $\mathcal{H}_R$ fits into an exact sequence
\[
1 \to \mathbb{Z}_2 \to \mathcal{H}_R \to H^1(X ; \mathbb{Z})^{-\sigma} \to 0.
\]
Given a splitting $s : H^1(X ; \mathbb{Z})^{-\sigma} \to \mathcal{H}_R$, the quotient of $\mathcal{A}_{R,H}$ and $\mathcal{C}_{R,H}$ by $s(H^1(X ; \mathbb{Z})^{-\sigma})$ are Hilbert bundles over $Jac_R^{\mathfrak{s}}(X) = Conn_{R,H}/H^1(X;\mathbb{Z})^{-\sigma}$ and $\mu_R$ descends to a $\mathbb{Z}_2$-equviariant map of Hilbert bundles over $Jac_R^{\mathfrak{s}}(X)$. This is the map whose finite-dimensional approximation will define the Real Bauer--Furuta invariant.

Since $H^1(X ; \mathbb{Z})^{-\sigma}$ is a free abelian group, there is no obstruction to the existence of a splitting $s : H^1(X ; \mathbb{Z})^{-\sigma} \to \mathcal{H}_R$. However the splitting is not unique, unless $b_1(X)^{-\sigma} = 0$. We conclude that the Real Bauer--Furuta invariant always exists, but its construction depends on the choice of a splitting.

Suppose a splitting has been chosen. Then we get a Real Bauer--Furuta map
\[
f : S^{V,U} \to S^{V',U'}
\]
where $V,V',U,U'$ are real vector bundles over $B = Jac_R(X)$ (we use the reference connection $A$ to identify $Jac_R^{\mathfrak{s}}(X)$ with $Jac_R(X)$). $f$ is a map of sphere bundles over $Jac_R(X)$ and is $\mathbb{Z}_2$-equivariant, where $\mathbb{Z}_2 = \langle \tau \rangle$ acts as follows: $\tau$ is multiplication by $-1$ on $V,V'$ and is trivial on $U,U'$. We also have that $f$ sends infinity of each fibre of $S^{V,U}$ to infinity of each fibre of $S^{V',U'}$. We have that $V - V' = D_R$, the real part of the families index of the family of spin$^c$-Dirac operators parametrised by $Jac^{\mathfrak{s}}_R(X)$, $U' = U \oplus H^+(X)^{-\sigma}$ and $f|_{S^U} : S^U \to S^{U'}$ is the map induced by the inclusion $U \to U'$.

Recall that $Jac^{\mathfrak{s}}_R(X) = Conn_{R,H}/H^1(X ; \mathbb{Z})^{-\sigma}$. The space $Conn_{R,H}$ naturally defines a family of spin$^c$-connections and hence there is a well-defined families index over $Conn_{R,H}$. To descend this family to $Jac^{\mathfrak{s}}_R(X)$, we need to lift $H^1(X ; \mathbb{Z})^{-\sigma}$ to the group of Real gauge transformations. In other words, we need a splitting $s : H^1(X ; \mathbb{Z})^{-\sigma} \to \mathcal{H}_R$. Thus, $D_R$ will generally depend on the choice of splitting. In fact, any two splittings differ by a homomorphism $\lambda : H^1(X ; \mathbb{Z})^{-\sigma} \to \mathbb{Z}_2$. Such a homomorphism corresponds to a flat real line bundle $L_\lambda \to Jac_R^{\mathfrak{s}}(X)$ and it is not hard to see that if we change splitting by $\lambda$, then the corresponding families index $D_R$ will change to $L_\lambda \otimes D_R$.

In \cite{bk}, we gave a cohomological description of the Seiberg--Witten invariants in terms of the Bauer--Furuta map. By adapting this to the Real setting, we will see how to obtain the Real Seiberg--Witten invariants in terms of the Real Bauer--Furuta map.

Let $a,a',b,b'$ denote the ranks of $V,V',U,U'$. So $a-a' = d = (c(\mathfrak{s})^2 - \sigma(X))/8$, $b' - b = b_+(X)^{-\sigma}$. Assume $b_+(X)^{-\sigma} > 0$ and let $\phi$ denote a chamber. As in \cite[\textsection 3.2]{bk}, the chamber $\phi$ determines a lift $\tau^\phi_{V',U'} \in H^{a'+b'}_{\mathbb{Z}_2}( S^{V',U'} , S^{U} ; \mathbb{Z}_2)$ of the Thom class to relative cohomology. Then we have $f^*( \tau^{\phi}_{V',U'}) \in H^{a'+b'}_{\mathbb{Z}_2}( S^{V,U} , S^U ; \mathbb{Z}_2)$. Let $\nu S^U$ denote an open tubular neighbourhood of $S^U$ in $S^{V,U}$ and set $\widetilde{Y} = S^{V,U} \setminus \nu S^U$. Then $\widetilde{Y}$ is a manifold with boundary and $\mathbb{Z}_2$ acts freely. Set $Y = \widetilde{Y}/\mathbb{Z}_2$. We have isomorphisms
\[
H^*_{\mathbb{Z}_2}( S^{V,U} , S^U ; \mathbb{Z}_2) \cong H^*_{\mathbb{Z}_2}( \widetilde{Y} , \partial \widetilde{Y} ; \mathbb{Z}_2) \cong H^*(Y , \partial Y ; \mathbb{Z}_2)
\]
where the first isomorphism is by excision and the second uses that $\mathbb{Z}_2$ acts freely on $\widetilde{Y}$. Thus we can identify $f^*(\tau^\phi_{V',U'})$ with a class in $H^{a'+b'}(Y , \partial Y ; \mathbb{Z}_2)$. As in \cite{bk}, we have that $f^*(\tau^{\phi}_{V',U'})$ can be written in the form
\[
f^*( \tau^\phi_{V',U'}) = \eta^\phi \delta \tau_U,
\]
where $\delta : H^{*-1}(\partial Y ; \mathbb{Z}_2) \to H^*(Y , \partial Y ; \mathbb{Z}_2)$ is the coboundary map and $\eta \in H^*_{\mathbb{Z}_2}(B ; \mathbb{Z}_2) \cong H^*(B ; \mathbb{Z}_2)[v]$. One has
\[
H^*( \mathbb{RP}(V) ; \mathbb{Z}_2) \cong \frac{H^*(B ; \mathbb{Z}_2)[v]}{e(V)}
\]
where $v$ is the image of the generator of $H^1_{\mathbb{Z}_2}(pt ; \mathbb{Z}_2)$ under the composition $H^*_{\mathbb{Z}_2}(pt ; \mathbb{Z}_2) \to H^*_{\mathbb{Z}_2}( S(V) ; \mathbb{Z}_2) \cong H^*( \mathbb{RP}(V) ; \mathbb{Z}_2)$ (alternatively, it is the first Stiefel--Whitney class of the hyperplane line bundle $\mathcal{O}_V(-1) \to \mathbb{RP}(V)$). Here $e(V)$ denotes the mod $2$ equivariant Euler class of $V$, which is given by:
\[
e(V) = v^a + v^{a-1} w_1(V) + v^{a-1} w_2(V) + \cdots + w_a(V).
\]
Then $f^*(\tau^\phi_{V',U'}) = \eta^\phi \delta \tau_U$ for some $\eta^\phi \in H^{a'+b_+(X)^{-\sigma} - 1}( \mathbb{RP}(V) ; \mathbb{Z}_2 )$. Let $\pi : \mathbb{RP}(V) \to B$ denote the projection map. By a similar argument to the one used for the ordinary Seiberg--Witten invariants \cite{bk}, we have:

\begin{theorem}
Choose a splitting $s : H^1(X ; \mathbb{Z})^{-\sigma} \to \mathcal{H}_R$ and let $f : S^{V,U} \to S^{V',U'}$ be the corresponding Real Bauer--Furuta map. Let $\phi$ be a chamber. Then
\[
\mathbf{SW}_R^\phi(X , \mathfrak{s})(\theta) = \pi_*( \eta^\phi \theta ) \in H^{j  - (d-b_+(X)^{-\sigma})}( Jac_R(X) ; \mathbb{Z}_2)
\]
for all $\theta \in H^j_{\mathbb{Z}_2}( pt ; \mathbb{Z}_2)$, where $\eta^\phi \in H^{a'+b_+(X)^{-\sigma}-1}( \mathbb{RP}(V) ; \mathbb{Z}_2)$ satisfies
\[
f^*( \tau^\phi_{V',U'} ) = \eta^\phi \delta \tau_U.
\]
\end{theorem}

Since $H^*_{\mathbb{Z}_2}(pt ; \mathbb{Z}_2)$ is generated by $v$, the map $\mathbf{SW}^\phi_R(X , \mathfrak{s})$ is equivalent to the collection of invariants
\[
SW_{R,m}^\phi(X ,\mathfrak{s}) = \mathbf{SW}_R^\phi(X,\mathfrak{s})(v^m) \in H^{m - (d-b_+^{-\sigma}(X))}( Jac_R(X) ; \mathbb{Z}_2).
\]
These invariants depend on the choice of splitting $H^1(X ; \mathbb{Z})^{-\sigma} \to \mathcal{H}_R$ for $m > 1$, but the $m=0$ invariant is splitting independent.

Under a change of splitting $v$ gets sent to $v' = v + a$ for some $a \in H^1( Jac_R(X) ; \mathbb{Z}_2)$. Then by expanding $(v+a)^m$ and using $a^2 = 0$, we have
\[
\mathbf{SW}_R^\phi(X , \mathfrak{s})((v')^m ) = SW_{R,m}^\phi(X,\mathfrak{s}) + ma SW_{R,m-1}^\phi(X,\mathfrak{s}).
\]
In particular, if $SW^\phi_{R,m-1}(X , \mathfrak{s}) = 0$ or if $m$ is even, then $SW^\phi_{R,m}(X,\mathfrak{s})$ is splitting independent.

Thus, the classes $SW^\phi_{R,m}(X,\mathfrak{s})$ for even $m$ are independent of the choice of splitting and hence define invariants of $(X , \sigma , \mathfrak{s})$. The classes $SW^\phi_{R,m}(X,\mathfrak{s})$ for odd $m$ will generally be splitting dependent. However, we have shown that under a change of splitting, $D_R$ changes to $L_\lambda \otimes D_R$, where $L_\lambda \to Jac_R(X)$ is the real line bundle corresponding to the change of splitting $\lambda \in Hom( H^1(X ; \mathbb{Z})^{-\sigma} , \mathbb{Z}_2)$. This gives:

\begin{proposition}\label{prop:doddsplit}
If $d = (c(\mathfrak{s})^2 - \sigma(X))/8$ is odd, then there is a unique splitting $s : H^1(X ; \mathbb{Z})^{-\sigma} \to \mathcal{H}_R$ such that $w_1(D_R) = 0$.
\end{proposition}
\begin{proof}
Under a change of splitting $D_R$ changes to $L_\lambda \otimes D_R$. Hence $w_1(D_R)$ changes to $w_1(D_R) + d \lambda$. If $d$ is odd, then there is a unique $\lambda$ for which $w_1(D_R) + \lambda = 0$ and this gives the desired splitting.
\end{proof}

If $d$ is odd, then we will take the unique splitting with $w_1(D_R) = 0$ unless stated otherwise. In this case, the classes $SW^\phi_{R,m}(X,\mathfrak{s})$ define invariants of $(X , \sigma , \mathfrak{s})$ regardless of the parity of $m$. However, we will later show that with respect to this splitting $SW^\phi_{R,m}(X,\mathfrak{s}) = 0$ when $m$ is odd (Corollary \ref{cor:vanish}), so only the invariants for even $m$ are of interest.

\section{Integer invariants}\label{sec:integral}

In this section we consider the construction of an integer-valued Real Seiberg--Witten invariant. We also introduce an invariant $deg_R(X,\mathfrak{s})$, which was considered by Miyazawa in the case $b_1(X)^{-\sigma} = 0$. We also examine the relation between these two invariants.

\subsection{Integer Real Seiberg--Witten invariant}\label{sec:isw}

We consider the problem of obtaining integer valued Real Seiberg--Witten invariants. Of course we want to do this by orienting the moduli space $\mathcal{M}_R$ and evaluating it against classes in $H^*( B\mathcal{G}_R ; \mathbb{Z})$. Recall that $B \mathcal{G}_R \cong B\mathbb{Z}_2 \times Jac_R(X)$, so $H^*(B\mathcal{G}_R ; \mathbb{Z}) \cong H^*( B \mathbb{Z}_2 \times Jac_R(X) ; \mathbb{Z}) \cong H^*(Jac_R(X) ; \mathbb{Z}) \otimes_{\mathbb{Z}} H^*( \mathbb{RP}^\infty ; \mathbb{Z})$. Since $H^*(\mathbb{RP}^\infty ; \mathbb{Z})$ is $2$-torsion in positive degrees, the only way to obtain non-vanishing integer invariants is to pair $[\mathcal{M}_R]$ with classes in $H^*( Jac_R(X) ; \mathbb{Z})$. In particular, such invariants, when they are defined, will not require a choice of splitting.

Assume that $b_+(X)^{-\sigma} > 0$ and that a generic perturbation has been chosen, so that $\mathcal{M}_R$ is smooth and compact. Imitating the usual determinant line argument from ordinary Seiberg--Witten theory, we obtain a canonical isomorphism
\begin{equation}\label{equ:detTM}
det(T\mathcal{M}_R) \cong det( \mathcal{D}_R ) \otimes det( H^1(X ; \mathbb{R})^{-\sigma} ) \otimes det( H^+(X)^{-\sigma}),
\end{equation}
where $\mathcal{D}_R$ denotes the real part of the families index of the family of spin$^c$-Dirac operators parametrised by $\mathcal{B}_R^*$. Choose a splitting $\mathcal{H} \cong \mathbb{Z}_2 \times H^1(X ; \mathbb{Z})^{-\sigma}$, then we get an identification $\mathcal{B}_R^* \cong \mathbb{RP}^\infty \times Jac_R(X)$ and we have an isomorphism
\[
\mathcal{D}_R \cong \mathcal{O}(1) \otimes \pi^*( D_R)
\]
where $\mathcal{O}(1) \to \mathbb{RP}^\infty$ is the universal real line bundle, $\pi$ is the projection to $Jac_R(X)$ and $D_R \to Jac_R(X)$ is the real part of the family of spin$^c$ Dirac operators parametrised by $Jac(X)$, as in Section \ref{sec:realbf}. In particular, we have
\[
w_1(\mathcal{D}_R) = \pi^*( w_1(D_R) ) + dv
\]
where $d = (c(\mathfrak{s})^2 - \sigma(X) )/8$. Since $H^1(X ; \mathbb{R})^{-\sigma}$, $H^+(X)^{-\sigma}$ define trivial bundles over $\mathcal{M}_R$, we see that in order to orient $\mathcal{M}_R$, we should demand that $w_1(\mathcal{D}_R) = 0$. Thus we need that $d$ is even and $w_1(D_R) = 0$.

In order to obtain integer invariants we need to orient the moduli spaces $\mathcal{M}_R$ in such a way that the orientations are compatible with variations of the choice of metric and perturbation. In other words, in such a way that the cobordisms between moduli spaces for different choice of metric and perturbation are oriented cobordisms. The most obvious way to do this is to trivialise $det(\mathcal{D}_R)$ over the whole of $\mathcal{B}^*_R$. In this case it is clear that we do get oriented cobordisms between generic metrics and perturbations lying in the same chamber.

Thus if $d$ is even and $w_1(D_R) = 0$, then we obtain a system of coherent orientations on the moduli spaces $\mathcal{M}_R$ and well-defined integer invariants. Of course the invariant will depend on a choice of orientations for $H^1(X ; \mathbb{R})^{-\sigma}$, $H^+(X)^{-\sigma}$ and $D_R$. Changing these orientations will change the sign of the Real Seiberg--Witten invariant. We are now ready to define the integer-valued Real Seiberg--Witten invariants

\begin{definition}\label{def:zsw}
Let $X$ be a compact, oriented smooth $4$-manifold, $\sigma$ a Real structure, $\mathfrak{s}$ a Real spin$^c$-structure. Assume that $d = (c(\mathfrak{s})^2 - \sigma(X))/8$ is even, $b_+(X)^{-\sigma} > 0$ and $w_1(D_R) = 0$. Let $\mathfrak{o}$ be an orientation of $H^+(X)^{-\sigma} \oplus D_R$ over $Jac_R(X)$ and let $\phi$ be a chamber. Then the {\em integer-valued Real Seiberg--Witten invariant} is the integer cohomology class
\[
\mathbf{SW}^\phi_{R,\mathbb{Z}}(X , \mathfrak{s} , \mathfrak{o}) \in H^{b_+(X)^{-\sigma} - d}(Jac_R(X) ; \mathbb{Z})
\]
characterised by
\[
\langle [Jac_R(X)] , \theta \smallsmile \mathbf{SW}^\phi_{R,\mathbb{Z}}(X,\mathfrak{s},\mathfrak{o}) \rangle = \langle [\mathcal{M}_R] , \theta \rangle
\]
for all $\theta \in H^\delta( Jac_R(X) ; \mathbb{Z})$, $\delta = d - b_+(X)^{-\sigma} + b_1(X)^{-\sigma}$, where $\mathcal{M}_R$ is the moduli space of solutions to the Real Seiberg--Witten equations for a generic metric and perturbation lying in the chamber $\phi$. In this equality, we choose an orientation $\mathfrak{o}'$ on $H^1(X ; \mathbb{R})^{-\sigma}$ inducing an orientation on $Jac_R(X)$ so that $[Jac_R(X)]$ is defined, and we use $\mathfrak{o}' \wedge \mathfrak{o}$ as an orientation on $\mathcal{M}_R$ using the isomorphism (\ref{equ:detTM}).
\end{definition}

\begin{remark}\label{rem:noorn}
We will usually omit the choice of orientation $\mathfrak{o}$ from the notation and write $\mathbf{SW}^\phi_{R,\mathbb{Z}}(X,\mathfrak{s})$. Of course one should bear in mind that without choosing $\mathfrak{o}$, the invariant $\mathbf{SW}^\phi_{R,\mathbb{Z}}(X,\mathfrak{s})$ is only defined up to an overall sign.
\end{remark}

In general, it seems difficult to determine whether or not $D_R$ is orientable. However, if $b_1(X)^{-\sigma} = 0$, then $Jac_R(X) = \{1\}$ is a point and so $D_R$ is obviously orientable. In this case $\mathbf{SW}^\phi_{R,\mathbb{Z}}(X , \mathfrak{s} , \mathfrak{o})$ is non-zero only if $d - b_+(X)^{-\sigma} = 0$, in other words, the dimension of the moduli space is zero. Then $\mathbf{SW}^\phi_{R,\mathbb{Z}}(X , \mathfrak{s} , \mathfrak{o}) \in \mathbb{Z}$ is simply a signed count of points of $\mathcal{M}_R$. In the case that $\sigma$ is free, the class $det(ind(D_R))$ can be computed using \cite[Theorem 5.9]{mi2}.

More generally, if $d$ is even and $D_R$ is orientable, we define the {\em pure, integer-valued Real Seiberg--Witten invariant} $SW^\phi_{R,\mathbb{Z}}(X,\mathfrak{s}) \in \mathbb{Z}$ to be
\[
SW^\phi_{R,\mathbb{Z}}(X,\mathfrak{s}) = \begin{cases} \langle [Jac_R(X)] , \mathbf{SW}^\phi_{R,\mathbb{Z}}(X,\mathfrak{s}) \rangle & \text{if } d - b_+(X)^{-\sigma} + b_1(X)^{-\sigma} = 0, \\ 0 & \text{otherwise}. \end{cases}
\]

\begin{remark}
Definition \ref{def:zsw} can be extended to the case $w_1(D_R) \neq 0$ by using local coefficients. The result is an invariant 
\[
\mathbf{SW}^\phi_{R,\mathbb{Z}}(X,\mathfrak{s}) \in H^{b_+(X)^{-\sigma} - d}(Jac_R(X) ; \mathbb{Z}_w)
\]
valued in the $\mathbb{Z}$-valued local coefficient system corresponding to $w = w_1(D_R)$. However, one finds that if $w \neq 0$, then the mod $2$ reduction map $H^*( Jac_R(X) ; \mathbb{Z}_w) \to H^*(Jac_R(X) ; \mathbb{Z}_2)$ is injective. So there is nothing to be gained from considering these invariants.
\end{remark}

Let us consider the integer invariants from the point of view of the Real monopole map. One repeats the construction given in Section \ref{sec:realbf}, but using cohomology with integer coefficients. For this to work we need that the bundles $V,V',U,U'$ are orientable. Since $U' - U = H^+(X)^{-\sigma}$ is a trivial bundle, we can stabilise the Real Bauer--Furuta map so that $U,U'$ are trivial. Since $V - V' = D_R$, if $w_1(D_R) = 0$, then we can stabilise the Real Bauer--Furuta map so that both $V$ and $V'$ are orientable. Supposing $w_1(D_R) = 0$, we can assume $V,V',U,U'$ are oriented and use cohomology with integer coefficients. The Real Seiberg--Witten invariant then takes the form of a map
\[
\mathbf{SW}^\phi_{R,\mathbb{Z}}(X , \mathfrak{s}) : H^*_{\mathbb{Z}_2}(pt ; \mathbb{Z}) \to H^{* - (d-b_+(X)^{-\sigma})}( Jac_R(X) ; \mathbb{Z}).
\]
But since $H^*_{\mathbb{Z}_2}(pt ; \mathbb{Z})$ is $2$-torsion in positive degrees, the only interesting part of this map is given by its evaluation on $1 \in H^0_{\mathbb{Z}_2}(pt ; \mathbb{Z})$, giving a class
\[
\mathbf{SW}^\phi_{R,\mathbb{Z}}(X,\mathfrak{s}) \in H^{-(d-b_+(X)^{-\sigma})}(Jac_R(X) ; \mathbb{Z}),
\]
which is the invariant given in Definition \ref{def:zsw}.

\subsection{Degree of the Real Bauer--Furuta map}\label{sec:deg}

Let $X$ be a compact, oriented, smooth $4$-manifold, $\sigma$ a Real structure and $\mathfrak{s}$ a Real spin$^c$-structure. Suppose that $d = b_+(X)^{-\sigma}$ and $b_1(X)^{-\sigma} = 0$. Then the Real Bauer--Furuta map of $(X,\sigma , \mathfrak{s})$ is a $\mathbb{Z}_2$-equivariant stable homotopy class of a map of spheres $f : S^{V,U} \to S^{V',U'}$ of the same dimension. Forgetting the $\mathbb{Z}_2$-action, the underlying non-equivariant stable homotopy class of $f$ is determined by its degree. Miyazawa \cite{mi} defined an invariant $|deg_R(X,\mathfrak{s})| \in \mathbb{Z}$ which is defined as the absolute value of the degree of the Real Bauer--Furuta map $f : S^{V,U} \to S^{V',U'}$. The absolute value is taken because the spheres $S^{V,U}, S^{V',U'}$ do not have canonical orientations, so the degree of $f$ is only defined up to an overall sign factor. Instead of taking absolute value, we will regard the degree as an integer which is only well-defined up to an overall sign factor, and denote it by $deg_R(X , \mathfrak{s})$. This is consistent with the approach we have adopted for the integer-valued Real Seiberg--Witten invariant $\mathbf{SW}^\phi_{R,\mathbb{Z}}(X,\mathfrak{s})$ (Remark \ref{rem:noorn}).

The invariant $deg_R(X,\mathfrak{s})$ can be extended to the case where $b_1(X)^{-\sigma} > 0$. Let $s : H^1(X ; \mathbb{Z})^{-\sigma} \to \mathcal{H}_R$ be a splitting and $f : S^{V,U} \to S^{V',U'}$ the associated Real Bauer--Furuta map. This is a map of sphere bundles over $B = Jac_R(X)$. Assume that $D = V - V'$ is orientable and choose orientations for $D$ and $H^+(X)^{-\sigma}$. After suspending, we can assume $V$ and $V'$ are orientable and we can choose orientations on $V,V',U,U'$ such that $V \cong D \oplus V'$, $U' \cong U \oplus H^+(X)^{-\sigma}$ orientation preservingly. Letting $i : B \to S^{V,U}, i' : B \to S^{V',U'}$ be the infinity sections, we can regard $B$ as a subspace of $S^{V,U}$ and $S^{V',U'}$. Let $\tau_{V,U} \in H^*( S^{V,U} , B ; \mathbb{Z} )$, $\tau_{V',U'} \in H^*( S^{V',U'} , B ; \mathbb{Z})$ denote the Thom classes of $S^{V,U}, S^{V',U'}$. 

\begin{definition}\label{def:degree}
Let $X$ be a compact, oriented smooth $4$-manifold, $\sigma$ a Real structure, $\mathfrak{s}$ a Real spin$^c$-structure. Choose a splitting $s : H^1(X ; \mathbb{Z})^{-\sigma} \to \mathcal{H}_R$ and assume that $w_1(D_R) = 0$. Then the {\em Real degree} of $(X , \sigma , \mathfrak{s})$ (with respect to $s$) is the integer cohomology class (well-defined up to an overall sign)
\[
deg_R(X , \mathfrak{s}) \in H^{b_+(X)^{-\sigma} - d}( Jac_R(X) ; \mathbb{Z})
\]
characterised by
\[
f^*( \tau_{V',U'} ) = deg_R(X , \mathfrak{s}) \tau_{V,U}.
\]
\end{definition}

\begin{remark}
When $d$ is odd, there is a unique splitting for which $w_1(D_R) = 0$ and so $deg_R(X,\mathfrak{s})$ is defined. However, we will see that $deg_R(X,\mathfrak{s}) = 0$ for odd $d$, so the degree is only interesting when $d$ is even. When $d$ is even, the condition $w_1(D_R) = 0$ is independent of the choice of splitting and we will see that in this case $deg_R(X,\mathfrak{s})$ does not depend on the choice of splitting.
\end{remark}

\begin{proposition}\label{prop:mod2deg}
If $b_+(X)^{-\sigma} = 0$, then $d \le 0$ and $w_{j}(-D_R) = 0$ for $j > -d$. Moreover, if $w_1(D_R) = 0$, then the mod $2$ reduction of $deg_R(X,\mathfrak{s})$ equals $w_{-d}(-D_R)$.
\end{proposition}
\begin{proof}
Define the mod $2$ $\mathbb{Z}_2$-equivariant degree $\beta_2 \in H^*_{\mathbb{Z}_2}( B ; \mathbb{Z}_2)$ by $f^*( \tau_{V',U'} ) = deg_{\mathbb{Z}_2}(f) \tau_{V,U}$, where $\tau_{V,U}, \tau_{V',U'}$ are the $\mathbb{Z}_2$-equivariant Thom classes for $S^{V',U'}$ and $S^{V,U}$. Since $\mathbb{Z}_2$ acts trivially on $B$, we have $H^*_{\mathbb{Z}_2}(B ; \mathbb{Z}_2) \cong H^*(B ; \mathbb{Z}_2)[v]$ where $v$ is the generator of $H^1_{\mathbb{Z}_2}(pt ; \mathbb{Z}_2)$. Suppose $b_+(X)^{-\sigma} = 0$. Using the same argument as in \cite[\textsection 3]{bar} but in $\mathbb{Z}_2$-equivariant cohomology rather than $S^1$-equivariant, we find that $d \le 0$, $w_j(-D_R) = 0$ for $j > -d$ and that
\[
\beta_2 = v^{-d} + v^{-d-1} w_1(-D_R) + v^{-d-2} w_2(-D_R) + \cdots + w_{-d}(-D_R).
\]
Now assume that $w_1(D_R) = 0$. Passing from equivariant to non-equivariant cohomology, $\beta_2$ is sent to the mod $2$ reduction of $deg_R(X,\mathfrak{s})$. Since the forgetful map $H^*_{\mathbb{Z}_2}(B ; \mathbb{Z}_2) \to H^*(B ; \mathbb{Z}_2)$ is given by setting $v=0$, we see that $deg_R(X,\mathfrak{s}) = w_{-d}(-D_R) \; ({\rm mod} \; 2)$.
\end{proof}

\begin{corollary}
If $d = b_+(X)^{-\sigma} = 0$, then $w_1(D_R) = 0$.
\end{corollary}

\begin{proposition}
Suppose $d$ is even and $w_1(D_R) = 0$. Then $deg_R(X,\mathfrak{s})$ does not depend on the choice of splitting.
\end{proposition}
\begin{proof}
Since any two splittings $s_1,s_2 : H^1(X ; \mathbb{Z})^{-\sigma} \to \mathcal{H}_R$ differ by a homomorphism $H^1(X ; \mathbb{Z})^{-\sigma} \to \mathbb{Z}_2$, they agree on the subgroup $2 H^1(X ; \mathbb{Z})^{-\sigma} \subseteq H^1(X ; \mathbb{Z})^{-\sigma}$. This means that there is a canonical lift $s : 2 H^1(X ; \mathbb{Z})^{-\sigma} \to \mathcal{H}_R$. Now we imitate the usual construction of the Real Bauer--Furuta map, except that instead of quotienting by $s(H^1(X ; \mathbb{Z})^{-\sigma})$, we quotient by $s(2H^1(X;\mathbb{Z})^{-\sigma})$. The result is a $\Gamma$-equivariant map 
\[
\widetilde{f} : S^{\widetilde{V} , \widetilde{U}} \to S^{\widetilde{V}' , \widetilde{U}'}
\]
over $\widetilde{Jac}_R(X) = H^1(X;\mathbb{R})^{-\sigma}/2H^1(X;\mathbb{Z})^{-\sigma}$, where $\Gamma = \mathcal{H}_R/s(H^1(X;\mathbb{Z})^{-\sigma}$. Observe that $\Gamma$ is a $\mathbb{Z}_2$ central extension of $\mathbb{Z}_2^{\beta}$, where $\beta = b_1(X)^{-\sigma}$ and $\widetilde{Jac}_R(X) \to Jac_R(X)$ is a $\mathbb{Z}_2^\beta$-covering space. Splittings $s' : \mathbb{Z}_2 \to \Gamma$ are in bijection with splittings $s' : H^1(X ; \mathbb{Z})^{-\sigma} \to \mathcal{H}_R$ (there is a natural map from splittings $H^1(X;\mathbb{Z})^{-\sigma} \to \mathcal{H}_R$ to splittings $\mathbb{Z}_2^\beta \to \Gamma$ and this map is easily seen to be a bijection). Given a splitting $s : \mathbb{Z}_2^\beta \to \Gamma$, the Bauer--Furuta map corresponding to $s$ is obtained by quotienting $\widetilde{f}$ by $s(\mathbb{Z}_2^\beta)$. Set $\widetilde{D} = \widetilde{V} - \widetilde{V}'$. Then the virtual bundle $D_R$ corresponding to a chosen splitting $s$ is given by $D_R = \widetilde{D}/s(\mathbb{Z}_2^\beta)$. Since $\widetilde{f}$ is $\Gamma$-equivariant, it has a $\Gamma$-equivariant degree
\[
deg_{\Gamma}( \widetilde{f} ) \in H^{b_+(X)^{-\sigma}-d}_\Gamma( \widetilde{Jac}_R(X) ; \mathbb{Z}).
\]
Pulling back under a splitting $s : \mathbb{Z}_2^\beta \to \Gamma$, we obtain
\[
s^*( deg_{\Gamma}(\widetilde{f})) \in H^{b_+(X)^{-\sigma}-d}_{\mathbb{Z}_2^\beta}( \widetilde{Jac}_R(X) ; \mathbb{Z}) \cong H^{b_+(X)^{-\sigma}-d}(Jac_R(X) ; \mathbb{Z}).
\]
Clearly $s^*(deg_{\Gamma}(\widetilde{f}))$ coincides with $deg(f)$, where $f$ is the Bauer--Furuta invariant corresponding to $s$. So it remains to show that $s^*( deg_{\Gamma}(\widetilde{f}))$ does not depend on the choice of splitting.

Fix a splitting $s_0$, which we use to identify $\Gamma$ with $\mathbb{Z}_2 \times \mathbb{Z}_2^\beta$. Under this splitting
\[
H^*_{\Gamma}(\widetilde{Jac}_R(X) ; \mathbb{Z}) \cong H^*_{\mathbb{Z}_2}( Jac_R(X) ; \mathbb{Z}).
\]
We have an isomorphism
\[
H^*_{\mathbb{Z}_2}( Jac_R(X) ; \mathbb{Z}) \cong H^*( Jac_R(X) ; \mathbb{Z}) \otimes H^*_{\mathbb{Z}_2}(pt ; \mathbb{Z}).
\]
Recall that $H^*_{\mathbb{Z}_2}(pt ; \mathbb{Z}) \cong \mathbb{Z}[x]/(2x)$, where $deg(x) = 2$.

Now let $s : \mathbb{Z}_2^\beta \to \Gamma$ be any splitting. Then the pullback
\[
s^* : H^*_{\mathbb{Z}_2}( Jac_R(X) ; \mathbb{Z}) \to H^*( Jac_R(X) ; \mathbb{Z})
\]
is the map
\[
H^*( Jac_R(X) ; \mathbb{Z}) \otimes \mathbb{Z}[x]/(2x) \to H^*(Jac_R(X) ; \mathbb{Z})
\]
which sends $x$ to some class $a = s^*(x) \in H^*(Jac_R(X) ; \mathbb{Z})$. Since $2x=0$, we must also have $2a=0$. But $H^*(Jac_R(X) ; \mathbb{Z})$ has no $2$-torsion and hence $s^*( deg_{\Gamma}(\widetilde{f}))$ does not depend on the choice of $s$.
\end{proof}

One advantage of $deg_R(X,\mathfrak{s})$ over the integer-valued Seiberg--Witten invariant is that it is defined even when $b_+(X)^{-\sigma} = 0$. On the other hand for $b_+(X)^{-\sigma} > 0$, the two invariants are closely related, as the following result shows.

\begin{proposition}\label{prop:deginteger}
Let $X$ be a compact, oriented, smooth $4$-manifold, $\sigma$ a Real structure and $\mathfrak{s}$ a Real spin$^c$-structure. Assume that $w_1(D_R) = 0$.
\begin{itemize}
\item[(1)]{If $d$ is odd, then $deg_R(X,\mathfrak{s}) = 0$.}
\item[(2)]{If $d$ is even and $b_+(X)^{-\sigma}>0$, then $deg_R(X,\mathfrak{s}) = 2 \, \mathbf{SW}^\phi_{R,\mathbb{Z}}(X , \mathfrak{s})$ for any chamber $\phi$. In particular, $\mathbf{SW}^\phi_{R,\mathbb{Z}}(X,\mathfrak{s})$ does not depend on the choice of chamber.}
\item[(3)]{If $d$ is even and $b_+(X)^{-\sigma} = 0$, then we have $d \le 0$, $w_j(-D_R) = 0$ for $j > -d$ and $deg_R(X,\mathfrak{s}) = w_{-d}(-D_R) \; ({\rm mod} \; 2)$. In particular, if $d=0$, then $w_1(D_R) = 0$ and $deg_R(X , \mathfrak{s}) = 1 \; ({\rm mod} \; 2)$.}
\end{itemize}
\end{proposition}
\begin{proof}
Let $f : S^{V,U} \to S^{V',U'}$ be the Real Bauer--Furuta map. Recall that $f$ is $\mathbb{Z}_2 = \langle \tau \rangle$-equivariant, where $\tau$ acts as $-1$ on $V,V'$ and as $+1$ on $U,U'$. Let $a,a',b,b'$ denote the ranks of $V,V',U,U'$. Hence $a-a' = d$, $b'-b = b_+(X)^{-\sigma} = d$.

(1) Since $\tau$ acts as $-1$ on $V$, $+1$ on $U$, $\tau : S^{V,U} \to S^{V,U}$ has degree $(-1)^a$. Similarly, $\tau : S^{V',U'} \to S^{V',U'}$ has degree $(-1)^{a'}$. Since $f \circ \tau = \tau \circ f$, taking degrees yields $(-1)^{a'} deg(f) = (-1)^{a} deg(f)$, hence $deg(f) = (-1)^d deg(f)$. If $d$ is odd, then $deg(f) = -deg(f)$, so $2\, deg(f) = 0$.

If $w_1(D_R) = 0$, then $deg(f) = 0$ because $H^*(Jac_R(X) ; \mathbb{Z})$ has no $2$-torsion.

(2) If $d$ is even and $b_+(X)^{-\sigma} > 0$, then the integer-valued Real Seiberg--Witten invariant $\mathbf{SW}^\phi_{R,\mathbb{Z}}(X,\mathfrak{s})$ is defined. By suspending, we can assume $a,a'$ are all even and that $U,V'$ are trivial. Let $\tau^\phi_{V',U'} \in H^{a'+b'}_{\mathbb{Z}_2}( S^{V',U'} , S^U ; \mathbb{Z})$ be the lifted Thom class in $\mathbb{Z}_2$-equivariant cohomology. Then as in Section \ref{sec:realbf}, we can write $f^*( \tau^\phi_{V',U'} ) = \eta^\phi \delta \tau_U$ for some $\eta^\phi \in H^{a-1}(\mathbb{RP}(V) ; \mathbb{Z})$. Then $\mathbf{SW}_{R,\mathbb{Z}}(X,\mathfrak{s}) = (\pi_{\mathbb{RP}(V)})_*( \eta^\phi )$, where $\pi_{\mathbb{RP}(V)} : \mathbb{RP}(V) \to B$ is the projection map. Let $\rho : S(V) \to \mathbb{RP}(V)$ be the quotient map. Then
\begin{equation}\label{equ:2SW}
2 \, \mathbf{SW}^\phi_{R,\mathbb{Z}}(X,\mathfrak{s}) = (\pi_{S(V)})_*( \rho^*(\eta^\phi) )
\end{equation}
where $\pi_{S(V)} : S(V) \to B$ is the projection map. Observe that $H^*(\mathbb{RP}(V) ; \mathbb{Z}) \cong H^*_{\mathbb{Z}_2}( S(V) ; \mathbb{Z})$. Under this isomorphism, the forgetful map $H^{a-1}(\mathbb{RP}(V) ; \mathbb{Z}) \to H^{a-1}(S(V) ; \mathbb{Z})$ is just the pullback map $\rho^*$. Hence $\rho^*(\eta^\phi)$ can be understood in terms of non-equivariant cohomology.

Consider the long exact sequence of the triple $(S^{V',U'} , S^U , B)$ in ordinary cohomology
\[
\cdots \to H^{*-1}( S^{U} ; \mathbb{Z}) \buildrel \delta \over \longrightarrow H^*( S^{V',U'} , S^U ; \mathbb{Z}) \to H^*( S^{V',U'}, B ; \mathbb{Z}) \to H^*( S^U , B ; \mathbb{Z}) \to \cdots
\]
Replacing $f$ by a suspension if necessary, we can assume that $a' > b_1(X)^{-\sigma} + 1 - b_+(X)^{-\sigma}$. Then it follows that the restriction map $H^{a'+b'}( S^{V',U'} , S^U ; \mathbb{Z}) \to H^{a'+b'}( S^{V',U'} , B ; \mathbb{Z})$ is an isomorphism. In particular, the Thom class $\tau_{V',U'}$ has a unique lift $\tau^\phi_{V',U'}$ (or put another way, the image of the lifted Thom class in non-equivariant cohomology is independent of $\phi$). Similarly, the map $H^{a'+b'}(S^{V,U} , S^U ; \mathbb{Z}) \to H^{a',b'}( S^{V,U} , B ; \mathbb{Z})$ is an isomorphism. The former group is generated as a module over $H^*(Jac_R(X) ; \mathbb{Z})$ by $\nu_{S(V)} \delta \tau_U$, where $\nu_{S(V)}$ is the generator of $H^{a-1}(S(V),B ; \mathbb{Z})$ and the latter is generated by $\tau_{V,U}$. Now if $\beta = deg_R(f)$, then by definition, $f^*(\tau_{V',U'}) = \beta \tau_{V,U}$. On the other hand, we have $f^*( \tau^\phi_{V',U'}) = \rho^*(\eta^\phi) \delta \tau_U$. By commutativity of the diagram
\[
\xymatrix{
H^{a'+b'}( S^{V',U'} , S^U ; \mathbb{Z}) \ar[d] \ar[r]^-{f^*} & H^{a'+b'}( S^{V,U} , S^U ; \mathbb{Z}) \ar[d] \\
H^{a'+b'}(S^{V',U'} , B ; \mathbb{Z}) \ar[r]^-{f^*} & H^{a'+b'}(S^{V,U} , B ; \mathbb{Z})
}
\]
if follows that $\rho^*(\eta^\phi) = \beta \nu_{S(V)}$. Then Equation (\ref{equ:2SW}) gives 
\[
2 \, \mathbf{SW}^\phi_{R,\mathbb{Z}}(X,\mathfrak{s}) = \beta = deg(f).
\]

(3) Immediate from Proposition \ref{prop:mod2deg}.
\end{proof}

\begin{remark}
In light of Proposition \ref{prop:deginteger}, we see that if $w_1(D_R)=0$, then $\mathbf{SW}^\phi_{R , \mathbb{Z}}(X , \mathfrak{s})$ does not depend on the choice of chamber and we simply write the invariant as $\mathbf{SW}_{R,\mathbb{Z}}(X , \mathfrak{s})$. Similarly we write $SW_{R,\mathbb{Z}}(X,\mathfrak{s})$ instead of $SW^\phi_{R,\mathbb{Z}}(X,\mathfrak{s})$.
\end{remark}

\begin{remark}\label{rem:degcount}
It is possible to calculate $deg_R(X,\mathfrak{s})$ in terms of moduli spaces. Assume $b_1(X)^{-\sigma} = 0$, $d = b_+(X)^{-\sigma}$ and that $d$ is even. Then $\mathcal{G}_R \cong \{ \pm 1\} \times \{ e^{if} \; | \; f \in \Omega^0(X)^{-\sigma} \}$ by Proposition \ref{prop:Rgauge}. Define the {\em framed Real gauge group} to be $\widetilde{\mathcal{G}}_R = \{ e^{if} \; | \; f \in \Omega^0(X)^{-\sigma} \}$ and the {\em framed Real moduli space} $\widetilde{M}_R = \mathcal{N}_R/ \widetilde{\mathcal{G}}_R$, where $\mathcal{N}_R$ is the set of solutions to the Real Seiberg--Witten equations for a given metric and perturbation. If $\widetilde{\mathcal{M}}_R$ is cut out transversally, then it is a finite set of points and an argument similar to the one used in the proof of \cite[Theorem 2.24]{bk}, shows that the degree $deg_R(X , \mathfrak{s})$ of the Bauer--Furuta map is equal to a signed count of the number of points in $\widetilde{\mathcal{M}}_R$.
\end{remark}

\section{Properties of the Real Seiberg--Witten invariants}\label{sec:prop}

In this section we prove a number of fundamental properties of the Real Seiberg--Witten invariants.

\subsection{Positive scalar curvature}\label{sec:psc}

Let $X$ be a compact, oriented, smooth $4$-manifold and $\sigma$ a real structure. Suppose that there exists a $\sigma$-invariant Riemannian metric $g$ with positive scalar curvature. Using standard estimates for solutions to the Seiberg--Witten equations we find the following vanishing result.

\begin{proposition}
Suppose that $\sigma$ preserves a metric of positive scalar curvature. Let $\mathfrak{s}$ be a Real spin$^c$-structure on $X$.
\begin{itemize}
\item[(1)]{If $b_+(X)^{-\sigma} > 1$, then the mod $2$ Real Seiberg--Witten invariants and the integral Real Seiberg--Witten invariants (when defined) all vanish.}
\item[(2)]{If $b_+(X)^{-\sigma} = 1$ then the mod $2$ Real Seiberg--Witten invariants and the integral Real Seiberg--Witten invariants (when defined) vanish for the chamber containing the zero perturbation. If the zero perturbation lies on the wall, then the Real Seiberg--Witten invariants vanish for both chambers}
\item[(3)]{If $b_+(X)^{-\sigma} = b_1(X)^{-\sigma} = 0$ and $\mathfrak{s}$ is a Real spin structure, then $deg_R(X , \mathfrak{s}) = \pm 1$.}
\end{itemize}

\end{proposition}
\begin{proof}
(1) and (2) follow from standard estimates for the Seiberg--Witten equations which imply that for all a metric of positive scalar curvature and all sufficiently small perturbations, the only solutions to the Seiberg--Witten equations are reducible. In case (3), first note that a spin $4$-manifold with positive scalar curvature has $\sigma(X) = 0$, hence $d = -\sigma(X)/8 = 0$. The same estimates give that the only solutions to the Real Seiberg--Witten equations are reducible. Since $b_1(X)^{-\sigma} = 0$, there is up to isomorphism a unique reducible solution $(A,0)$. To prove the claim that $deg_R(X , \mathfrak{s}) = \pm 1$, it suffices to check that $(A,0)$ is cut out transversally (using Remark \ref{rem:degcount}). Since $(A,0)$ is reducible, the deformation complex for $(A,0)$ decomposes into sum of the Dirac operator and the Atiyah--Hitchin--Singer complex. It follows that the obstruction space is $H^+(X)^{-\sigma} \oplus Coker(D)^{\widetilde{\sigma}}$, where $D$ denotes the Dirac operator. But $b_+(X)^{-\sigma} = 0$ by assumption and $Coker(D)^{\widetilde{\sigma}} = 0$ because $\mathfrak{s}$ is a spin structure and $g$ has positive scalar curvature. This proves that $(A,0)$ is unobstructed, hence contributes $\pm 1$ to the degree.
\end{proof}

\subsection{Wall-crossing formula}\label{sec:wcf}

In light of Proposition \ref{prop:deginteger} (2), it suffices to consider wall-crossing only for the mod $2$ invariants. Choose a splitting and let $f : S^{V,U} \to S^{V',U'}$ be the Real Bauer--Furuta map over $B = Jac_R(X)$. Assume $b_+(X)^{-\sigma} = 1$ and let $\phi$ be a chamber. Then $SW^\phi_R(X,\mathfrak{s})(\theta) = \pi_*( \eta^\phi \theta)$, where $\pi$ is the projection $\pi : \mathbb{RP}(V) \to B$ and $f^*( \tau^\phi_{V',U'}) = \eta^\phi \delta \tau_U$. Choosing an orientation for $H^+(X)^{-\sigma}$ we may speak of the positive chamber and negative chamber and write $\mathbf{SW}^{\pm}_R(X , \mathfrak{s})$ for the corresponding invariants. Set $\Delta_R(X,\mathfrak{s}) = \mathbf{SW}^+_R(X , \mathfrak{s}) - \mathbf{SW}^-_R(X,\mathfrak{s})$. 

Using the approach of \cite[\textsection 5]{bk}, we find that $\tau^+_{V',U'} - \tau^-_{V',U'} =  e(V') \delta \tau_{U}$, where 
\[
e(V') = v^{a'} + w_1(V') v^{a'-1} + \cdots + w_{a'}(V').
\]
Therefore $f^*(\tau^+_{V',U'}) - f^*(\tau^-_{V',U'}) = f^*( e(V') \delta \tau_U ) = e(V') \delta \tau_U$. Hence $\eta^+ - \eta^- = e(V')$ and $\Delta_R(X,\mathfrak{s})(\theta) = \pi_*( e(V')\theta)$. Recall that Stiefel--Whitney classes can be extended to virtual bundles by setting $w(A-B) = w(A)w(B)^{-1}$, where $w = 1 + w_1 + \cdots $ denotes the total Stiefel--Whitney class. Adapting \cite[Proposition 3.5]{bk} to the case of real projective bundles, we find that $\pi_*( v^m ) = w_{m-(a-1)}(-V)$, where we set $w_i = 0$ if $i < 0$. By stabilising $f$ we can assume that $V'$ is trivial. Hence $e(V') = v^{a'}$ and $D_R = V - \mathbb{R}^{a'}$, which implies that $w_i(-D_R) = w_i(-V)$ for all $i$. Then $\Delta_R(X,\mathfrak{s})(\theta) = \pi_*( e(V')\theta ) = \pi_*( v^{a'} \theta)$. If we take $\theta = v^m$, then 
\[
\Delta_R(X , \mathfrak{s})(v^m) = \pi_*( v^{m+a'} ) = w_{m+a'-(a-1)}(-V) = w_{m-(d-1)}(-D_R).
\]
This gives:
\begin{theorem}\label{thm:wcf}
If $b_+(X)^{-\sigma}=1$, then
\[
SW^+_{R,m}(X,\mathfrak{s}) - SW^-_{R,m}(X, \mathfrak{s}) = w_{m-(d-1)}(-D_R).
\]
\end{theorem}

In the case $b_1(X)^{-\sigma} = 0$, $w_i(-D_R) = 0$ for all $i \neq 0$. Moreover the dimension of the moduli space is $d - b_+(X)^{-\sigma} = d-1$, hence we have

\begin{corollary}
If $b_+(X)^{-\sigma} = 1$ and $b_1(X)^{-\sigma} = 0$, then
\[
SW^+_R(X,\mathfrak{s}) - SW^-_R(X,\mathfrak{s}) = \begin{cases} 1 & \text{if } d > 0, \\ 0 & \text{otherwise}. \end{cases}
\]
\end{corollary}

\subsection{Identities from Steenrod squares}\label{sec:ident}

In \cite{bk} we gave a formula for the Steenrod squares of the families Seiberg--Witten invariants. In this section we will carry out a similar computation for the Real Seiberg--Witten invariants. However since the Steenrod squares are all vanishing for a torus, we obtain a series of identities that must be satisfied by the Real Seiberg--Witten invariants.

Let $(X,\sigma)$ be a compact, oriented, smooth $4$-manifold with real structure $\sigma$. Assume that $b_+(X)^{-\sigma} > 0$. Let $(\mathfrak{s} , \widetilde{\sigma})$ be a Real spin$^c$-structure. Let $f : S^{V,U} \to S^{V',U'}$ be the Real Bauer--Furuta invariant of $(X , \mathfrak{s})$ with respect to a splitting. By stabilising we may assume that $V$ and $U$ are trivial. Let $\phi$ be a chamber. We start with the expression $f^*( \tau^\phi_{V',U'} ) = \eta^\phi \delta \tau_U$. Recall that $SW^\phi_R(X,\mathfrak{s})(v^m) = \pi_*( \eta^\phi v^m )$, where $\pi$ is the projection $\pi : \mathbb{RP}(V) \to Jac_R(X)$. Our assumption that $V$ is trivial implies that $\pi_*(v^j) = 0$ for $j \neq a-1$ and $\pi_*(v^{a-1}) = 1$. Set $\omega_m = SW^\phi_{R,m}(X,\mathfrak{s})$. Then it follows that $\eta^\phi$ is given by
\begin{equation}\label{equ:eta}
\eta^\phi = \sum_{j=0}^{a-1} \omega_j v^{(a-1)-j}.
\end{equation}
Let $j > 0$. Since $\omega_l$ is a cohomology class on the torus $Jac_R(X)$, it follows that $Sq^j(\omega_l) = 0$ for all $l$. Applying $Sq^j$ to both sides of (\ref{equ:eta}) therefore gives
\begin{align*}
Sq^j(\eta^\phi) &= \sum_{l=0}^{a-1} \binom{(a-1)-l}{j} v^{(a-1)-l+j} \omega_l \\
&= \sum_{l=-j}^{a-1-j} \binom{(a-1)-l-j}{j} v^{(a-1)-l} \omega_{l+j} \\
&= \sum_{l=0}^{a-1} \binom{(a-1)-l-j}{j} v^{(a-1)-l}\omega_{l+j}.
\end{align*}
The last equality holds for if $l < 0$, then $v^{(a-1)-l} = 0$ and if $l+j > (a-1)$, then $\omega_{l+j} = 0$.

On the other hand the argument given in \cite[\textsection 4]{bk} gives
\[
Sq^j(\eta^\phi) = w_{\mathbb{Z}_2 , j}(V' \oplus U') \eta^\phi
\]
where $w_{\mathbb{Z}_2 , j}$ denotes the $j$-th equivariant Stiefel--Whitney class. Now since $U$ is trivial, $U' = H^+(X)^{-\sigma} \oplus U$ is also trivial. So 
\[
w_{\mathbb{Z}_2,j}(V'\oplus U') = w_{\mathbb{Z},j}(V') = \sum_{l = 0}^{j} \binom{a'-j+l}{l} v^{l} w_{j-l}(V').
\]
The second equality follows by adapting \cite[Lemma 4.1]{bk} to Stiefel--Whitney classes. Furthermore, $V$ is trivial and $D_R = V - V'$, hence $w_l(V') = w_l(-D_R)$. So we have
\begin{align*}
Sq^j(\eta^\phi) &= \sum_{l=0}^j \binom{a'-j+l}{l} v^{l} w_{j-l}(-D_R) \eta^\phi \\
&= \sum_{l=0}^j \sum_{m=0}^{a-1} \binom{a'-j+l}{l} v^{(a-1)-m+l} w_{j-l}(-D_R) \omega_m \\
&= \sum_{l=0}^{j} \sum_{m = -l}^{a-1-l} \binom{a' -j +l}{l} v^{(a-1)-m} w_{j-l}(-D_R) \omega_{m+l} \\
&= \sum_{l=0}^j \sum_{m=0}^{a-1} \binom{a'-j+l}{l} v^{(a-1)-m} w_{j-l}(-D_R) \omega_{m+l}.
\end{align*}
The last equality holds for if $m < 0$, then $v^{(a-1)-m} = 0$ and if $m+l > a-1$, then $\omega_{m+l} = 0$.

We have obtained two expressions for $Sq^j(\eta^\phi)$. Equating powers of $v$ in these expressions gives:
\begin{equation}\label{equ:binom}
\binom{a-1-m-j}{j}\omega_{m+j}= \sum_{l=0}^j \binom{a'-j+l}{l} w_{j-l}(-D_R) \omega_{m+l}.
\end{equation}

Stabilising $f$ by a trivial bundle, we can take $a$ to equal any sufficiently large integer. Fix $m$ and $j$ and take $a = m + 2^N$ where $2^N > j$. Then $\binom{a-1-m-j}{j} = \binom{ (2^N-1)-j }{j} = \binom{-1-j}{j} \; ({\rm mod} \; 2) = \binom{2j}{j} \; ({\rm mod} \; 2) = 0 \; ({\rm mod} \; 2)$ since $j>0$. Also since $a-a' = d$, we have that $\binom{a'-j+l}{l} = \binom{m-d+2^N - j + l}{j} = \binom{m-d-j+l}{l} \; ({\rm mod} \; 2) = \binom{d-1-m+j}{l} \; ({\rm mod} \; 2)$. Substituting into Equation (\ref{equ:binom}), we get

\[
\sum_{l=0}^j \binom{ d-1 - m + j}{l} w_{j-l}(-D_R) \omega_{m+l} = 0
\]
for all $m \ge 0$, $j > 0$. Since $\omega_m = SW^\phi_{R,m}(X,\mathfrak{s})$, this gives

\begin{theorem}\label{thm:rswid}
The Real Seiberg--Witten invariants satisfy the following identities
\[
\sum_{l=0}^j \binom{ d-1 - m + j}{l} w_{j-l}(-D_R) SW^\phi_{R,m+l}(X,\mathfrak{s}) = 0
\]
for all $m \ge 0$, $j > 0$.
\end{theorem}

For example, the $j=1,2$ cases give:
\begin{align*}
(d+m) SW^\phi_{R,m+1}(X,\mathfrak{s}) + w_1(-D_R) SW^\phi_{R,m}(X,\mathfrak{s}) &= 0, \\
\binom{d+1-m}{2} SW^\phi_{R,m+2}(X,\mathfrak{s}) + (d+1+m)w_1(-D_R) SW^\phi_{R,m+1}(X,\mathfrak{s}) & \\
+ w_2(-D_R) SW^\phi_{R,m}(X,\mathfrak{s}) &= 0.
\end{align*}

Suppose $d$ is odd. By Proposition \ref{prop:doddsplit}, there is a unique splitting for which $w_1(D_R) = 0$. Choosing this splitting the $j=1$ identity reduces to
\[
(m+1)SW^\phi_{R,m+1}(X,\mathfrak{s}) = 0
\]
for all $m \ge 0$. Equivalently, $SW^\phi_{R,m}(X,\mathfrak{s}) = 0$ for all odd $m$.

\begin{corollary}\label{cor:vanish}
Suppose that $d$ is odd. Choose the unique splitting for which $w_1(D_R) = 0$. Then $SW^\phi_{R,m}(X,\mathfrak{s}) = 0$ for all odd $m$.
\end{corollary}

In the case $b_+(X)^{-\sigma} = 1$, comparing with the wall-crossing formula gives:

\begin{corollary}\label{cor:oddvanish}
If $d > 0$ is odd, $b_+(X)^{-\sigma} = 1$ and we choose the splitting for which $w_1(D_R) = 0$, then $w_j(D_R) = 0$ for all odd $j$.
\end{corollary}
\begin{proof}
Let $m \ge 0$ be odd. Corollary \ref{cor:vanish} gives $SW^+_{R,m}(X,\mathfrak{s}) = SW^-_{R,m}(X,\mathfrak{s}) = 0$. However the wall-crossing formula (Theorem \ref{thm:wcf}) gives $SW^+_{R,m}(X,\mathfrak{s}) - SW^-_{R,m}(X,\mathfrak{s}) = w_{m-(d-1)}(-D_R)$. Hence $w_j(-D_R) = 0$ for all odd $j$. Hence $w(-D_R) = 1 + w_2(-D_R) + w_4(-D_R) + \cdots$ contains only terms of even degree. It follows that $w(D_R) = w(-D_R)^{-1}$ similarly contains only terms of even degree, hence $w_j(D_R) = 0$ for all odd $j$.
\end{proof}

\section{Real spin structures}\label{sec:spin}

By a {\em Real spin structure}, we mean a spin structure $\mathfrak{s}$ for which the involution $\sigma$ is odd. In such a case $\sigma$ admits a linear lift $\sigma'$ to the spinor bundles such that $\sigma' \circ \sigma' = -1$. Let $j : S_{\pm} \to S_{\pm}$ denote charge conjugation. Recall that $j$ is anti-linear and $j^2 = - 1$. We also have that $\sigma'$ commutes with $j$ because $\sigma'$ corresponds to a lift of $\sigma$ to the principal $Spin(4)$-bundles associated to $\mathfrak{s}$ and $Spin(4)$ preserves $j$. Set $\widetilde{\sigma} = j \circ \sigma'$. Then $\widetilde{\sigma}$ is an antilinear involutive lift of $\sigma$ and makes $\mathfrak{s}$ into a Real spin$^c$-structure. Up to isomorphism the resulting Real structure does not depend on the choice of lift $\sigma'$. Indeed taking the other lift $-\sigma'$ would give $-\widetilde{\sigma}$, but this is conjugate to $\widetilde{\sigma}$ as $-\widetilde{\sigma} = i^{-1} \circ \widetilde{\sigma} \circ i$.

The Seiberg--Witten equations (with zero perturbation) in the presence of a Real spin structure has an additional symmetry given by $j$. Now we attempt to construct the Bauer--Furuta map keeping track of this additional symmetry. We start as we did before with the map $\mu$ between trivial Hilbert bundles. Gauge fix as before to get a map of Hilbert bundles over $Conn_{R,H}$. The group of symmetries of this map is generated by $\mathcal{H}_R$ and $j$. Call this group $\mathcal{H}'_R$. Then we have a short exact sequence
\[
1 \to \mathbb{Z}_4 \to \mathcal{H}'_R \to H^1(X ; \mathbb{Z})^{-} \to 0
\]
where $\mathbb{Z}_4 = \langle j \rangle$ is the subgroup generated by $j$. A choice of splitting $s : H^1(X ; \mathbb{Z})^{-} \to \mathcal{H}_R$ is also a choice of splitting for $\mathcal{H}'_R$. Then we can quotient by $s(H^1(X ; \mathbb{Z})^{-\sigma})$ to obtain a map of Hilbert bundles over $Jac_R(X)$. Then we take a finite-dimensional approximation. The result is a $\mathbb{Z}_4$-equivariant map
\[
f : S^{V,U} \to S^{V',U'}
\]
such that if we restrict to $\mathbb{Z}_2 \subset \mathbb{Z}_4$, then we have the Real Bauer--Furuta map as before. In other words, for a Real spin structure, the Real Bauer--Furuta map has $\mathbb{Z}_4$-symmetry rather than just $\mathbb{Z}_2$. The symmetry group acts on $V,V',U,U'$ as follows: $j$ acts as $-1$ on $U,U'$ and $j^2$ acts as $-1$ on $V,V'$. Thus $j$ defines complex structures on $V,V'$. In particular this means that the real ranks $a,a'$ are even and thus $d = a-a'$ is even. Of course since $\mathfrak{s}$ is a spin structure, $d = -\sigma(X)/8$, which is even by Rokhlin's theorem.

The symmetry $j$ acts on $D_R$ and satisfies $j^2 = -1$. 

\begin{lemma}
We have $w_j(D_R) = 0$ for all odd $j$.
\end{lemma}
\begin{proof}
Since $D_R = V-V'$, it suffices to prove that $w_j(V) = w_j(V') = 0$ for all odd $j$. Consider the case of $V$, the case of $V'$ is similar. Let $e_{\mathbb{Z}_4}(V) \in H^*_{\mathbb{Z}_4}(B ; \mathbb{Z}_2)$ denote the $\mathbb{Z}_4$-equivariant Euler class of $V$ with $\mathbb{Z}_2$-coefficients. Let $R = H^*(B ; \mathbb{Z}_2)$. By a straighforward spectral sequence calculation, one finds $H^*_{\mathbb{Z}_4}(B ; \mathbb{Z}_2) \cong R[v,w]/(v^2)$, $deg(v) = 1$, $deg(w) = 2$. Consider the subgroup $\mathbb{Z}_2 \subset \mathbb{Z}_4$. Since $j^2 = 1$ on $B$, the subgroup $\mathbb{Z}_2$ acts trivially on $B$ and one finds $H^*_{\mathbb{Z}_2}(B ; \mathbb{Z}_2) \cong R[v]$, $deg(v) = 1$. It follows that we can expand $e_{\mathbb{Z}_4}(V)$ as a polynomial in $w$:
\[
e_{\mathbb{Z}_4}(V) = h_0(V) w^{a/2} + h_1(V) w^{a/2-1} + \cdots + h_{a/2}(V)
\]
for some $h_j(V) \in R[v]/(v^2)$. Now restrict from $\mathbb{Z}_4$-equivariant cohomology to $\mathbb{Z}_2$-equviariant cohomology. Then $e_{\mathbb{Z}_4}(V)$ maps to the $\mathbb{Z}_2$-equivariant Euler class of $V$, where the $\mathbb{Z}_2$-action acts as $-1$ on $V$ (and trivially on $B$). By the splitting principle, this is given by
\[
e_{\mathbb{Z}_2}(V) = v^a + w_1(V) v^{a-1} + \cdots + w_a(V).
\]
But under the forgetful map $H^*_{\mathbb{Z}_4}(B;\mathbb{Z}_2) \to H^*_{\mathbb{Z}_2}(B ; \mathbb{Z}_2)$ we have $v \mapsto 0$, $w \mapsto v^2$. Hence if we let $h'_j(V) \in R$ denote the image of $h_j(V)$ under the map $R[v]/(v^2) \to R$ given by setting $v=0$, we get
\[
h'_0(V) v^a + h'_1(V) v^{a-2} + \cdots + h'_{a/2}(V) = v^a + w_1(V) v^{a-1} + \cdots + w_a(V).
\]
Equating coefficients, we have $w_j(V) = 0$ for odd $j$.
\end{proof}

Since $w_1(D_R) = 0$, the $j=1$ case of Theorem \ref{thm:rswid} gives $m SW^\phi_{R,m+1}(X , \mathfrak{s}) = 0$ for all $m \ge 0$, equivalently $SW^\phi_{R,m}(X,\mathfrak{s}) = 0$ for all even, positive $m$. 

If $b_+(X)^{-\sigma} = 1$, then the wall-crossing formula gives $SW^\phi_{R,0}(X , \mathfrak{s}) - SW^\phi_{R,0}(X,\mathfrak{s}) = w_{1-d}(-D_R) = 0$, because $1-d$ is odd. Hence the invariant $SW^\phi_{R,0}(X,\mathfrak{s})$ is independent of the chamber and we will denote it by $SW_{R,0}(X,\mathfrak{s})$.

Recall that $H^*_{\mathbb{Z}_4}( pt ; \mathbb{Z}_2 ) \cong \mathbb{Z}_2[v,w]/(v^2)$, where $deg(v) = 1$, $deg(w) = 2$. Since $j$, the generator of the $\mathbb{Z}_4$-action acts as $-1$ on $H^+(X)^{-\sigma}$, there are no $j$-invariant chambers. However we can overcome this using the same method as \cite{bar1}. Namely, we consider the pullback of the Bauer--Furuta map $f$ to $\widehat{B} = S(H^+(X)^{-\sigma}) \times Jac_R(X)$. The group $\mathbb{Z}_4 = \langle j \rangle$ acts on $\widehat{B}$ acting as the antipodal map on $S(H^+(X)^{-\sigma})$ and as inversion on $Jac_R(X)$. On $\widehat{B}$ there is a tautological choice of chamber given by the projection map $\widehat{\phi} : \widehat{B} \to S(H^+(X)^{-\sigma})$. Using this chamber we can carry out the construction of Seiberg--Witten invariants $\mathbb{Z}_4$-equivariantly. Following the same method as used in \cite{bar1} (but with $\mathbb{Z}_4$ in place of $Pin(2)$), we obtain a map
\[
\mathbf{SW}^{\mathbb{Z}_4}_R(X , \mathfrak{s}) : H^*_{\mathbb{Z}_4}(pt ; \mathbb{Z}_2) \to H^{*-\delta}_{\mathbb{Z}_2}( \widehat{B} ; \mathbb{Z}_2)
\]
provided $b_+(X)^{-\sigma} > 0$, where $\delta = d - b_+(X)^{-\sigma}$. The map $\mathbf{SW}^{\mathbb{Z}_4}_R(X , \mathfrak{s})$ is a map of $H^*_{\mathbb{Z}_2}(pt ; \mathbb{Z}_2)$-modules and is a refinement of the Real Seiberg--Witten invariant in the sense that we have a commutative diagram
\[
\xymatrix{
H^*_{\mathbb{Z}_4}(pt ; \mathbb{Z}_2) \ar[d] \ar[rr]^-{\mathbf{SW}_R^{\mathbb{Z}_4}(X,\mathfrak{s}) } & & H^{*-(d - b_+(X)^{-\sigma})}_{\mathbb{Z}_2}( \widehat{B} ; \mathbb{Z}_2 ) \ar[d]^-{r^*} \\
H^*_{\mathbb{Z}_2}(pt ; \mathbb{Z}_2) \ar[rr]^-{\mathbf{SW}_R(X , \mathfrak{s})} & & H^*( Jac_R(X) ; \mathbb{Z}_2 )
}
\]
where $r^*$ is the map 
\[
H^{*}_{\mathbb{Z}_2}( S(H^+(X)^{-\sigma}) \times Jac_R(X) ; \mathbb{Z}_2 ) \to H^*_{\mathbb{Z}_2}( S^0 \times Jac_R(X) ; \mathbb{Z}_2) \cong H^*( Jac_R(X) ; \mathbb{Z}_2)
\]
induced by an inclusion $r : S^0 \to S(H^+(X)^{-\sigma})$. The forgetful map 
\[
H^*_{\mathbb{Z}_4}(pt ; \mathbb{Z}_2) \cong \mathbb{Z}_2[v,w]/(v^2) \to H^*_{\mathbb{Z}_2}(pt ; \mathbb{Z}_2) \cong \mathbb{Z}_2[v]
\]
is given by $v \mapsto 0$, $w \mapsto v^2$. Hence $v^m$ is in the image of the forgetful map if and only if $m$ is even. So we can recover the Real Seiberg--Witten invariant $SW_{R,0}(X,\mathfrak{s})$ by
\[
SW_{R,0}(X,\mathfrak{s}) = r^*( \mathbf{SW}^{\mathbb{Z}_4}_R(X , \mathfrak{s})( 1 ) ).
\]

The following vanishing result is the Real counterpart of \cite[Theorem 3.4]{bar1}.

\begin{theorem}\label{thm:spinvanish}
Let $\mathfrak{s}$ be a Real spin structure. If $b_+(X)^{-\sigma} > 2$, then $SW_{R,0}(X,\mathfrak{s}) = 0$.
\end{theorem}
\begin{proof}
Similar to \cite[Proposition 3.1]{bar1}, we have an isomorphism
\begin{equation}\label{equ:cohomiso1}
H^*_{\mathbb{Z}_2}( \widehat{B} ; \mathbb{Z}_2) \cong H^*( Jac_R(X) ; \mathbb{Z}_2)[v]/( v^{b_+(X)^{-\sigma}} ).
\end{equation}
Then since $v^2 = 0$ in $H^*_{\mathbb{Z}_4}(pt ; \mathbb{Z}_2)$, we have
\[
v^2 \mathbf{SW}^{\mathbb{Z}_4}_R(X , \mathfrak{s})(1) = \mathbf{SW}^{\mathbb{Z}_4}_R(X , \mathfrak{s})(v^2) = 0.
\]
Therefore under the isomorphism (\ref{equ:cohomiso1}), $\mathbf{SW}^{\mathbb{Z}_4}_R(X,\mathfrak{s})(1)$ must be of the form 
\[
\mathbf{SW}^{\mathbb{Z}_4}_R(X,\mathfrak{s})(1) = \alpha v^{b_+(X)^{-\sigma}-2} + \beta v^{b_+(X)^{-\sigma}-1}
\]
for some $\alpha,\beta \in H^*( Jac_R(X) ; \mathbb{Z}_2)$. But $SW_{R,0}(X,\mathfrak{s}) = r^*( \mathbf{SW}^{\mathbb{Z}_4}_R(X , \mathfrak{s})(1) )$ and $r^*(v^j ) = 0$ for all $j > 0$. So if $b_+(X)^{-\sigma} > 2$, then $r^*( v^{b_+(X)^{-\sigma}-2} ) = r^*(v^{b_+(X)^{-\sigma}-1}) = 0$ and hence $SW_{R,0}(X,\mathfrak{s}) = 0$.
\end{proof}

\begin{remark}
Theorem \ref{thm:spinvanish} says that if $\mathfrak{s}$ is a Real spin structure then $SW_{R,0}(X,\mathfrak{s}) = 0$ unless $b_+(X)^{-\sigma} = 1$ or $2$. In the case $b_+(X)^{-\sigma} = 1$ or $2$, one can use the methods of \cite{bar1} to express $SW_{R,0}(X,\mathfrak{s})$ in terms of certain characteristic classes of $D_R$. In particular, for $b_+(X)^{-\sigma} = 2$ one can show that $SW_{R,0}(X,\mathfrak{s}) = w_{2 + \sigma(X)/8}(-D_R)$. In particular, if $b_+(X)^{-\sigma} = 2$ and $\sigma(X) = -16$, then $SW_{R,0}(X,\mathfrak{s}) = 1$. We omit the details of the calculation.
\end{remark}

\section{Localisation formula}\label{sec:loc}

In this section we will use localisation techniques to compare the ordinary and Real Seiberg--Witten invariants. The nature of the localisation argument means that it will only give a mod $2$ relation. So we work throughout with $\mathbb{Z}_2$-coefficients and only consider the mod $2$ Real Seiberg--Witten invariants.

\subsection{Mod 2 Seiberg--Witten invariants}\label{sec:z2restr}

Before getting to the localisation argument itself we wish to show that the ordinary mod $2$ Seiberg--Witten invariant can be recovered from the Bauer--Furuta map with $S^1$-action restricted to $\mathbb{Z}_2 = \{ \pm 1 \} \subset S^1$. Let $f : S^{V,U} \to S^{V',U'}$ be an $S^1$-equivariant monopole map over a base $B$. Thus $V,V'$ are complex vector bundles, $U,U'$ are real vector bundles, $S^1$ acts by scalar multiplication on $V,V'$ and trivially on $U,U'$. We assume $f$ sends the infinity section of $S^{V,U}$ to the infinity section of $S^{V',U'}$ and that $f|_{S^U}$ is the map induced by an inclusion $U \to U'$. Set $D = V -V'$ and write $U' = U \oplus H^+$. Let $a,a'$ be the complex ranks of $V,V'$, $b,b'$ the real ranks of $U,U'$, set $d = a-a'$, $b_+ = b'-b$. Assume $b_+ > 0$ and let $\phi$ be a chamber. Then we may write 
\[
f^*( \tau^\phi_{V',U'}) = \eta^\phi \delta \tau_U,
\]
where $\eta^\phi \in H^*_{S^1}( \mathbb{CP}(V) ; \mathbb{Z}_2)$ and the abstract Seiberg--Witten invariants of $f$ are given by 
\[
SW_m^\phi(f) = (\pi_{\mathbb{C}})_*( \eta^\phi U^m ) \in H^{2m-(2d-b_+-1)}(B ; \mathbb{Z}_2),
\]
where $\pi_{\mathbb{C}}$ is the projection map $(\pi_{\mathbb{C}}) : \mathbb{CP}(V) \to B$. Note that $H^*( \mathbb{CP}(V) ; \mathbb{Z}_2) \cong H^*_{S^1}( S(V) ; \mathbb{Z}_2)$, where $S(V)$ is the unit sphere bundle of $V$. This gives $H^*( \mathbb{CP}(V) ; \mathbb{Z}_2)$ the structure of a module over $H^*_{S^1}(pt ; \mathbb{Z}_2) \cong \mathbb{Z}_2[U]$, $deg(U) = 2$. Let $R = H^*(B ; \mathbb{Z}_2)$. Then one finds
\begin{equation}\label{equ:cp}
H^*( \mathbb{CP}(V) ; \mathbb{Z}_2) \cong R[U]/\langle e_{S^1}(V) \rangle
\end{equation}
where $e_{S^1}(V) \in H^*_{S^1}(B ; \mathbb{Z}_2) \cong R[U]$ is the $S^1$-equivariant Euler class of $V$ (with mod $2$-coefficients):
\[
e_{S^1}(V) = U^a + w_2(V) U^{a-1} + \cdots + w_{2a}(V).
\]
We can similarly consider the real projective bundle $\pi_{\mathbb{R}} : \mathbb{RP}(V) \to B$. Then $H^*( \mathbb{RP}(V) ; \mathbb{Z}_2) \cong H^*_{\mathbb{Z}_2}( S(V) ; \mathbb{Z}_2)$ is a module over $H^*_{\mathbb{Z}_2}(pt ; \mathbb{Z}_2) \cong \mathbb{Z}_2[v]$, $deg(v) = 1$, namely
\begin{equation}\label{equ:rp}
H^*(\mathbb{RP}(V) ; \mathbb{Z}_2) \cong R[v]/ \langle e_{\mathbb{Z}_2}(V) \rangle
\end{equation}
where $e_{\mathbb{Z}_2}(V)$ is the $\mathbb{Z}_2$-equivariant Euler class of $V$ with mod $2$-coefficients:
\[
e_{\mathbb{Z}_2}(V) = v^{2a} + w_2(V) v^{2a-2} + \cdots + w_{2a}(V).
\]
Under the forgetful map $H^*_{S^1}(pt ; \mathbb{Z}_2) \to H^*_{\mathbb{Z}_2}(pt ; \mathbb{Z}_2)$, $U$ maps to $v^2$ and similarly under the forgetful map $H^*_{S^1}(B ; \mathbb{Z}_2) \to H^*_{\mathbb{Z}_2}(B ; \mathbb{Z}_2)$, $e_{S^1}(V)$ maps to $e_{\mathbb{Z}_2}(V)$. Let $\rho : \mathbb{RP}(V) \to \mathbb{CP}(V)$ be the natural projection. The induced pullback map $\rho^* : H^*( \mathbb{CP}(V) ; \mathbb{Z}_2) \to H^*(\mathbb{RP}(V) ; \mathbb{Z}_2)$ can be identified with the forgetful map $H^*_{S^1}(S(V) ; \mathbb{Z}_2) \to H^*_{\mathbb{Z}_2}(S(V) ; \mathbb{Z}_2)$. In terms of the isomorphisms (\ref{equ:cp}), (\ref{equ:rp}), $\rho^*$ is simply given by $U \mapsto v^2$.

Recall that $(\pi_{\mathbb{C}})_*( U^m ) = w_{2m-(2a-2)}(-V)$ and $(\pi_{\mathbb{R}})_*( v^m ) = w_{m-(2a-1)}(-V)$. It follows that $(\pi_{\mathbb{C}})_*( \theta ) = (\pi_{\mathbb{R}})_*( v \rho^*(\theta) )$ for all $\theta \in H^*(\mathbb{CP}(V) ; \mathbb{Z}_2)$. In particular, 
\begin{equation}\label{equ:rpcp}
SW_m^\phi(f) = (\pi_{\mathbb{C}})_*( \eta^\phi U^m ) = (\pi_{\mathbb{R}})_*( \rho^*(\eta^\phi) v^{2m+1} ).
\end{equation}
The isomorphism $H^*( \mathbb{RP}(V) ; \mathbb{Z}_2 ) \cong H^*_{\mathbb{Z}_2}( S(V) ; \mathbb{Z}_2)$ expresses the cohomology of $\mathbb{RP}(V)$ in terms of $\mathbb{Z}_2$-equivariant cohomology of $S(V)$. Using this, it follows from Equation (\ref{equ:rpcp}) that the mod $2$ Seiberg--Witten invariants $SW^\phi_m(f)$ of $f$ can be computed using $\mathbb{Z}_2$-equivariant cohomology instead of $S^1$-equivariant cohomology.

\subsection{Localisation}\label{sec:loc2}

Now we proceed to the localisation argument. The idea is to begin with the ordinary Bauer--Furuta monopole map, but keeping track of the additional symmetry given by the Real structure. Recall that this requires a choice of equivariant splitting $H^1(X ; \mathbb{Z}) \to \mathcal{H}$. A splitting exists if $\sigma$ does not act freely, or if $b_1(X)^{-\sigma} = 0$. We assume a splitting exists so that the Bauer--Furuta map exists and takes the form of an $O(2)$-equivariant map
\[
f : S^{V,U} \to S^{V',U'}
\]
over $B = Jac(X)$, where $V,V'$ are Real vector bundles (with respect to the real structure $\sigma : Jac(X) \to Jac(X)$ given by $L \mapsto \sigma^*(L)^*$), the subgroup $S^1 \subset O(2)$ acts on $V,V'$ by scalar multiplication and acts trivially on $U,U'$. The bundles $U,U'$ are real and $\mathbb{Z}_2$-equivariant. $V - V' = D$, the index of the Dirac operator and $U' = U \oplus H^+(X)$. The action of $\mathbb{Z}_2 = \langle \sigma \rangle$ on $H^+(X)$ is by minus pullback $a \mapsto -\sigma^*(a)$.

We have seen that the mod $2$ Seiberg--Witten invariants of $f$ are detected using only $\mathbb{Z}_2 \subset S^1$. Keeping track of the Real structure $\sigma$, this motivates us to consider the subgroup $\mathbb{Z}_2 \times \mathbb{Z}_2 \subset O(2)$, generated by $\sigma$ and $\tau = -1 \in S^1$. So now we forget the full $O(2)$-action and regard $f$ as a $\mathbb{Z}_2 \times \mathbb{Z}_2 = \langle \tau,\sigma \rangle$-equivariant map $f : S^{V,U} \to S^{V',U'}$. Here $\tau$ acts trivially on the base $B = Jac(X)$, acts as $-1$ on $V,V'$, acts trivially on $U,U'$ and $\sigma$ acts as described previously. If we restrict to the $\sigma$-invariant part
\[
f^\sigma : S^{V^\sigma , U^\sigma} \to S^{(V')^\sigma , (U')^\sigma}
\]
we obtain a $\mathbb{Z}_2 = \langle \tau \rangle$-equivariant map over $B^\sigma$. By Proposition \ref{prop:jacfix}, we have
\[
B^{\sigma} = \bigcup_{\mathfrak{s}'} B^{\mathfrak{s}'}
\]
where the union is over Real structures on $\mathfrak{s}$ and $B^{\mathfrak{s}'} = Jac_R^{\mathfrak{s}'}(X)$. Furthermore, the restriction $f^{\mathfrak{s}'} = f^\sigma |_{B^{\mathfrak{s}'}}$ is the Real Bauer--Furuta map for the Real spin$^c$-structure $\mathfrak{s}'$. Thus $f^{\sigma} = \bigcup_{\mathfrak{s}'} f^{\mathfrak{s}'}$ is a disjoint union of Real Bauer--Furuta maps.

We use localisation to compare the mod $2$ Seiberg--Witten invariants of $f$ and $f^\sigma$. Let the real ranks of $V,V'$ be $2a,2a'$ and let the real ranks of $U,U'$ be $b,b'$. Then $a-a' = d = (c(\mathfrak{s})^2 - \sigma(X))/8$, $b' - b  = b_+(X)$. Let $\phi$ be a chamber for $f^{\sigma}$. Then under the inclusion $H^+(X)^{-\sigma} \to H^+(X)$, $\phi$ also defines a chamber for $f$. 

We recall the definitions of the Seiberg--Witten invariants of $f$ and $f^\sigma$. First consider $f$. Let $\tau^\phi_{V',U'} \in H^{2a'+b'}_{\mathbb{Z}_2}( S^{V',U'} , S^U ; \mathbb{Z}_2)$ be the lifted Thom class. Then similar to Section \ref{sec:realbf}, we can write
\[
f^*( \tau^\phi_{V',U'} ) = \eta^\phi \delta \tau_U,
\]
where $\eta^\phi \in H^{2a'+b_+-1}_{\mathbb{Z}_2}( S(V) ; \mathbb{Z}_2) \cong H^{2a'+b_+-1}( \mathbb{RP}(V) ; \mathbb{Z}_2)$ and by Equation (\ref{equ:rpcp}) the Seiberg--Witten invariants $SW^\phi_m(f)$ of $f$ are given by
\[
SW^\phi_m(f) = (\pi_V)_*( \eta^\phi v^{2m+1} ) \in H^{2m-(2d-b_+(X)-1)}(B ; \mathbb{Z}_2)
\]
where $\pi_V$ is the projection $\pi_V : \mathbb{RP}(V) \to B$.

Next consider $f^\sigma$. Write $b = b_+ + b_-$, $b' = b'_+ + b'_-$, where $b_{\pm}$ are the ranks of $U^{\pm \sigma}$ (the $\pm 1$-eigenspaces of $\sigma$ on $U$) and $b'_{\pm}$ are defined similarly. Then $b'_+ - b_+ = b_+(X)^{-\sigma}$, $b'_- - b_- = b_+(X)^{\sigma}$. Observe that since $V,V'$ are complex vector bundles and $\sigma$ acts anti-linearly, the $\pm 1$-eigenspaces of $\sigma$ on $V$ each have real dimension $a$ and the $\pm 1$-eigenspaces of $\sigma$ on $V'$ each have real dimension $a'$. We have 
\[
(f^\sigma)^*( \tau^\phi_{(V')^\sigma , (U')^\sigma} ) = \eta^\phi_R \delta \tau_{U^\sigma},
\]
where $\eta^\phi_R \in H^{a'-1+b_+(X)^{-\sigma}}( \mathbb{RP}(V^\sigma) ; \mathbb{Z}_2)$ and the Seiberg--Witten invariants of $f^\sigma$ are given by 
\[
SW^\phi_m(f^\sigma) = (\pi_{V^\sigma})_*( \eta^\phi_R v^m ) \in H^{m - (d-b_+(X)^{-\sigma})}( B^\sigma ; \mathbb{Z}_2).
\]
If $f$ is the Bauer--Furuta invariant of $(X , \mathfrak{s})$, then similar to the construction in Section \ref{sec:realbf}, we get a mod $2$ Seiberg--Witten invariant
\[
\mathbf{SW}^\phi : H^*_{S^1}(pt ; \mathbb{Z}_2) \to H^{*-(2d-b_+(X)-1)}(Jac(X) ; \mathbb{Z}_2).
\]
and $SW^\phi_m(f) = \mathbf{SW}^\phi_m(X,\mathfrak{s})(U^m)$, where $U$ is the generator of $H^2_{S^1}(pt ; \mathbb{Z}_2)$. The ordinary pure mod $2$ Seiberg--Witten invariant of $(X,\mathfrak{s})$ is given by $SW^\phi(X,\mathfrak{s}) = \langle [Jac(X)] , SW^\phi_\delta(f) \rangle$, where $\delta = (2d - b_+(X) - 1 + b_1(X))/2$ is half the dimension of the moduli space.

Similarly, if $f$ is the Bauer--Furuta invariant of $(X,\mathfrak{s})$, then $f^\sigma$ is the disjoint union of Real Bauer--Furuta invariants of $(X,\sigma , \mathfrak{s}')$ where $\mathfrak{s}'$ runs over all Real structures on $\mathfrak{s}$. Moreover $H^*( B^\sigma ; \mathbb{Z}_2) = \bigoplus_{\mathfrak{s}'} H^*( B^{\mathfrak{s}'} ; \mathbb{Z}_2)$ and 
\[
SW^\phi_m(f^\sigma) = \bigoplus_{\mathfrak{s}'} SW^\phi_{R,m}(X,\mathfrak{s}')
\]
is the direct sum of the Real mod $2$ Seiberg--Witten invariants of $(X,\mathfrak{s}',\sigma)$.

It will be useful also to consider $f^{\tau \sigma}$, the restriction of $f$ to the fixed points of $\tau \sigma$. Since $\tau \sigma = i^{-1} \circ \sigma \circ i$, we see that $f^{\tau \sigma}$ and $f^{\sigma}$ can be identified with each other. However they play slightly different roles in the localisation theorem so it will be important to consider both maps.

The Seiberg--Witten invariants $SW^\phi_m(f), SW^\phi_m(f^\sigma)$ that we are interested in are defined using $\mathbb{Z}_2 = \langle \tau \rangle$-equivariant cohomology. For localisation, we need to make use of the larger symmetry group $\mathbb{Z}_2 \times \mathbb{Z}_2 = \langle \tau, \sigma \rangle$. Repeating the constructions of these invariants but keeping track of the extra symmetry, we will get classes
\[
\widehat{SW}^\phi_{m_0,m_1}(f) \in H^*_{\mathbb{Z}_2}( B ; \mathbb{Z}_2), \quad \widehat{SW}^\phi_{m_0,m_1}(f^{\sigma}) \in H^*_{\mathbb{Z}_2}( B^\sigma ; \mathbb{Z}_2)
\]
where the $\mathbb{Z}_2$-action on $B$ is by $\sigma$ and the $\mathbb{Z}_2$-action on $B^\sigma$ is trivial. These are defined similar to their non-equivariant counterparts. We have lifted Thom classes $\tau^\phi_{V',U'} \in H^*_{\mathbb{Z}_2 \times \mathbb{Z}_2}( S^{V',U'} , S^{U} ; \mathbb{Z}_2)$, then
\[
f^*( \tau^\phi_{V',U'} ) = \widehat{\eta}^\phi \delta \tau_U
\]
for some $\widehat{\eta}^\phi \in H^*_{\mathbb{Z}_2}( \mathbb{RP}(V) ; \mathbb{Z}_2)$ and
\[
\widehat{SW}^\phi_{m_0,m_1}(f) = (\pi_V)_*( \widehat{\eta}^\phi v^{m_0} (v+u)^{m_1} )
\]
and where $u$ is the generator of $H^1_{\mathbb{Z}_2}(pt ; \mathbb{Z}_2)$ for the group $\mathbb{Z}_2 = \langle \sigma \rangle$. Similarly, we will have
\[
(f^\sigma)^*( \tau^\phi_{(V')^\sigma , (U')^\sigma} ) = \widehat{\eta}^\phi_R \delta \tau_{U^\sigma},
\]
for some $\widehat{\eta}^\phi_R \in H^*_{\mathbb{Z}_2}( \mathbb{RP}(V^\sigma) ; \mathbb{Z}_2)$ and then
\[
\widehat{SW}^\phi_{m_0,m_1}(f^{\sigma}) = (\pi_{V^{\sigma}})_*( \widehat{\eta}^\phi_R v^{m_0} (v+u)^{m_1}).
\]
Under the forgetful maps from equivariant to ordinary cohomology, we have $\widehat{\eta}^\phi \mapsto \eta^\phi$, $\widehat{\eta}^\phi_R \mapsto \eta^\phi_R$ and $\widehat{SW}^\phi_{m_0,m_1}(f^\sigma) \mapsto SW^\phi_{m_0+m_1}(f^\sigma)$. Furthermore, if $m_0 + m_1 = 2m+1$ for some $m \ge 0$, then $\widehat{SW}^\phi_{m_0,m_1}(f) \mapsto SW^\phi_{m}(f)$.

In the case of $f^\sigma$, the group $\mathbb{Z}_2 = \langle \sigma \rangle$ acts trivially on all spaces involved. In particular, $H^*_{\mathbb{Z}_2}( \mathbb{RP}(V^\sigma) ; \mathbb{Z}_2) \cong H^*(\mathbb{RP}(V^\sigma) ; \mathbb{Z}_2)[u]$ and it follows easily that $\widehat{\eta}^\phi_R = \eta^\phi_R$ under this isomorphism. We also have $H^*_{\mathbb{Z}_2}( B^\sigma ; \mathbb{Z}_2) \cong H^*( B^\sigma ; \mathbb{Z}_2)[u]$ and by expanding $(v+u)^{m_1}$, we find
\[
\widehat{SW}^\phi_{m_0,m_1}(f^\sigma) = \sum_{j \ge 0} \binom{m_1}{j} u^{m_1-j} SW^\phi_{m_0+j}(f^\sigma).
\]

While the action of $\sigma$ is non-trivial on $B$, a simple spectral sequence computation shows that we still have an isomorphism $H^*_{\mathbb{Z}_2}( B ; \mathbb{Z}_2) \cong H^*(B ; \mathbb{Z}_2)[u]$ and the forgetful map $H^*_{\mathbb{Z}_2}(B ; \mathbb{Z}_2) \to H^*(B ; \mathbb{Z}_2)$ is given by setting $u=0$.

\subsection{Case $b_1(X) = 0$}\label{sec:b1=0loc}

We first consider the case where $b_1(X) = 0$. Then $B = Jac(X) = \{1\}$ is a single point, $V,V,U,U'$ are vector spaces and $f : S^{V,U} \to S^{V',U'}$ is a map of spheres.

Observe that $\mathbb{RP}(V)^\sigma = \mathbb{RP}(V^\sigma) \cup \mathbb{RP}(V^{\tau\sigma})$. Let $\iota_\sigma : \mathbb{RP}(V^\sigma) \to \mathbb{RP}(V)$, $\iota_{\tau \sigma} : \mathbb{RP}(V^{\tau \sigma}) \to \mathbb{RP}(V)$ be the inclusions. The group $\mathbb{Z}_2 = \langle \sigma \rangle$ acts on $\mathbb{RP}(V)$ and trivially acts on $\mathbb{RP}(V^\sigma), \mathbb{RP}(V^{\tau \sigma})$. Let $N^\sigma, N^{\tau \sigma}$ denote the normal bundles. These are $\mathbb{Z}_2$-equivariant vector bundles over $\mathbb{RP}(V^\sigma), \mathbb{RP}(V^{\tau \sigma})$ where $\sigma$ acts as multiplication by $-1$. Clearly $N^{\sigma} \cong V^{\sigma} \otimes \mathcal{O}(1)$, where $\mathcal{O}(1)$ denotes the hyperplane line bundle on $\mathbb{RP}(V^\sigma)$. Similarly $N^{\tau \sigma} \cong V^{\tau \sigma} \otimes \mathcal{O}(1)$. Now consider $\mathbb{Z}_2$-equivariant cohomology groups. We have $H^*_{\mathbb{Z}_2}(pt ; \mathbb{Z}_2) \cong \mathbb{Z}_2[u]$, where $deg(u) = 1$. Consider the hyperplane line bundle $\mathcal{O}(1) \to \mathbb{RP}(V)$. There are two ways of making this into a $\mathbb{Z}_2$-equivariant line bundle. We choose the lift of $\sigma$ to $\mathcal{O}(1)$ which equals the identity over $\mathbb{RP}(V^\sigma)$ and equals minus the identity over $\mathbb{RP}(V^{\tau\sigma})$. Set $v \in H^*_{\mathbb{Z}_2}( \mathbb{RP}(V) ; \mathbb{Z}_2)$ be the equivariant first Stiefel--Whitney class of $\mathcal{O}(1)$ with respect to this lift. We have
\[
H^*_{\mathbb{Z}_2}( \mathbb{RP}(V) ; \mathbb{Z}_2 ) \cong \mathbb{Z}_2[u,v]/( v^a(v+u)^a ).
\]
Similarly
\begin{equation}\label{equ:rpiso}
H^*_{\mathbb{Z}_2}( \mathbb{RP}(V^\sigma) \; \mathbb{Z}_2) \cong H^*_{\mathbb{Z}_2}(\mathbb{RP}(V^{\tau \sigma}) \cong \mathbb{Z}_2[u,v]/(v^a)
\end{equation}
where in both cases, $v$ equals the first equivariant Stiefel--Whitney class of $\mathcal{O}(1)$ with respect to the trivial lift of the $\mathbb{Z}_2$-action. It follows that $\iota_\sigma^*(v) = v$, $\iota_{\tau \sigma}^*(v) = v+u$.

\begin{lemma}\label{lem:restr}
Under the isomorphisms given in (\ref{equ:rpiso}), we have
\[
\iota_\sigma^*(\widehat{\eta}^\phi) = \iota_{\tau \sigma}^*(\widehat{\eta}^\phi) = u^{b_+(X)^\sigma} (u+v)^{a'} \eta^\phi_R.
\]

\end{lemma}
\begin{proof}
The argument is similar to \cite[Lemma 6.4]{bar2}. Consider the commutative diagram
\[
\xymatrix{
S^{V,U} \ar[r]^-{f} & S^{V',U'} \\
S^{V^\sigma,U^\sigma} \ar[r]^-{f^\sigma} \ar[u]^-{\iota_\sigma} & S^{(V')^{\sigma},(U')^{\sigma}} \ar[u]^-{\iota_\sigma}
}
\]
This induces a commutative diagram in equivariant cohomology groups. We have
\begin{align*}
\iota_{\sigma}^* f^*( \tau^\phi_{V',U'} ) &= \iota^*_\sigma( \widehat{\eta}^\phi \delta \tau_U ) \\
&= u^{b_-} \iota^*_\sigma( \widehat{\eta}^\phi ) \delta \tau_{U^\sigma}
\end{align*}
and
\begin{align*}
(f^\sigma)^* \iota^*_\sigma( \tau^\phi_{V',U'}) &= (f^\sigma)^*( u^{b'_-}(u+v)^{a'} \tau^\phi_{(V')^\sigma , (U')^\sigma} ) \\
&= u^{b'_-} (u+v)^{a'} \eta^\phi_R \delta \tau_U.
\end{align*}
Equating the two expressions give
\[
\iota^*_\sigma(\widehat{\eta}^\phi) = u^{b'_-  - b_-} (u+v)^{a'} \eta^\phi_R = u^{b_+(X)^\sigma} (u+v)^{a'} \eta^\phi_R.
\]
Similarly, we get
\[
\iota^*_{\tau \sigma}(\widehat{\eta}^\phi) = u^{b_+(X)^\sigma} (u+v)^{a'} \nu^\phi_R
\]
where $\nu^\phi_R \in H^*( \mathbb{RP}(V^{\tau \sigma} ; \mathbb{Z}_2) \cong \mathbb{Z}_2[v]/(v^a)$ is defined by 
\[
(f^{\tau \sigma})^*( \tau^\phi_{(V')^{\tau \sigma} , (U')^{\tau \sigma}} ) = \nu^\phi_R \delta \tau_{U^{\sigma}}.
\]
To complete the result, it remains to show that $\nu^\phi_R = \eta^\phi_R$. However this follows immediately from the following $\mathbb{Z}_2 = \langle \tau \rangle$-equivariant commutative diagram
\[
\xymatrix{
S^{V^\sigma , U^\sigma} \ar[r]^-{f^\sigma} & S^{(V')^\sigma , (U')^\sigma} \\
S^{V^{\tau\sigma} , U^\sigma} \ar[u]^-{i} \ar[r]^-{f^{\tau \sigma}} & S^{(V')^{\tau \sigma} , (U')^\sigma} \ar[u]^-{i}
}
\]
where the vertical arrows are multiplication by $i \in S^1$.
\end{proof}

To apply the localisation theorem we will work in equivariant cohomology localised with respect to $u$. For instance the localised equivariant cohomology of $\mathbb{RP}(V^\sigma)$ is $u^{-1} H^*_{\mathbb{Z}_2}(\mathbb{RP}(V^\sigma) ; \mathbb{Z}_2) \cong \mathbb{Z}_2[u,u^{-1},v]/(v^a)$. In this ring $(u+v)$ is invertible, for we have
\[
(u+v)^{-1} = u^{-1} + u^{-2}v + u^{-3}v^2 + \cdots + u^{-a}v^{a-1}.
\]
Similarly $(u+v)$ is invertible in $u^{-1} H^*_{\mathbb{Z}_2}( \mathbb{RP}(V^{\tau \sigma}) ; \mathbb{Z}_2)$.

\begin{lemma}\label{lem:loc}
For any $\alpha \in H^*_{\mathbb{Z}_2}( \mathbb{RP}(V) ; \mathbb{Z}_2)$, we have
\[
\alpha = (\iota_{\sigma})_*( (u+v)^{-a} \iota_{\sigma}^*(\alpha) ) + (\iota_{\tau \sigma})_*( (u+v)^{-a} \iota_{\tau \sigma}^*(\alpha) )
\]
in the localised ring $u^{-1} H^*_{\mathbb{Z}_2}( \mathbb{RP}(V) ; \mathbb{Z}_2)$.

\end{lemma}
\begin{proof}
The localisation theorem \cite[III (3.8)]{die} implies that the restriction map
\[
\iota^*_\sigma \oplus \iota^*_{\tau \sigma} : u^{-1} H^*_{\mathbb{Z}_2}( \mathbb{RP}(V) ; \mathbb{Z}_2) \to u^{-1}H^*_{\mathbb{Z}_2}( \mathbb{RP}(V^\sigma) ; \mathbb{Z}_2) \oplus u^{-1} H^*_{\mathbb{Z}_2}(\mathbb{RP}(V^{\tau\sigma} ; \mathbb{Z}_2))
\]
is an isomorphism. Let
\[
\mu = (\iota_{\sigma})_*( (u+v)^{-a} \iota_{\sigma}^*(\alpha) ) + (\iota_{\tau \sigma})_*( (u+v)^{-a} \iota_{\tau \sigma}^*(\alpha) ) \in u^{-1} H^*_{\mathbb{Z}_2}( \mathbb{RP}(V) ; \mathbb{Z}_2).
\]
The 
\begin{align*}
\iota^*_{\sigma}(\mu) &= \iota_\sigma^* (\iota_\sigma)_* ( (u+v)^{-a} \iota_\sigma^*(\alpha) ) \\
&= (u+v)^a (u+v)^{-a} \iota_\sigma^*(\alpha) \\
&= \iota^*_\sigma(\alpha).
\end{align*}
Similarly $\iota^*_{\tau \sigma}(\mu) = \iota^*_{\tau \sigma}(\alpha)$. Hence $\mu = \alpha$ in $u^{-1} H^*_{\mathbb{Z}_2}( \mathbb{RP}(V) ; \mathbb{Z}_2)$.
\end{proof}

\begin{corollary}\label{cor:loc}
For any $\alpha \in H^*_{\mathbb{Z}_2}(\mathbb{RP}(V) ; \mathbb{Z}_2)$, we have
\[
(\pi_V)_*(\alpha) = (\pi_{V^\sigma})_*( (u+v)^{-a} \iota_{\sigma}^*(\alpha) ) + (\pi_{V^{\tau\sigma}})_*( (u+v)^{-a} \iota_{\tau \sigma}^*(\alpha) ).
\]
\end{corollary}
\begin{proof}
From Lemma \ref{lem:loc}, we have
\[
\alpha = (\iota_{\sigma})_*( (u+v)^{-a} \iota_{\sigma}^*(\alpha) ) + (\iota_{\tau \sigma})_*( (u+v)^{-a} \iota_{\tau \sigma}^*(\alpha) ).
\]
Apply $(\pi_V)_*$ to both sides and use $(\pi_V)_* (\iota_{\sigma})_* = (\pi_{V^\sigma})_*$, $(\pi_V)_* (\iota_{\tau \sigma})_* = (\pi_{V^{\tau \sigma}})_*$.
\end{proof}

Let $m \ge 0$. Choose $m_0,m_1 \ge 0$ such that $m_0 + m_1 = 2m+1$. Now take $\alpha = v^{m_0}(v+u)^{m_1} \widehat{\eta}^\phi$ in Corollary \ref{cor:loc} and use Lemma \ref{lem:restr} to obtain:
\begin{align*}
(\pi_V)_*( v^{m_0}(v+u)^{m_1} \widehat{\eta}^\phi ) &= (\pi_{V^\sigma})_*( (u+v)^{-a}\iota_{\sigma}^*( v^{m_0} (v+u)^{m_1} \widehat{\eta}^\phi ) ) \\
& \quad \quad + (\pi_{V^{\tau \sigma}})_*( (u+v)^{-a} \iota_{\tau \sigma}^*( v^{m_0} (v+u)^{m_1} \widehat{\eta}^\phi ) ) \\
&= (\pi_{V^\sigma})_*( (u+v)^{m_1-a} v^{m_0} \iota_{\sigma}^*(\widehat{\eta}^\phi) ) \\
& \quad \quad + (\pi_{V^{\tau \sigma}})_*( (u+v)^{m_0-a} v^{m_1} \iota_{\tau \sigma}^*(\widehat{\eta}^\phi) ) \\
&= u^{b_+(X)^\sigma} (\pi_{V^\sigma})_*( (u+v)^{m_1 - d} v^{m_0} \eta^\phi_R) \\
& \quad \quad + u^{b_+(X)^\sigma} (\pi_{V^{\tau \sigma}})_*( (u+v)^{m_0-d} v^{m_1} \eta^\phi_R).
\end{align*}

Consider the pushforward $(\pi_{V^{\sigma}})_*( (u+v)^{m_1 - d} v^{m_0} \eta^\phi_R)$. Observe that the binomial $(u+v)^{m_1-d}$ can be expanded as $(u+v)^{m_1-d} = \sum_{j \ge 0} \binom{m_1-d}{j} v^j u^{m_1-d-j}$. This is true regarless of whether or not $m_1-d$ is positive. Then $(\pi_{V^{\sigma}})_*( (u+v)^{m_1 - d} v^{m_0} \eta^\phi_R)$ is a sum of terms of the form $(\pi_{V^{\sigma}})_*( u^i v^j \eta^\phi_R) = u^i (\pi_{V^{\sigma}})_*( v^j \eta^\phi_R)$. However, all such terms are zero except when $j = \delta$, where $\delta = d - b_+(X)^{-\sigma}$ is the dimension of the Real Seiberg--Witten moduli space, in which case it equals $u^i SW^\phi_R(X,\mathfrak{s})$. Using this, we find that
\[
(\pi_{V^{\sigma}})_*( (u+v)^{m_1 - d} v^{m_0} \eta^\phi_R) = u^{2m+1-2d+b_+(X)^{-\sigma}} \binom{m_1 - d}{\delta - m_0} SW^\phi_R(X,\mathfrak{s})
\]
and similarly
\[
(\pi_{V^{\tau\sigma}})_*( (u+v)^{m_0 - d} v^{m_1} \eta^\phi_R) = u^{2m+1-2d+b_+(X)^{-\sigma}} \binom{m_0 - d}{\delta - m_1} SW^\phi_R(X,\mathfrak{s}).
\]
Therefore, we have
\begin{align*}
& \quad \quad (\pi_V)_*( v^{m_0}(v+u)^{m_1} \eta^\phi ) = \\
& u^{2m - (2d-b_+(X)-1)}\left( \binom{m_1-d}{\delta-m_0} + \binom{m_0-d}{\delta-m_1} \right) SW^\phi_R(X,\mathfrak{s}) \in H^0_{\mathbb{Z}_2}(pt ; \mathbb{Z}_2) \cong \mathbb{Z}_2.
\end{align*}
But the left hand side is precisely $\widehat{SW}^\phi_{m_0,m_1}(f)$, so we have shown that
\[
\widehat{SW}^\phi_{m_0,m_1}(f) = u^{2m-(2d-b+(X)-1)}\left( \binom{m_1-d}{\delta-m_0} + \binom{m_0-d}{\delta-m_1} \right) SW^\phi_R(X,\mathfrak{s}).
\]
Assume the dimension $2d - b_+(X) - 1$ of the ordinary Seiberg--Witten moduli space is even and non-negative and choose $m = (2d-b_+(X)-1)/2$. Then under the forgetful map from $\mathbb{Z}_2$-equivariant cohomology to ordinary cohomology, the left hand side reduces to $SW^\phi(X,\mathfrak{s})$ (mod $2$), the ordinary mod $2$ Seiberg--Witten invariant of $(X,\mathfrak{s})$. Thus we have proven:

\begin{theorem}\label{thm:loc}
Let $X$ be a compact, oriented, smooth $4$-manifold with $b_1(X) = 0$ and $b_+(X)$ odd. Let $\sigma$ be a Real structure on $X$ and $\mathfrak{s}$ a Real spin$^c$-structure. Assume that $b_+(X)^{-\sigma}>0$ and let $\phi$ be a chamber. Then
\begin{equation}\label{equ:swloc1}
SW^\phi(X,\mathfrak{s}) = \left( \binom{m_1-d}{\delta-m_0} + \binom{m_0-d}{\delta-m_1} \right) SW^\phi_R(X,\mathfrak{s}) \; ({\rm mod} \; 2),
\end{equation}
where $m_0,m_1 \ge 0$ and any non-negative integers such that $m_0 + m_1 = 2d - b_+(X)$ and $\delta = d - b_+(X)^{-\sigma}$.

\end{theorem}

Recall that $\binom{a}{b} = (-1)^b \binom{b-a-1}{b}$ (this follows by writing $\binom{a}{b} = a(a-1) \cdots (a-b+1)/b!$ and $\binom{b-a-1}{b} = (b-a-1)(b-a-2)\cdots (-a)/b! = (-1)^b a(a-1) \cdots (a-b+1)/b!$). Thus
\begin{align*}
\binom{m_1-d}{\delta-m_0}  = \binom{\delta - m_0 - m_1 + d - 1}{\delta-m_0} &= \binom{2d - b_+(X)^{-\sigma} - 1 -m_0 - m_1}{\delta-m_0} \\
&= \binom{ b_+(X)^{\sigma} - 1 }{\delta-m_0} \; ({\rm mod} \; 2).
\end{align*}
Hence Theorem \ref{thm:loc} can be re-written as
\[
SW^\phi(X,\mathfrak{s}) = \left( \binom{b_+(X)^{\sigma} - 1}{\delta-m_0} + \binom{ b_+(X)^{\sigma} -1}{\delta-m_1} \right) SW^\phi_R(X,\mathfrak{s}) \; ({\rm mod} \; 2).
\]

Assume that $2d - b_+(X) - 1 = 2m$ is even and non-negative (as must be the case if $SW^\phi(X,\mathfrak{s})$ is non-zero). Then $m_0 + m_1 = 2m+1$. Choose $m_0 = m$, $m_1 = m+1$. Then 
\begin{align*}
\binom{b_+(X)^{\sigma}-1}{\delta-m_0} + \binom{b_+(X)^{\sigma}-1}{\delta-m_1} &= \binom{b_+(X)^\sigma - 1}{\delta-m} + \binom{b_+(X)^\sigma - 1}{\delta-m-1} \\
& = \binom{b_+(X)^\sigma}{\delta-m} \\
& = \binom{ b_+(X)^\sigma}{ \frac{1}{2}(b_+(X)+1) - b_+(X)^{-\sigma} } \\
& = \binom{b_+(X)^\sigma}{ \frac{1}{2}( b_+(X)^{\sigma} - b_+(X)^{-\sigma} + 1) }.
\end{align*}

Next, using $\binom{a}{b} = \binom{2a}{2b} \; ({\rm mod} \; 2)$ (which follows from Lucas's theorem), we can further re-write this as
\[
\binom{ 2b_+(X)^\sigma}{ b_+(X)^{\sigma} - b_+(X)^{-\sigma} + 1} = \binom{ 2b_+(X)^\sigma }{ b_+(X)^\sigma + b_+(X)^{-\sigma} - 1 } = \binom{2b_+(X)^\sigma}{ b_+(X) - 1},
\]
giving
\begin{corollary}
Let $X$ be a compact, oriented, smooth $4$-manifold with $b_1(X) = 0$. Let $\sigma$ be a Real structure on $X$ and $\mathfrak{s}$ a Real spin$^c$-structure. Assume that $b_+(X)^{-\sigma}>0$ and let $\phi$ be a chamber. If $2d - b_+(X) - 1 \ge 0$, then
\[
SW^\phi(X,\mathfrak{s}) = \binom{ 2b_+(X)^{\sigma} }{ b_+(X)-1 } SW^\phi_R(X,\mathfrak{s}) \; ({\rm mod} \; 2).
\]
\end{corollary}
\begin{proof}
If $2d - b_+(X) - 1$ is even and non-negative, then the result follows from the calculation given above. If $2d - b_+(X) - 1$ is odd, then $b_+(X)$ must be even and $SW^\phi(X,\mathfrak{s}) = 0$ by definition. But we also have $\binom{ 2b_+(X)^{\sigma} }{b_+(X) - 1} = 0 \; ({\rm mod} \; 2)$ because $\binom{a}{b} = 0 \; ({\rm mod} \; 2)$ whenever $a$ is even and $b$ is odd.
\end{proof}

\subsection{Case $b_1(X)^{-\sigma}=0$}\label{sec:loc3}

We will now relax the condition that $b_1(X) = 0$ to the weaker condition that $b_1(X)^{-\sigma} = 0$. Under this assumption an equivariant splitting exists by Lemma \ref{lem:eqs}. Hence we have an $O(2)$-equivariant Bauer--Furuta map $f : S^{V,U} \to S^{V',U'}$ over $B = Jac^{\mathfrak{s}}(X)$. Since $b_1(X)^{-\sigma} = 0$, the action of $\sigma$ on the base is the invesion map. The fixed point set $B^\sigma$ consists of $2^{b_1(X)}$ isolated fixed points. By Proposition \ref{prop:jacfix}, these fixed points correspond to the different Real structures on $\mathfrak{s}$.

As before, we restrict the $O(2)$ action to $\mathbb{Z}_2 \times \mathbb{Z}_2$. Recall that we have defined invariants
\[
\widehat{SW}^\phi_{m_0,m_1}(f) \in H^*_{\mathbb{Z}_2}( B ; \mathbb{Z}_2) \cong H^*(B ; \mathbb{Z}_2)[u]
\]
with the property that if $m_0 + m_1 = 2m+1$, then $\widehat{SW}^\phi_{m_0,m_1}|_{u=0} = SW^\phi_m(X,\mathfrak{s})$.

Rather than compute $\widehat{SW}^\phi_{m_0,m_1}(f)$ directly, we use localisation on $B$ with respect to $\mathbb{Z}_2 = \langle \sigma \rangle$. Let $\mathfrak{s}' \in B$ be a fixed point, which as explained above can be identified with a Real structure on $\mathfrak{s}$. Let $\iota_{\mathfrak{s}'} : \{ \mathfrak{s}' \} \to B$ be the inclusion map. The localisation theorem gives
\[
\widehat{SW}^\phi_{m_0,m_1}(f) = u^{-b_1(X)} \sum_{\mathfrak{s}'} (\iota_{\mathfrak{s}'})_* ( \widehat{SW}^\phi_{m_0,m_1}( f|_{\mathfrak{s'}} ) )
\]
where 
\[
\widehat{SW}^\phi_{m_0,m_1}(f|_{\mathfrak{s}'}) = (\iota_{\mathfrak{s}'})^*( \widehat{SW}^\phi_{m_0,m_1}(f) ) \in H^*_{\mathbb{Z}_2}( \{ \mathfrak{s}' \} ; \mathbb{Z}_2 ) \cong \mathbb{Z}_2[u].
\]
Thus to compute $\widehat{SW}^\phi_{m_0,m_1}(f)$, it suffices to compute $\widehat{SW}^\phi_{m_0,m_1}( f|_{\mathfrak{s}'})$ for each Real structure $\mathfrak{s}'$. This quantity is defined exactly the same way as $\widehat{SW}^\phi_{m_0,m_1}(f)$ except using the restriction $f|_{\mathfrak{s}'}$ of $f$ to $\mathfrak{s}'$ in place of $f$. Now by essentially the same calculation as in Section \ref{sec:b1=0loc}, we find
\begin{align*}
(\pi_{V_{\mathfrak{s}'}})_*( v^{m_0}(v+u)^{m_1} \widehat{\eta}^\phi_{\mathfrak{s}'} ) &= u^{b_+(X)^\sigma} (\pi_{V^\sigma_{\mathfrak{s}'}})_*( (u+v)^{m_1 - d} v^{m_0} (\eta^\phi_R)_{\mathfrak{s}'}) \\
& \quad \quad + u^{b_+(X)^\sigma} (\pi_{V_{\mathfrak{s}'}^{\tau \sigma}})_*( (u+v)^{m_0-d} v^{m_1} (\eta^\phi_R)_{\mathfrak{s}'}),
\end{align*}
where the $\mathfrak{s}'$ subscripts denotes restriction to $\mathfrak{s}'$. As before, we expand the binomials using $(u+v)^n = \sum_{j \ge 0} \binom{n}{j} v^j u^{n-j}$ and collecting only terms of the form $v^\delta$, where $\delta = d - b_+(X)^{-\sigma}$ is the dimension of the Real moduli space. We obtain essentially the same formula as (\ref{equ:swloc1}):
\[
\widehat{SW}^\phi_{m_0,m_1}(f|_{\mathfrak{s}'}) = u^{2m-(2d-b+(X)-1)}\left( \binom{m_1-d}{\delta-m_0} + \binom{m_0-d}{\delta-m_1} \right) SW^\phi_R(X,\mathfrak{s}').
\]
Assume that $m_0 + m_1 = 2m+1$ where $m \ge 0$. Then we can choose $m_0 = m$, $m_1 = m+1$. Then similar to the calculation in Section \ref{sec:b1=0loc}, we find
\[
\binom{m_1-d}{\delta - m_0} + \binom{m_0-d}{\delta-m_1} = \binom{\delta+d-2m-1}{\delta-m} = \binom{m-d}{\delta-m} \; ({\rm mod} \; 2).
\]
Hence
\[
\widehat{SW}^\phi_{m,m+1}(f) = u^{2m-\Delta} \binom{m-d}{\delta-m} \sum_{\mathfrak{s}'} (\iota_{\mathfrak{s}'})_*(1) SW^\phi_R(X, \mathfrak{s}'),
\]
where $\Delta = 2d - b_+(X) - 1 + b_1(X)$ is the dimension of the ordinary Seiberg--Witten moduli space. This implies that the right hand side of this equality must be a polynomial in $u$ and that setting $u=0$ recovers $SW^\phi_m(X,\mathfrak{s})$. We state this as:

\begin{theorem}\label{thm:loc2}
Let $X$ be a compact, oriented, smooth $4$-manifold. Let $\sigma$ be a Real structure on $X$ with $b_1(X)^{-\sigma} = 0$ and $b_+(X)^{-\sigma} > 0$. Let $\mathfrak{s}$ be a spin$^c$-structure which admits a Real structure and let $\phi$ be a chamber. Then for each $m \ge 0$, the expression
\[
u^{2m-\Delta} \binom{m-d}{\delta-m} \sum_{\mathfrak{s}'} (\iota_{\mathfrak{s}'})_*(1) SW^\phi_R(X, \mathfrak{s}') \in H^*( Jac^{\mathfrak{s}}(X) ; \mathbb{Z}_2)[u,u^{-1}]
\]
contains no negative powers of $u$ and the $u^0$-term equals $SW^\phi_m(X,\mathfrak{s})$.
\end{theorem}

In the special case that $\delta = 0$, the right hand side in Theorem \ref{thm:loc2} is zero unless $m=0$. Hence $SW^\phi_m(X,\mathfrak{s}) = 0 \; ({\rm mod} \; 2)$ for all $m > 0$ and
\[
SW^\phi_0(X,\mathfrak{s}) = \left.\left( u^{-\Delta} \sum_{\mathfrak{s}'} (\iota_{\mathfrak{s}'})_*(1) SW^\phi_R(X, \mathfrak{s}') \right)\right|_{u=0} \; ({\rm mod} \; 2).
\]
In particular, if $\delta = \Delta = 0$, then
\[
SW^\phi(X,\mathfrak{s}) = \sum_{\mathfrak{s}'} SW^\phi_R(X , \mathfrak{s}') \; ({\rm mod} \; 2).
\]
When $\delta= 0$ and $\Delta > 0$, the absence of negative powers of $u$ implies relations between the Real Seiberg--Witten invariants, as the following example illustrates.

\begin{example}
Suppose that $b_1(X) = 1$. Then there are two Real structures on $\mathfrak{s}$, call then $\mathfrak{s}_0, \mathfrak{s}_1$. We have $B = S^1$ where $\sigma$ acts by complex conjugation. We have $H^*_{\mathbb{Z}_2}( S^1 ; \mathbb{Z}_2) \cong \mathbb{Z}_2[u,v]$, where we set $v = (\iota_{\mathfrak{s}_0})_*(1)$. The localisation theorem implies that $(\iota_{\mathfrak{s}_1})_*(1) \neq v$, hence we must have $(\iota_{\mathfrak{s}_1})_*(1) = v+u$. Suppose $\delta = 0$. Then
\begin{align*}
\widehat{SW}^\phi_{m,m+1}(f) &= u^{-\Delta} ( SW^\phi_R(X,\mathfrak{s}_0)v + SW^\phi_R(X,\mathfrak{s})(v+u) ) \\
&= u^{-\Delta}( SW^\phi_R(X,\mathfrak{s}_0) + SW^\phi_R(X,\mathfrak{s}) )v + u^{1-\Delta} SW^\phi_R(X,\mathfrak{s}_1). \\
\end{align*}
The expression on the right hand side must not contain negative powers of $u$. If $\Delta = 0$, we get $SW^\phi(X,\mathfrak{s}) = SW^\phi_R(X,\mathfrak{s}_0) + SW^\phi_R(X,\mathfrak{s}_1) \; ({\rm mod} \; 2)$. If $\Delta = 1$, we get that $SW^\phi_R(X,\mathfrak{s}_0) = SW^\phi_R(X , \mathfrak{s}_1) = SW^\phi(X,\mathfrak{s}) \; ({\rm mod} \; 2)$ and if $\Delta > 2$, we get that $SW^\phi_R(X,\mathfrak{s}_0) = SW^\phi_R(X, \mathfrak{s}_1) = 0 \; ({\rm mod} \; 2)$.
\end{example}

\section{Connected sum formula}\label{sec:csf}

In this section we prove a formula for the Real Seiberg--Witten invariants of connected sums. Unlike the ordinary Seiberg--Witten invariant which vanish on a connected sum $X_1 \# X_2$ with $b_+(X_1)$, $b_+(X_2)$ both positive, we will see that the Real Seiberg--Witten invariants do not have this property. While the mod $2$ invariants do vanish if $b_+(X_1)^{-\sigma}$ and $b_+(X_2)^{-\sigma}$ are both positive, the integer invariants satisfy a strong non-vanishing property: if the integer-valued Real Seiberg--Witten of $X_1$ and $X_2$ are non-zero, then so is the integer-valued Real Seiberg--Witten invariant of $X_1 \# X_2$.

Let $X_1,X_2$ be compact oriented, smooth $4$-manifolds equipped with Real structures $\sigma_1,\sigma_2$. Assume that the fixed point sets of $X_1,X_2$ are non-empty. Suppose also that $X_1,X_2$ admit at least one Real spin$^c$-structure. Then every component of the fixed point sets of $\sigma_1,\sigma_2$ have codimension $2$. Therefore we can form the equivariant connected sum $X_1 \# X_2$ by removing open balls around fixed points $x_1 \in X_1$, $x_2 \in X_2$ and identifying their boundaries. Let $\sigma$ denote the resulting involution on $X = X_1 \# X_2$. In general the isomorphism class of $\sigma$ could depend on the choice of fixed points $x_1,x_2$ used in the connected sum, but we will not indicate this dependence in our notation. Now let $\mathfrak{s}_1, \mathfrak{s}_2$ be Real spin$^c$-structures on $X_1,X_2$. Removing open balls in $X_1,X_2$ around $x_1,x_2$, the resulting $4$-manifolds $X'_1,X'_2$ have boundary $S^3$ with involution $\sigma = diag(1,1,-1,-1)$, where we think of $S^3$ as the unit sphere in $\mathbb{R}^4$. In order to glue the Real spin$^c$-structures $\mathfrak{s}_1, \mathfrak{s}_2$ together, we need an isomorphism $\varphi : P_1|_{S^3} \to P_2|_{S^3}$, where $P_1,P_2$ are the principal $Spin^c(4)$-bundles corresponding to $\mathfrak{s}_1,\mathfrak{s}_2$. Since $H^2_{\mathbb{Z}_2}(S^2 ; \mathbb{Z}_-) = 0$, there is up to isomorphism a unique Real spin$^c$-structure on $S^3$, hence such an isomorphism $\varphi$ exists. The isomorphism is unique up to an element of the Real gauge group $\{ h : S^3 \to S^1 \; | \sigma^*(h) = h^{-1} \}$. This group has two components, which are represented by the constant gauge transformations $\{ \pm 1\}$. Since the constant gauge transformations $\{ \pm 1\}$ extend as constant gauge transformations over $X_1$ (or $X_2$), the isomorphism class of the glued together Real spin$^c$-structure $\mathfrak{s}_1 |_{X'_1} \cup_{\varphi} \mathfrak{s}_2 |_{X'_2}$ does not depend on the choice of $\varphi$. We denote it by $\mathfrak{s}_1 \# \mathfrak{s}_2$.

Let $d_i = (c(\mathfrak{s}_i)^2 - \sigma(X_i))/8$ and $d = ( c(\mathfrak{s})^2 - \sigma(X))/8$ so that $d = d_1 + d_2$.

First we consider the mod $2$ Real Seiberg--Witten invariants. Recall that the invariants depend on a choice of splitting. As in Section \ref{sec:realswi}, a component of the fixed point set of $\sigma$ defines such a splitting. We will choose the splittings for $X_1,X_2$ determined by the component of the fixed point set along which we perform the equivariant connected sum. This also determines a component of the fixed point set on $X_1 \# X_2$, hence a splitting for $X_1 \# X_2$. Throughout this section we will always use the splittings for $X_1,X_2$ and $X_1 \# X_2$ determined in this way. Bauer's gluing theorem \cite{b2} can easily be adapted to the Real case and it implies that the Real Bauer--Furuta invariant for $X_1 \# X_2$ is the external smash product of the Real Bauer--Furuta invariants of $X_1$ and $X_2$ (using the splittings described above).

\begin{theorem}\label{thm:mod2glue}
Assume that $b_+(X_1)^{-\sigma} > 0$ and let $\phi$ be a chamber for $(X_1 , \sigma_1)$. Then $\phi$ also defines a chamber for $(X_1 \# X_2 , \sigma )$.
\begin{itemize}
\item[(1)]{If $b_+(X_2)^{-\sigma} > 0$, then $SW^\phi_{R,m}(X_1 \# X_2 , \mathfrak{s}_1 \# \mathfrak{s}_2) = 0$ for all $m \ge 0$.}
\item[(2)]{If $b_+(X_2)^{-\sigma} = 0$, then
\[
SW^\phi_{R,m}(X_1 \# X_2 , \mathfrak{s}_1 \# \mathfrak{s}_2) = \sum_{k \ge 0} w_k( -D_R(X_2,\mathfrak{s}_2) ) SW^\phi_{R,m-d_2-k}(X_1 , \mathfrak{s}_1)
\]
for all $m \ge 0$. In particular, if $b_1(X_2)^{-\sigma} = 0$, then
\[
SW^\phi_{R,m}(X_1 \# X_2 , \mathfrak{s}_1 \# \mathfrak{s}_2) = SW^\phi_{R,m-d_2}(X_1 , \mathfrak{s}_1).
\]
}
\end{itemize}

\end{theorem}
\begin{proof}
As explained above, the Real Bauer--Furuta invariant for $X_1 \# X_2$ is the external smash product of the Real Bauer--Furuta invariants for $X_1$ and $X_2$. The result then follows by a similar computation to the one give in \cite[\textsection 8]{bar2}.
\end{proof}

Now we consider integer invariants. We assume that $d$ is even and that $w_1(D_R) = 0$. Recall that there are two types of integer invariants that we consider. The first is the degree
\[
deg_R(X,\mathfrak{s}) \in H^{b_+(X)^{-\sigma}-d}(Jac_R(X) ; \mathbb{Z}).
\]
If $b_+(X)^{-\sigma} > 0$, then we also have the integer-valued Real Seiberg--Witten invariant
\[
\mathbf{SW}_{R,\mathbb{Z}}(X,\mathfrak{s}) \in H^{b_+(X)^{-\sigma}-d}( Jac_R(X) ; \mathbb{Z})
\]
and the two invariants are related by $deg_R(X,\mathfrak{s}) = 2 \, \mathbf{SW}_{R,\mathbb{Z}}(X,\mathfrak{s})$. Note that if $X = X_1 \# X_2$ is an equivariant connected sum, then as explained above this imposes distinguished choices of splittings for $X_1,X_2$ and $X$. With these splittings $w_1(D_R(X,\mathfrak{s})) = 0$ if and only if $w_1(D_R(X_i , \mathfrak{s}_i)) = 0$ for $i=1,2$. The behaviour of the integer invariants under connected sum is given by the following results.
\begin{theorem}\label{thm:glue1}
Suppose $d_1,d_2$ are even and $w_1( D_R(X,\mathfrak{s})) = 0$. Then
\[
deg_R(X_1 \# X_2 , \mathfrak{s}_1 \# \mathfrak{s}_2) = deg_R(X_1 ,\mathfrak{s}_1) \smallsmile deg_R(X_2 , \mathfrak{s}_2)
\]
where the right hand side is understood as an external cup product
\[
H^i( Jac_R(X_1) ; \mathbb{Z}) \times H^j( Jac_R(X_2) ; \mathbb{Z}) \to H^{i+j}(Jac_R(X_1) \times Jac_R(X_2) ; \mathbb{Z})
\]
followed by the isomorphism $Jac_R(X_1) \times Jac_R(X_2) \cong Jac_R(X_1 \# X_2)$.
\end{theorem}
\begin{proof}
Immediate since the Real Bauer--Furuta invariant of $X_1 \# X_2$ is the external smash product of the Real Bauer--Furuta invariants for $X_1$ and $X_2$.
\end{proof}

\begin{theorem}\label{thm:glue2}
Suppose that $b_+(X_1)^{-\sigma} > 0$ and $w_1(D_R(X,\mathfrak{s})) = 0$.

\begin{itemize}
\item[(1)]{If $d_1,d_2$ are even, then
\[
\mathbf{SW}_{R,\mathbb{Z}}(X_1 \# X_2 , \mathfrak{s}_1 \# \mathfrak{s}_2) = \mathbf{SW}_{R,\mathbb{Z}}(X_1 , \mathfrak{s}_1) \smallsmile deg_R(X_2 , \mathfrak{s}_2).
\]
In particular, if $b_+(X_1)^{-\sigma}, b_+(X_2)^{-\sigma} > 0$, then
\[
\mathbf{SW}_{R,\mathbb{Z}}(X_1 \# X_2 , \mathfrak{s}_1 \# \mathfrak{s}_2) = 2 \, \mathbf{SW}_{R,\mathbb{Z}}(X_1 , \mathfrak{s}_1) \smallsmile \mathbf{SW}_{R,\mathbb{Z}}(X_2, \mathfrak{s}_2).
\]
}
\item[(2)]{If $d_1,d_2$ are odd, then
\[
\mathbf{SW}_{R,\mathbb{Z}}(X_1 \# X_2 , \mathfrak{s}_1 \# \mathfrak{s}_2) = 0.
\]
}
\end{itemize}
\end{theorem}
\begin{proof}
(1). Given that $deg_R(X_1 , \mathfrak{s}_1) = 2 \, \mathbf{SW}_{R,\mathbb{Z}}(X,\mathfrak{s}_1)$ and $deg_R(X_1 \# X_2 , \mathfrak{s}_1 \# \mathfrak{s}_2) = 2 \, \mathbf{SW}_{R,\mathbb{Z}}(X_1 \# X_2,\mathfrak{s}_1 \# \mathfrak{s}_2)$, the result is immediate from Theorem \ref{thm:glue1}.

(2). We have $deg_R(X_1 \# X_2 , \mathfrak{s}_1 \# \mathfrak{s}_2) = deg_R(X_1 ,\mathfrak{s}_1) \smallsmile deg_R(X_2 , \mathfrak{s}_2)$. But since $d_1$ is odd, we have $deg_R(X_1 ,\mathfrak{s}_1) = 0$ and the result follows.
\end{proof}

\begin{remark}
An interesting consequence of Theorem \ref{thm:glue2} is that if $(X,\mathfrak{s})$ is the equivariant sum of $r$ summands $(X_i , \mathfrak{s}_i)$, $i=1,\dots , r$, each with $b_+(X_i)^{-\sigma} > 0$ and $d_i$ even, then $\mathbf{SW}_{R,\mathbb{Z}}(X,\mathfrak{s})$ is divisible by $2^{r-1}$.
\end{remark}

Let $X$ be a compact, oriented, smooth $4$-manifold and $\mathfrak{s}$ a spin$^c$-structure on $X$. Consider the connected sum $X \# X$ of $X$ with itself with an involution that swaps the two copies of $X$. We can construct such an involution as follows. First start with $S^4$ the unit sphere in $\mathbb{R}^5$ with involution $\sigma_0 = diag(1,1,1,-1,-1)$. This is an odd involution with respect to the unique spin structure $\mathfrak{s}_0$ on $S^4$, hence we obtain a Real structure on $\mathfrak{s}_0$. Now take a point $x_1 \in S^4$ which is not fixed by $\sigma_0$ and set $x_2 = \sigma(x_1)$. Attach a copy of $X$ at $x_1$ and another copy of $X$ at $x_2$. This gives an involution $\sigma$ on $S^4 \# X \# X = X \# X$. Furthermore the spin$^c$-structure $\mathfrak{s}_0 \# \mathfrak{s} \# (-\mathfrak{s})$ inherits a Real structure. Set $d = (c(\mathfrak{s})^2 - \sigma(X))/8$ and observe that $b_1(X \# X)^{\pm \sigma} = b_1(X)$, $b_+(X \# X)^{\pm \sigma} = b_+(X)$. If $b_+(X) = 1$, then $b_+(X)^{-\sigma} = 1$ and a chamber for $X$ defines a chamber for $(X\# X , \sigma)$ in a natural way.

Recall that the ordinary mod $2$ Seiberg--Witten invariants of $(X,\mathfrak{s})$ can be viewed as a map
\[
\mathbf{SW}^\phi(X , \mathfrak{s}) : H^*_{S^1}(pt ; \mathbb{Z}_2) \to H^{* - (2d-b_+(X)-1)}(Jac(X) ; \mathbb{Z}_2),
\]
or as a collection of cohomology classes
\[
SW^\phi_m(X,\mathfrak{s}) = \mathbf{SW}^\phi(X,\mathfrak{s})(U^m) \in H^{2m-(2d-b_+(X)-1)}(Jac(X) ; \mathbb{Z}_2).
\]

\begin{proposition}\label{prop:XX}

Suppose that $b_+(X) > 0$ and if $b_+(X) = 1$, let $\phi$ denote a chamber for $X$. For all $m \ge 0$, we have that $SW^\phi_{R,2m}(X \# X , \mathfrak{s} \# (-\mathfrak{s}) ) = 0$ and
\[
SW^\phi_{R,2m+1}(X \# X , \mathfrak{s} \# (-\mathfrak{s}) ) = SW_m^\phi(X,\mathfrak{s})
\]
under the identification $Jac_R(X \# X) \cong Jac(X)$.

We also have
\[
deg_R( X \# X , \mathfrak{s} \# (-\mathfrak{s}) ) = \begin{cases} 1 & \text{if } b_+(X) = d = 0, \\ 0 & \text{otherwise}. \end{cases}
\]

\end{proposition}
\begin{proof}
Let $f : S^{V,U} \to S^{V',U'}$ denote the Bauer--Furuta invariant of $(X , \mathfrak{s})$. Then charge conjugation identifies the Bauer--Furuta invariant of $(X, -\mathfrak{s})$ with $f$, but in an anti-linear way. In other words, if we let $\bar{f}$ denote $f : S^{V,U} \to S^{V',U'}$ but where the $S^1$-action on the domain and codomain are complex conjugated, then $\bar{f}$ is the Bauer--Furuta invariant of $(X , -\mathfrak{s})$. Then the $O(2)$-equivariant Bauer--Furuta map of $(X \# X , \mathfrak{s} \# (-\mathfrak{s}) )$ is $f \wedge \bar{f} : S^{V,U} \wedge S^{V,U} \to S^{V',U'} \wedge S^{V',U'}$, where $\sigma$ acts on the domain and codomain by swapping the two factors of the smash product. The Real Bauer--Furuta map of $(X \# X , \mathfrak{s} \# (-\mathfrak{s}))$ is then given by restricting $f \wedge \bar{f}$ to the fixed points of $\sigma$. Clearly this just gives back $f$, but where the $S^1$-action is restricted to the subgroup $\mathbb{Z}_2 \subset S^1$. Then the calculation given in Section \ref{sec:z2restr} implies the result.

For the calculation of the degree, note that the real part of the families index of the the spin$^c$ Dirac operator on $X \# X$ parametrised by $Jac_R^{\mathfrak{s} \# -\mathfrak{s}}(X \# X)$ can be identified with $D$, the families index of the usual family of spin$^c$-Dirac operators on $X$ parametrised by $Jac_R^{\mathfrak{s}}(X)$. Moreover, $D = V - V'$. Then $w_1(D) = 0$ because $D$ admits a complex structure. So the invariant $deg_R(X \# X , \mathfrak{s} \# -\mathfrak{s})$ is defined and equals the ordinary degree of $f : S^{V,U} \to S^{V',U'}$. Denote by $\beta_{S^1} \in H^*_{S^1}( Jac(X) ; \mathbb{Z})$ the $S^1$-equivariant degree of $f$. Since $S^1$ acts trivially on $B = Jac(X)$, we have $H^*_{S^1}(B ; \mathbb{Z}) \cong H^*(B ; \mathbb{Z})[U]$, where $U$ is the generator of $H^2_{S^1}(pt ; \mathbb{Z})$. Then from \cite[\textsection 3]{bar}, we have that $\beta_{S^1} = 0$ if $b_+(X) > 0$ and
\[
\beta_{S^1} = U^{-d} + U^{-d-1} c_1(-D) + U^{-d-2} c_2(-D) + \cdots
\]
if $b_+(X) = 0$. In this case, the non-equivariant degree of $f$ is given by setting $U=0$, giving $deg(f) = c_{-d}(-D)$ (note that $b_+(X) = 0$ implies that $d \le 0$ by Donaldson's theorem). If $d=0$, then $deg(f) = c_0(-D) = 1$ as claimed. To finish, we will show that if $b_+(X) = 0$, then $c_j(-D) = 0$ for all $j > 0$.

Since $H^*(B ; \mathbb{Z})$ has no torsion, it suffices to compute Chern classes in rational cohomology which we can do using the families index theorem. First note that since $b_+(X) = 0$, $H^2(X ; \mathbb{Q})$ is negative definite. Now let $a,b \in H^1(X ; \mathbb{Q})$. Then $(a \smallsmile b )^2 = a \smallsmile b \smallsmile a \smallsmile b = 0$. But $H^2(X ; \mathbb{Q})$ is negative definite and so $a \smallsmile b = 0$. So all cup products of degree $1$ classes vanish rationally. Using the same method as \cite[\textsection 5.3]{bk}, the families index theorem gives $c_j(D) = 0$ for all $j > 0$. Since $c(-D) = c(D)^{-1}=1$, we also have $c_j(-D) = 0$ for all $j >0$ as claimed.
\end{proof}

\section{Generalised fibre sum}\label{sec:gfs}

Let $X_1,X_2$ be compact, oriented, smooth $4$-manifolds and let $\sigma_1,\sigma_2$ be Real structures on $X_1,X_2$. Suppose for $i=1,2$, that the fixed point set of $\sigma_i$ contains an embedded torus $T_i \subset X_i$ of self-intersection zero. Let $X'_i$ denote the complement of a $\sigma_i$-invariant tubular neighbourhood of $T_i$ and let $\sigma'_i$ be the restriction of $\sigma_i$ to $X'_i$. Then $X'_i$ has boundary $T^3 = (S^1)^3$, with involution $\sigma_{T^3}(x,y,z) = (-x,y,z)$ (where we think of $S^1$ as the unit circle in $\mathbb{C}$). Let $\varphi$ be an orientation reversing diffeomorphism $\varphi : \partial X'_1 \to \partial X'_2$ satisfying $\varphi \circ \sigma'_1 = \sigma'_2 \circ \varphi$. We obtain a $4$-manifold $X = X_1 \#_\varphi X_2$ is the $4$-manifold obtained by attaching $X'_1$ and $X'_2$ along their boundaries using the map $\varphi$. Clearly $\sigma'_1, \sigma'_2$ patch together to form a Real structure $\sigma$ on $X$. In the case that $\varphi : \partial X'_1 \to \partial X'_2$ is induced by a diffeomorphism $T_1 \to T_2$, together with trivialisations $\nu T_1 \cong D^2 \times T_1$, $\nu T_2 \cong D^2 \times T_2$ of the normal bundles, then $X$ is called a {\em generalised fibre sum} \cite[Definition 7.1.11]{gs}. 

Let $\mathfrak{s}_1, \mathfrak{s}_2$ denote Real spin$^c$-structures on $X_1,X_2$ and $\mathfrak{s}'_1, \mathfrak{s}'_2$ their restrictions to $X'_1, X'_2$. If $\varphi^*(\mathfrak{s}'_2 |_{T^3}) \cong \mathfrak{s}'_1 |_{T^3}$, then by choosing an isomorphism we can glue $\mathfrak{s}'_1$ and $\mathfrak{s}'_2$ together to form a Real spin$^c$-structure on $X$. Since $b_1(T^3)^{-\sigma_{T^3}} = 0$, the Real gauge group on $T^3$ has two components and so using the same argument as in the connected sum case, we see that the resulting spin$^c$-structure does not depend on the choice of isomorphism. We denote it by $\mathfrak{s}_1 \#_f \mathfrak{s}_2$.

It is easily seen that each Real spin$^c$-structure on $(T^3 , \sigma_{T^3})$ comes from a uniquely determined odd spin structure. Let $\mathfrak{s}_{T^3}$ denote the translation invariant spin structure on $T^3$ and let $l_{\epsilon_1 , \epsilon_2, \epsilon_3}$ denote the real line bundle on $T^3$ with holonomy $(-1)^{\epsilon_i}$ around the $i$-th circle. There are four odd spin structures given by $l_{1 , \epsilon_2, \epsilon_3} \otimes \mathfrak{s}_{T_3}$, $\epsilon_2,\epsilon_3 = 0,1$. The set of $\sigma_{T^3}$-equivariant, orientation reversing diffeomorphisms $\varphi : T^3 \to T^3$ acts transitively on the set of odd spin structures. To see this, it suffices to consider diffeomorphisms of the form $\varphi(x,y,z) = (\overline{x}y^a z^b  , y , z)$, for $a,b = 0,1$. Hence for any $\mathfrak{s}_1,\mathfrak{s}_2$, there always exists a suitable choice of $\varphi$ that can be used to glue $\mathfrak{s}_1, \mathfrak{s}_2$. In general, the $4$-manifold $X_1 \#_\varphi X_2$ depends on the choice of $\varphi$, however if $X_1$ is a $\mathcal{C}^\infty$-elliptic fibration, $T_1$ a smooth fibre and $X_1$ has a cusp fibre, then the diffeomorphism type of the generalised fibre sum of $(X_1 , T_1)$ and $(X_2 , T_2)$ does not depend on $\varphi$ \cite[Lemma 8.3.6]{gs}.

Set $d_i = (c(\mathfrak{s}_i)^2 - \sigma(X_i) )/8$. Observe that since $H^i( T^3 ; \mathbb{Z} )^{-\sigma_{T^3}} = 0$ for $i=1,2$, then $H^1( X ; \mathbb{Z} ) ^{-\sigma} = 0$ and $H^2(X ; \mathbb{R})^{-\sigma} \cong H^2(X_1 ; \mathbb{R})^{-\sigma_1} \oplus H^2(X_2 ; \mathbb{R})^{-\sigma_2}$. Hence $c(\mathfrak{s})^2 = c(\mathfrak{s}_1)^2 + c(\mathfrak{s}_2)^2$. Also $\sigma(X_1 \#_\varphi X_2 ) = \sigma(X_1) + \sigma(X_2)$ by Novikov additivity \cite[Remark 9.1.7]{gs}. Thus $d = d_1 + d_2$, where $d = (c(\mathfrak{s})^2 - \sigma(X))/8$. In particular, if $d_1,d_2$ are even then so is $d$.

\begin{theorem}\label{thm:fsum}
Suppose $b_1(X_1)^{-\sigma_1} = b_1(X_2)^{-\sigma_2} = 0$ and $d_1, d_2$ are even. Then
\[
deg_R(X_1 \#_\varphi X_2 , \mathfrak{s}_1 \#_\varphi \mathfrak{s}_2) = deg_R(X_1 ,\mathfrak{s}_1) deg_R(X_2 , \mathfrak{s}_2).
\]
In particular, if $b_+(X_i)^{-\sigma_i} > 0$ for $i=1,2$, then
\[
SW_{R,\mathbb{Z}}( X_1 \#_\varphi X_2 , \mathfrak{s}_1 \#_\varphi \mathfrak{s}_2 ) = 2 \, SW_{R,\mathbb{Z}}(X_1 , \mathfrak{s}_1) SW_{R,\mathbb{Z}}(X_2 ,\mathfrak{s}_2).
\]
\end{theorem}
\begin{proof}
The proof is similar to the connected sum case. The only difference is that since we are gluing manifolds $X'_1,X'_2$ with boundary $T^3$, we need to use relative Bauer--Furuta maps. However it turns out that the Real Floer homotopy type of $(T^3 , \mathfrak{s} , \sigma_{T^3})$ for any Real spin$^c$-structure on $T^3$ is equal to $S^0$ \cite[Proposition 4.24]{mi}. Hence if $(W , \sigma_W)$ is any compact, oriented, smooth Real $4$-manifold with $b_1(W)^{-\sigma_W} = 0$ and with boundary $(T^3 , \sigma_{T^3})$ and $\mathfrak{s}_W$ is any Real spin$^c$-structure on $W$, then we can define $deg_R(W , \mathfrak{s}_W)$ to be the degree of the Real Bauer--Furuta map of $(W , \mathfrak{s}_W)$ exactly as in the boundaryless case. Then if $W = W_1 \cup_{T^3} W_2$ is a closed $4$-manifold obtained by gluing together two such $4$-manifolds, it follows that the degree is multiplicative: $deg_R(W , \mathfrak{s}_W ) = deg_R(W_1 , \mathfrak{s}_{W_1}) deg_R(W_2 , \mathfrak{s}_{W_2})$. If we apply this to $X_1 = X' \cup_{T^3} (D^2 \times T^2)$, we get
\[
deg_R(X_1 , \mathfrak{s}_1) = deg_R(X'_1 , \mathfrak{s}'_1) deg_R( D^2 \times T^2 , \mathfrak{s}_1 |_{D^2 \times T^2}).
\]
We claim that $deg_R( D^2 \times T^2 , \mathfrak{s}_1 |_{D^2 \times T^2}) = \pm 1$. To see this, consider $X_0 = S^2 \times T^2$ with involution $\sigma_0 = \sigma_{S^2} \times id_{T^2}$, where $\sigma_{S^2}$ is a rotation of $S^2$ by $\pi$. Clearly $S^2 \times T^2$ is obtained by gluing two copies of $D^2 \times T^2$ together. Note that all four Real spin$^c$-structures on $T^3$ extend to $D^2 \times T^2$ and to $S^2 \times T^2$. We have $b_1(X_0)^{-\sigma} = b_+(X_0)^{-\sigma} = \sigma(X_0) = 0$ and $X_0$ has positive scalar curvature, hence $deg_R( X_0 , \mathfrak{s}_0) = 1$ for any Real spin$^c$-structure $\mathfrak{s}_0$. This is only possible if $deg_R(D^2 \times T^2 , \mathfrak{s}_0 |_{D^2 \times T^2} ) = \pm 1$. Therefore, $deg_R(X_1 , \mathfrak{s}_1) = \pm deg_R(X'_1 , \mathfrak{s}'_1)$ and similarly $deg_R(X_2 , \mathfrak{s}_2) = deg_R(X'_2 , \mathfrak{s}'_2)$. This gives
\begin{align*}
deg_R( X_1 \#_\varphi X_2 , \mathfrak{s}_1 \#_\varphi \mathfrak{s}_2 ) &= deg_R(X'_1 , \mathfrak{s}'_1) deg_R(X'_2 , \mathfrak{s}'_2) \\
&= deg_R(X_1 , \mathfrak{s}_1) deg_R(X_2 , \mathfrak{s}_2).
\end{align*}

\end{proof}

\begin{example}
Start with an elliptic fibration $\pi : E(1) \to \mathbb{CP}^1$. Let $p : \mathbb{CP}^1 \to \mathbb{CP}^1$ be a branched double cover branched along two points $x_1,x_2 \in \mathbb{CP}^1$ which are regular values of $\pi$. The pullback of $E(1)$ along $p$ gives an elliptic fibration whose total space $X$ is $E(2)$, a K3-surface. Let $\sigma : X \to X$ be the covering involution. Then the fixed point set of $\sigma$ is the preimages of the fibres of $E(1)$ over $x_1,x_2$. Hence the fixed point set of $\sigma$ consists of two tori. Let $\mathfrak{s}$ be the spin structure on $X$. Our construction gives a complex structure on $X$ for which $\sigma$ is holomorphic. However, we can use a hyperK\"ahler structure on $X$ to find a complex structure for which $\sigma$ is anti-holomorphic (let $I$ denote the complex structure on $X$. Choose a K\"ahler metric $g$ for $(X,I)$, which by averaging may be assumed to be $\sigma$-invariant. Let $\omega_{g,I}$ denote the K\"ahler metric associated to $g$ and $I$. Then by Yau's theorem there exists a unique Ricci flat K\"ahler metric $g'$ for $(X,I)$ such that $[\omega_{g',I}] = [\omega_{g,I}]$. By uniqueness of $g'$ and $\sigma$-invariance of the K\"ahler class $[\omega_{g,I}]$, it follows that $g'$ is $\sigma$-invariant. Thus $g'$ is a $\sigma$-invariant hyperK\"ahler metric for $X$. Let $S$ denote the $2$-sphere of hyperK\"ahler metrics for $g'$. Note that $S$ corresponds to the unit sphere of $H^+(X)_{g'}$, the space of $g'$-self dual harmonic $2$-forms on $X$ via the map which sends a complex structure $J$ to the corresponding K\"ahler form $\omega_{g',J}$. By construction the action of $\sigma$ on $X$ has non-isolated fixed points. Thus $\sigma$ is an odd involution and therefore acts non-trivially on $H^+(X)_{g'}$. If follows that there exists a $J \in S$ such that $\sigma(J) = -J$. With respect to such a $J$, $\sigma$ is anti-holomorphic). Then the corresponding K\"ahler $2$-form $\omega$ satisfies $\sigma^*(\omega) = -\omega$. The moduli space of the Seiberg--Witten equations for $(X, \mathfrak{s})$ with perturbation of the form $i \lambda \omega$ for sufficiently large $\lambda$ consists of a single point cut out transversally. This unique solution is preserved by $\sigma$ and hence $SW_{R,\mathbb{Z}}(X , \mathfrak{s}) = 1$. Now for any $n \ge 1$, consider the $n$-fold fibre sum of $(X , \sigma , \mathfrak{s})$. The resulting $4$-manifold is the $n$-fold fibre sum of $E(2)$, namely $E(2n)$. In fact, it is not hard to see that the resulting involution on $E(2n)$ comes from taking an elliptic fibration $E(n) \to \mathbb{CP}^1$ and pulling back under the branched double cover $p : \mathbb{CP}^1 \to \mathbb{CP}^1$. The $n$-fold fibre sum of $\mathfrak{s}$ is the unique spin structure $\mathfrak{s}_{E(2n)}$ on $E(2n)$ and Theorem \ref{thm:fsum} gives $SW_{R , \mathbb{Z}}( E(2n) , \mathfrak{s}_{E(2n)} ) = \pm 2^{n-1}$. By way of contrast, the ordinary Seiberg--Witten invariant is given by $SW(E(2n) , \mathfrak{s}_{E(2n)}) = \pm \binom{2n-2}{n-1}$.

\end{example}

\section{Exotic involutions and exotic embedded surfaces}\label{sec:exotic}

Recall from the introduction the notion of an admissible pair $(X,\sigma)$:

\begin{definition}\label{def:admissible}
Let $X$ be a compact, oriented, smooth $4$-manifold and $\sigma$ an orientation preserving smooth involution on $X$. We will say the pair $(X , \sigma)$ is {\em admissible} if $b_1(X)^{-\sigma} = 0$, $\sigma$ has a non-isolated fixed point and $(X,\sigma)$ satisfies one of the following conditions:
\begin{itemize}
\item[(1)]{$X$ admits a symplectic structure with $\sigma^*(\omega) = -\omega$, and $b_+(X) - b_1(X) = 3 \; ({\rm mod} \; 4)$.}
\item[(2)]{$X$ has a spin structure preserved by $\sigma$ and $b_1(X)^{-\sigma} = \sigma(X) = 0$.}
\item[(3)]{$X$ has a spin$^c$-structure $\mathfrak{s}$ with $\sigma^*(\mathfrak{s}) = -\mathfrak{s}$, $SW(X,\mathfrak{s})$ is odd, $b_+(X) - b_1(X) = 3 \; ({\rm mod} \; 4)$, and
\[
\frac{ c(\mathfrak{s})^2 - \sigma(X) }{8} = \frac{ b_+(X) - b_1(X) + 1}{2} = b_+(X)^{-\sigma}.
\]
}
\item[(4)]{$X = \overline{\mathbb{CP}}^2$ with an involution such that $b_+(X)^{-\sigma} = 0$.}
\item[(5)]{$X = N \# N$ and $\sigma$ is the involution which swaps the two factors (as in Section \ref{sec:csf}), where $N$ is negative definite, $b_1(N) = 0$ and there is a spin$^c$-structure on $N$ with $c(\mathfrak{s})^2 = -b_2(N)$.}
\end{itemize}

\end{definition}

\begin{proposition}\label{prop:nzdeg}
Let $(X,\sigma)$ be admissible. Then $X$ has a Real spin$^c$-structure $\mathfrak{s}$ such that $deg_R(X,\mathfrak{s})$ is defined and non-zero.
\end{proposition}
\begin{proof}
We consider separately cases (1)-(5) of Definition \ref{def:admissible}.

(1). Let $\mathfrak{s}$ be the canonical spin$^c$-structure associated to $\omega$. Then $\sigma^*(\mathfrak{s}) \cong -\mathfrak{s}$, hence $\mathfrak{s}$ admits a Real structure by Proposition \ref{prop:realspinc}. Let $S \subset X$ denote the fixed point set of $\sigma$. Let $X_0 = X/\sigma$ and let $S \subset X_0$ be the image of $S$ in $X_0$. Then $X \to X_0$ is a branched double cover, branched over $S$. Since $\sigma^*(\omega) = -\omega$, $S$ is Lagrangian. By the Lagrangian neighbourhood theorem, the normal bundle of $S$ in $X$ coincides with the normal bundle of the zero section $S \to T^*S$. Hence $e(S) = -\chi(S)$, where $e(S)$ denotes the self-intersection number of $S$ (note that $S$ could be non-orientable). Set $d = (c(\mathfrak{s})^2 - \sigma(X))/8$. Since $\mathfrak{s}$ is the canonical spin$^c$-structure associated to $\omega$, we have $2d = b_+(X)-b_1(X) + 1$. Since $X \to X_0$ is a branched double cover, we have $\chi(X_0) = (\chi(X) + \chi(S))/2$ and $\sigma(X_0) = (\sigma(X) + e(S))/2 = (\sigma(X) - \chi(S))/2$. Hence $\sigma(X_0) + \chi(X_0) = (\sigma(X) + \chi(X))/2$, or equivalently
\[
b_+(X) - b_1(X) + 1 = 2( b_+(X_0) - b_1(X_0) + 1) = 2( b_+(X)^{\sigma} - b_1(X)^{\sigma} + 1).
\]
Therefore,
\[
b_+(X)^{-\sigma} - b_1(X)^{-\sigma} = b_+(X)^\sigma - b_1(X)^\sigma + 1 = \frac{1}{2}( b_+(X) - b_1(X) + 1) = d.
\]
Since by assumption $b_1(X)^{-\sigma} = 0$, this reduces to $b_+(X)^{-\sigma} = d$. We also have that $SW(X,\mathfrak{s})$ is odd, because $X$ is symplectic. Thus case (1) is a special case of case (3).

(2). By assumption $X$ has a spin structure $\mathfrak{s}$ whose isomorphism class is preserved by $\sigma$. Since $\sigma$ has a non-isolated fixed point, $\sigma$ is odd with respect to $\mathfrak{s}$, thus the underlying spin$^c$-structure of $\mathfrak{s}$ admits a Real structure. Then $d = (c(\mathfrak{s})^2 - \sigma(X))/8 = -\sigma(X)/8 = 0$ by the assumption that $\sigma(X) = 0$. So $d = b_+(X)^{-\sigma} = b_1(X)^{-\sigma} = 0$. Hence $deg_R(X,\mathfrak{s})$ is defined and is odd, by Proposition \ref{prop:deginteger} (3).

(3). By assumption $X$ has a spin$^c$-structure $\mathfrak{s}$ with $\sigma^*(\mathfrak{s}) = -\mathfrak{s}$. Hence $\mathfrak{s}$ admits a Real structure by Proposition \ref{prop:realspinc}. Set $d = (c(\mathfrak{s})^2 - \sigma(X))/8$. Then by assumption $2d - b_+(X) + b_1(X) - 1 = 0$ and $d - b_+(X)^{-\sigma} = 0$, so the moduli spaces of the ordinary and Real Seiberg--Witten equations for $\mathfrak{s}$ are zero-dimensional. If $b_+(X)^{-\sigma} = 0$, then $d = b_+(X)^{-\sigma} = 0$ and $deg_R(X,\mathfrak{s})$ is defined and odd, by Proposition \ref{prop:deginteger} (3). If $b_+(X)^{-\sigma} > 0$, then $b_+(X) \ge b_+(X)^{-\sigma} \ge 2$, because $b_+(X)^{-\sigma} = d$ and the assumptions $b_+(X) - b_1(X) = 3 \; ({\rm mod} \; 4)$ and $2d = b_+(X)-b_1(X)+1$ imply that $d$ is even. Theorem \ref{thm:loc2} gives
\[
SW(X,\mathfrak{s}) = \sum_{\mathfrak{s}'} SW_R(X,\mathfrak{s}') \; ({\rm mod} \; 2)
\]
where the sum is over all Real structures on $\mathfrak{s}$. Hence there is at least one Real structure $\mathfrak{s}'$ on $\mathfrak{s}$ for which $SW_R(X,\mathfrak{s}')$ is odd. Since $d$ is even, $deg_R(X,\mathfrak{s}')$ and $SW_{R,\mathbb{Z}}(X,\mathfrak{s}')$ are defined. Since $SW_{R,\mathbb{Z}}(X,\mathfrak{s}') = SW_R(X,\mathfrak{s}') \; ({\rm mod} \; 2)$, we have that $SW_{R,\mathbb{Z}}(X,\mathfrak{s}')$ is non-zero and hence $deg_R(X,\mathfrak{s}') = 2 SW_{R,\mathbb{Z}}(X,\mathfrak{s}')$ is also non-zero.

(4). Let $\mathfrak{s}$ be a spin$^c$-structure with $c(\mathfrak{s})^2 = -1$. Then clearly $\mathfrak{s}$ admits a Real structure and $d = b_+(X)^{-\sigma} = b_1(X)^{-\sigma} = 0$. Hence $deg_R(X,\mathfrak{s})$ is defined and is odd by Proposition \ref{prop:deginteger} (3).

(5). This is immediate from Proposition \ref{prop:XX}.
\end{proof}

\begin{example}
We give some examples of admissible pairs $(X,\sigma)$.
\begin{itemize}
\item[(1)]{A symplectic manifold $(X,\omega)$ with $b_1(X) = 0$, $b_+(X) = 3 \; ({\rm mod} \; 4)$, $\sigma$ a non-free involution such that $\sigma^*(\omega) = -\omega$.}
\item[(2)]{An integer homology $4$-sphere with any orientation preserving, odd involution (this satisfies condition (2)). Note that any orientation preserving involution on such a $4$-manifold has non-empty fixed point set.}
\item[(3)]{An integer homology $S^1 \times S^3$-manifold with an odd, non-free involution $\sigma$ acting trivially on $H^1(X;\mathbb{R})$ and preserving the isomorphism class of a spin structure.}
\item[(4)]{A $4$-manifold $X$ homeomorphic to $K3$ and $\sigma$ an odd involution which acts non-freely.}
\end{itemize}

\end{example}

Let $S \subset S^4$ be an embedded sphere in $S^4$ and let $\Sigma_2( S^4 , S )$ denote the double cover of $S^4$ branched over $S$ and let $\sigma : \Sigma_2(S^4 , S) \to \Sigma_2(S^4 , S)$ denote the covering involution. Given a knot $K$, one constructs a $2$-knot $\tau_{m,n}(K) \subset S^4$ called the $m$-twist, $n$-roll spin of $K$ \cite{lit}. Let $K$ be the $(-2,3,7)$-pretzel knot. Miyazawa proves that $M_n = \Sigma_2( S^4 , \# n \tau_{0,1}(K))$ is homeomorphic to $S^4$ and that $|deg_R( M_n , \mathfrak{s} , \sigma_n)| = 3^n$ \cite{mi}. Here $\mathfrak{s}$ denotes the unique Real spin$^c$-structure on $M_n$ and $\sigma_n$ denotes the covering involution. Similarly, if $P_+ \subset S^4$ denotes the standard embedding of $\mathbb{RP}^2$ in $S^4$ with self-intersection $2$, then $M'_n = \Sigma_2(S^4 , P_+ \# n \tau_{0,1}(K))$ is homeomorphic to $\overline{\mathbb{CP}}^2$ and $|deg_R( M'_n , \mathfrak{s} )| = 3^n$, where $\mathfrak{s}$ is a spin$^c$-structure with $c(\mathfrak{s})^2 = -1$. It was subsequently shown by Hughes--Kim--Miller that $M'_n$ is diffeomorphic to $\overline{\mathbb{CP}}^2$ for any $n \ge 0$ \cite{hkm}. In particular, this means that $\overline{\mathbb{CP}}^2$ has an infinite collection of involutions $\{ \sigma'_n \}_{n \ge 0}$, each with fixed point set an embedded $\mathbb{RP}^2$ and with $|deg_R( \overline{\mathbb{CP}}^2 , \mathfrak{s} , \sigma'_n)| = 3^n$, where $c(\mathfrak{s})^2 = -1$. Note that $(\overline{\mathbb{CP}}^2 , \sigma'_n) = (M_n , \sigma_n) \# (\overline{\mathbb{CP}}^2 , c)$ is the blow-up of $(M_n , \sigma_n)$. 

In \cite{mi} it is proven that the surfaces $P_+ \# n \tau_{0,1}(K)$ are all topologically isotopic (through locally flat topological embeddings) using a result of Conway--Orson--Powell \cite{cop}. It follows that all of the involutions $\sigma'_n$ are topologically equivalent, that is, there exists homeomorphisms $f_n : \overline{\mathbb{CP}}^2 \to \overline{\mathbb{CP}}^2$ such that $\sigma'_n = f_n \circ c \circ f_n^{-1}$, where $c$ denotes complex conjugation. So the involutions $\{\sigma'_n\}$ are exotic copies of complex conjugation.

\begin{theorem}\label{thm:exotic}
Let $(X_1, \sigma_1), \dots , (X_k , \sigma_k)$ be admissible pairs. Let $c : \mathbb{CP}^2 \to \mathbb{CP}^2$ denote complex conjugation. Let $X = X_1 \# \cdots \# X_k$ and $\sigma = \sigma_1 \# \cdots \# \sigma_k$. Then 
\begin{itemize}
\item[(1)]{The involution $\sigma \# c$ on $X \# \overline{\mathbb{CP}}^2$ admits infinitely many exotic copies.}
\item[(2)]{Suppose the fixed point sets of $\sigma_1, \dots , \sigma_k$ are connected. Let $S \subset X/\sigma$ denote the image of the fixed point set of $\sigma$ in $X_0 = X/\sigma$. Then the embedding $P_+ \# S \subset X_0$ admits infinitely many exotic copies.}
\end{itemize}
\end{theorem}
\begin{proof}
(1) Since $(X_1,\sigma_1), \dots , (X_k , \sigma_k)$ are admissible, there exists Real spin$^c$-structures $\mathfrak{s}_1, \dots , \mathfrak{s}_k$ such that $deg_R(X_i , \mathfrak{s}_i)$ is defined an non-zero for each $i$ (Proposition \ref{prop:nzdeg}). Hence there is a Real spin$^c$-structure $\mathfrak{s} = \mathfrak{s}_1 \# \cdots \# \mathfrak{s}_k$ on $X$ such that $deg_R(X , \mathfrak{s})$ is defined and non-zero (Theorem \ref{thm:glue1}). Let $\mathfrak{s}_c$ denote a Real spin$^c$-structure on $(\overline{\mathbb{CP}}^2,c)$ with $c(\mathfrak{s}_c)^2 = -1$. Then $\mathfrak{s} \# \mathfrak{s}_c$ is a Real spin$^c$-structure on $(X' , \sigma') = (X \# \overline{\mathbb{CP}}^2 , \sigma \# c )$ with non-zero degree. Let $\mathcal{S}$ denote the set of Real spin$^c$-structures $\mathfrak{s}'$ on $X' = X \# \overline{\mathbb{CP}}^2$ such that $deg_R( X' , \mathfrak{s}')$ is defined and non-zero (the condition for the degree to be defined is that $(c(\mathfrak{s}')^2 - \sigma(X'))/8 = b_+(X')^{-\sigma'}$). By the compactness properties of the Seiberg--Witten equations, the moduli space of solutions to the Seiberg--Witten equations for fixed metric and perturbation is non-empty for only finitely many spin$^c$-structures. Hence $\mathcal{S}$ is a finite set. Choose an integer $r > 0$ for which $3^r > |deg_R(X' , \mathfrak{s}')|$ for all $\mathfrak{s}' \in \mathcal{S}$. Now consider the infinite collection of involutions $\{ \sigma_{X',n} \}_{n \ge 0}$ on $X'$ defined by $(X' , \sigma_{X',n}) = ( X' , \sigma' ) \# (M_{rn} , \sigma_{rn})$. Note that $X' \# M_{rn} = X \# \overline{\mathbb{CP}}^2 \# M_{rn} \cong X \# \overline{\mathbb{CP}}^2 = X'$. We claim that the involutions $\{ \sigma_{X',n} \}_{n \ge 0}$ are pairwise non-isomorphic.

Let $\mathcal{S}_n$ denote the set of Real spin$^c$-structures on $(X' , \sigma_{X',n})$ for which the degree is defined an non-zero. Let $\mathfrak{t}_n$ denote the unique Real spin$^c$-structure on $M_{n}$. Then it is easily seen that the map $\mathcal{S} \to \mathcal{S}_n$ given by $\mathfrak{s}' \mapsto \mathfrak{s}' \# \mathfrak{t}_{rn}$ is a bijection.

For each $n \ge 0$, define $A_n = \{ |deg_R( X' , \mathfrak{s}' , \sigma_{X',n})| \; | \; \mathfrak{s}' \in \mathcal{S}_n \} \subset \mathbb{Z}$. Note that $A_n$ is finite and non-empty for each $n$. If there is an orientation preserving diffeomorphism $f : X'_{n_1} \to X'_{n_2}$ such that $f \circ \sigma_{X',n_1} = \sigma_{X',n_2} \circ f$, then it follows that $f$ induces a bijection $f : \mathcal{S}_{n_1} \to \mathcal{S}_{n_2}$ which respects degree and hence $A_{n_1} = A_{n_2}$. On the other hand, we have that
\[
|deg_R( X' , \mathfrak{s}' \# \mathfrak{t}_{rn} , \sigma_{X',rn}) | = |deg_R(X' , \mathfrak{s}')| |deg_R(M_{rn} , \mathfrak{t}_{rn})| = 3^{rn} |deg_R( X' , \mathfrak{s}')|.
\]
Hence $A_n = 3^{rn}A_0$ for each $n \ge 0$. If $A_{n_1} = A_{n_2}$, then $3^{rn_1} A_0 = 3^{rn_2} A_0$. But since $3^r > a$ for any $a \in A_0$, it follows that this equality can only occur if $n_1 = n_2$, thus the involutions $\{ \sigma_{X',n} \}_{n \ge 0}$ are pairwise non-isomorphic. Lastly, since each involution $\sigma_{rn}$ is homeomorphic to $c$, it follows that each $\sigma_{X',n}$ is homeomorphic to $\sigma' = \sigma \# c$, so they are exotic copies of $\sigma \# c$.

(2) Observe that $P_+ \# S \subset X'/\sigma' = X_0$ is the image of the fixed point set of $\sigma'$ and that $P_+ \# rn \tau_{0,1}(K) \# S \subset X_0$ is the image of the fixed point set of $\sigma_{X',n}$ (where $K$ is the $(-2,3,7)$-pretzel knot). Since the surfaces $P_+ \# rn \tau_{0,1}(K) \subset S^4$ are all topologically isotopic, the same is true of the surfaces $P_+ \# rn \tau_{0,1}(K) \# S$. However, no two can be smoothly isotopic, for then their branched double covers would be equivariantly diffeomorphic.
\end{proof}

\begin{corollary}
For all $a\neq b$, $a \mathbb{CP}^2 \# b \overline{\mathbb{CP}}^2$ admits an involution which has infinitely many exotic copies.
\end{corollary}
\begin{proof}
Reversing orientation if necessary, we can assume $b > a$. Then $X = a \mathbb{CP}^2 \# b \overline{\mathbb{CP}}^2 = a(S^2 \times S^2) \# (b-a) \overline{\mathbb{CP}}^2$. Now equip $S^2 \times S^2$ with the involution $\sigma(x,y) = (y,x)$ and equip $\overline{\mathbb{CP}}^2$ with the involution $c$ given by complex conjugation. Then $(S^2 \times S^2 , \sigma)$ and $(\overline{\mathbb{CP}}^2 , c)$ are admissible, so the result follows from Theorem \ref{thm:exotic}.
\end{proof}

\begin{corollary}
For all $a \ge 0$, $b>0$, $aK3 \# b \overline{\mathbb{CP}}^2$ has an involution with infinitely many exotic copies.
\end{corollary}
\begin{proof}
Let $\sigma$ be any odd involution on $K3$ which does not act freely. Then $(K3 , \sigma)$ is admissible. Now apply Theorem \ref{thm:exotic} to $(K3 , \sigma)$ and $(\overline{\mathbb{CP}}^2 , c)$.
\end{proof}


\bibliographystyle{amsplain}

\begin{thebibliography}{99}
\bibitem{bar}D. Baraglia, Constraints on families of smooth 4-manifolds from Bauer-Furuta invariants. {\em Algebr. Geom. Topol.} {\bf 21} (2021) 317-349.
\bibitem{bar1}D. Baraglia, The mod 2 Seiberg--Witten invariants of spin structures and spin families. arXiv:2303.06883 (2023).
\bibitem{bar3}D. Baraglia, An adjunction inequality obstruction to isotopy of embedded surfaces in 4-manifolds. {\em Math. Res. Lett.} {\bf 31} (2024), no. 2, 329-352.
\bibitem{bar2}D. Baraglia, Equivariant Seiberg--Witten theory. arXiv:2406.00642 (2024).
\bibitem{bh}D. Baraglia, P. Hekmati, New invariants of involutions from Seiberg--Witten Floer theory. arXiv:2403.00203 (2024).
\bibitem{bk}D. Baraglia, H. Konno, On the Bauer-Furuta and Seiberg-Witten invariants of families of 4-manifolds. {\em J. Topol.} {\bf 15} (2022), no. 2, 505-586. 
\bibitem{bf}S. Bauer, M. Furuta, A stable cohomotopy refinement of Seiberg--Witten invariants. I. {\em Invent. Math.} {\bf 155} (2004), no. 1, 1-19.
\bibitem{b2}S. Bauer, A stable cohomotopy refinement of Seiberg--Witten invariants. II. {\em Invent. Math.} {\bf 155} (2004), no. 1, 21-40.
\bibitem{cop}A. Conway, P. Orson, M. Powell, Unknotting nonorientable surfaces. To appear in {\em J. Eur. Math. Soc.} arXiv:2312.03617 (2023).
\bibitem{die}T. tom Dieck, Transformation groups. De Gruyter Studies in Mathematics, 8. Walter de Gruyter \& Co., Berlin, (1987).
\bibitem{fin}S. Finashin, Knotting of algebraic curves in complex surfaces. {\em Turkish J. Math.} {\bf 25} (2001), no. 1, 147-158.
\bibitem{fs}R. Fintushel, R. J. Stern, Surfaces in 4-manifolds. {\em Math. Res. Lett.} {\bf 4} (1997), no. 6, 907-914.
\bibitem{fss}R. Fintushel, R. J. Stern, N. Sunukjian, Exotic group actions on simply connected smooth 4-manifolds. {\em J. Topol.} {\bf 2} (2009), no. 4, 769-778.
\bibitem{gs}R. E. Gompf, A. I. Stipsicz, {\em $4$-manifolds and Kirby calculus}. Graduate Studies in Mathematics, {\bf 20}. American Mathematical Society, Providence, RI, (1999). xvi+558 pp.
\bibitem{hkm}M. Hughes, S. Kim, M. Miller, Branched covers of twist-roll spun knots. arXiv:2402.11706 (2024).
\bibitem{kato}Y. Kato, Nonsmoothable actions of $\mathbb{Z}_2 \times \mathbb{Z}_2$ on spin four-manifolds. {\em Topology Appl.} {\bf 307} (2022), Paper No. 107868, 13 pp. 
\bibitem{kim}H. J. Kim, Modifying surfaces in 4-manifolds by twist spinning. {\em Geom. Topol.} {\bf 10} (2006), 27-56.
\bibitem{kr}H. J. Kim, D. Ruberman, Topological triviality of smoothly knotted surfaces in 4-manifolds. {\em Trans. Amer. Math. Soc.} {\bf 360} (2008), no. 11, 5869-5881.
\bibitem{kmt}H. Konno, J. Miyazawa, M. Taniguchi, Involutions, links, and Floer cohomologies. {\em J. Topol.} {\bf 17} (2024), no. 2, Paper No. e12340, 47 pp.
\bibitem{lit}R. A. Litherland, Deforming twist-spun knots. {\em Trans. Amer. Math. Soc.} {\bf 250} (1979), 311-331. 
\bibitem{mark}T. E. Mark, Knotted surfaces in 4-manifolds. {\em Forum Math.} {\bf 25} (2013), no. 3, 597-637.
\bibitem{mi2}J. Miyazawa, Localization of a $KO^*(pt)$-valued index and the orientability of the $Pin^-(2)$ monopole moduli space. {\em Algebr. Geom. Topol.} {\bf 25} (2025) 887-918.
\bibitem{mi}J. Miyazawa, A gauge theoretic invariant of embedded surfaces in 4-manifolds and exotic $P^2$-knots. arXiv:2312.02041 (2023).
\bibitem{na}N. Nakamura, $Pin^-(2)$-monopole equations and intersection forms with local coefficients of four-manifolds. {\em Math. Ann.} {\bf 357} (2013), no. 3, 915-939. 
\bibitem{na2}N. Nakamura, $Pin^-(2)$-monopole invariants. {\em J. Differential Geom.} {\bf 101} (2015), no. 3, 507-549.
\bibitem{sti}A. Stieglitz, Equivariant sheaf cohomology. {\em Manuscripta Math.} {\bf 26} (1978/79), no. 1-2, 201-221. 
\bibitem{sung}C. Sung, Some exotic actions of finite groups on smooth 4-manifolds. {\em Osaka J. Math.} {\bf 53} (2016), no. 4, 1055-1061.
\bibitem{tw}G. Tian, S. Wang, Orientability and real Seiberg-Witten invariants. {\em Internat. J. Math.} {\bf 20} (2009), no. 5, 573-604. 
\bibitem{ue}M. Ue, Exotic group actions in dimension four and Seiberg-Witten theory. {\em Proc. Japan Acad. Ser. A Math. Sci.} {\bf 74} (1998), no. 4, 68-70.
\end{thebibliography}

\end{document}